\documentclass[aos,preprint]{imsart}

\RequirePackage[OT1]{fontenc}
\RequirePackage{amsthm,amsmath}
\RequirePackage[numbers]{natbib}
\RequirePackage[colorlinks,citecolor=blue,urlcolor=blue]{hyperref}
\usepackage{amsmath, amsthm, amsfonts,amssymb}
\usepackage{graphicx, color}
\usepackage{subfigure}
\usepackage{appendix}
\usepackage{epstopdf}
\usepackage{float}
\usepackage{hyperref}
\usepackage{tikz}
\usetikzlibrary{shapes,snakes}
\usepackage{stmaryrd}
\usepackage{bm}

\renewcommand{\Im}{{\rm{Im}}}
\renewcommand{\Re}{{\rm{Re}}}

\DeclareMathOperator{\dd}{\rm{d}}
\DeclareMathOperator{\ii}{\rm{i}}
\startlocaldefs
\numberwithin{equation}{section}
\theoremstyle{plain}
\newtheorem{thm}{Theorem}[section]
\newtheorem{cor}[thm]{Corollary}
\theoremstyle{remark}

\newtheorem{rem}[thm]{Remark}
\newtheorem{assm}[thm]{Assumption}
\newtheorem{defn}[thm]{Definition}
\newtheorem{lem}[thm]{Lemma}

\endlocaldefs

\begin{document}

\begin{frontmatter}
%\title{Almost optimal bounds for the convergence rates of eigenvector empirical spectral distribution of sample covariance matrices}
\title{Convergence of eigenvector empirical spectral distribution of sample covariance matrices}

\runtitle{}

\begin{aug}

\author{\fnms{Haokai} \snm{Xi} \thanksref{m1}\ead[label=e1]{haokai@math.wisc.edu}},
\author{\fnms{Fan} \snm{Yang} \thanksref{m1,m2}\ead[label=e2]{fyang75@math.ucla.edu}}
\and
\author{\fnms{Jun}  \snm{Yin} \thanksref{m1,m2,T1} \ead[label=e3]{jyin@math.ucla.edu}}

\thankstext{T1}{Supported by NSF Career Grant DMS-1552192 and Sloan fellowship.}

% \thankstext{t1}{Some comment}
%\thankstext{t2}{First supporter of the project}
% \thankstext{t3}{Second supporter of the project}
\runauthor{H. Xi, F. Yang and J. Yin}

\affiliation{University of Wisconsin-Madison\thanksmark{m1} and University of California, Los Angeles\thanksmark{m2}}

%\address{Department of Statistics \\
%University of Toronto \\
%100 St. George St. \\
%Toronto, Ontario M5S 3G3 \\
%Canada\\
%\printead{e1}\\
%\phantom{E-mail:\ }
%\printead*{}}

\address{Department of Mathematics\\ 
University of California, Los Angeles, \\
Los Angeles, CA 90095, \\
USA\\
\printead{e2}\\
\phantom{E-mail:\ }\printead*{e3}
}

%{\color{red}
\address{Department of Mathematics \\
University of Wisconsin-Madison\\
%480 Lincoln St. \\
Madison, WI 53706 \\
USA\\
\printead{e1}\\
%\phantom{E-mail:\ }\printead*{e2}\\
%\phantom{E-mail:\ }\printead*{e3}
%}
}
\end{aug}

%\begin{abstract}
%The {\it{eigenvector empirical spectral distribution}} (VESD) is a useful tool in studying the limiting behavior of eigenvalues and eigenvectors of covariance matrices. In this paper, we obtain almost optimal bounds on the convergence rates of the VESD of sample covariance matrices to the Mar{\v c}enko-Pastur (MP) distribution, which improves the previous result obtained in \cite{XYZ2013}. Consider the sample covariance matrices of the form $XX^*$, where $X$ is an $M\times N$ random matrix, whose entries are independent random variables with mean zero and variance $N^{-1}$. We show that the Kolmogorov distance between the {\it{expected VESD}} and the MP distribution is of order at most $N^{-1+\epsilon}$ for any constant $\epsilon>0$, provided that the data dimension $M$ to sample size $N$ ratio converges to some $d\ne 1$ (in contrast to the rate $O(N^{-1/2})$ under the condition $d<1$ in \cite{XYZ2013}). This result is proved under the finite 6th moment condition of the $X$ entries, which also relaxes the finite 10th moment assumption in \cite{XYZ2013}. Moreover, we also prove that the convergence rate of the VESD is $O(N^{-1/2+\epsilon})$ almost surely for any constant $\epsilon>0$, 
%%(in contrast to $O(N^{-1/4})$ in \cite{XYZ2013}), 
%requiring only finite $(4+c)$-moment of the underlying distribution for some $c>0$ (in contrast to $O(N^{-1/4+\epsilon})$ under the finite 8th moment assumption in \cite{XYZ2013}).
%%(in contrast to the finite 8th moment assumption in \cite{XYZ2013}).
%\end{abstract}

\begin{abstract}
The {\it{eigenvector empirical spectral distribution}} (VESD) is a useful tool in studying the limiting behavior of eigenvalues and eigenvectors of covariance matrices. In this paper, we study the convergence rate of the VESD of sample covariance matrices to the deformed Mar{\v c}enko-Pastur (MP) distribution. Consider sample covariance matrices of the form $\Sigma^{1/2} X X^* \Sigma^{1/2}$, where $X=(x_{ij})$ is an $M\times N$ random matrix whose entries are independent  random variables with mean zero and variance $N^{-1}$, and  $\Sigma$ is a deterministic positive-definite matrix. We prove that the Kolmogorov distance between the {\it{expected VESD}} and the deformed MP distribution is bounded by $N^{-1+\epsilon}$ for any fixed $\epsilon>0$, provided that the entries $\sqrt{N}x_{ij}$ have uniformly bounded 6th moments and $|N/M-1|\ge \tau$ for some constant $\tau>0$. This result improves the previous one obtained in \cite{XYZ2013}, which gave the convergence rate $O(N^{-1/2})$ assuming $i.i.d.$ $X$ entries, bounded 10th moment, $\Sigma=I$ and $M<N$. Moreover, we also prove that under the finite $8$th moment assumption, the convergence rate of the VESD is $O(N^{-1/2+\epsilon})$ almost surely for any fixed $\epsilon>0$, 
 which improves the previous bound $N^{-1/4+\epsilon}$ in \cite{XYZ2013}.  
\end{abstract}



\end{frontmatter}

\section{Introduction and main results}\label{Sec_intro}

Sample covariance matrices are fundamental objects in multivariate statistics. The population covariance matrix of a centered random vector $\mathbf y\in \mathbb R^M$ is $\Sigma=\mathbb E \mathbf y \mathbf y^*$. Given $N$ independent samples $(\mathbf y_1, \cdots, \mathbf y_N)$ of $\mathbf y$, the sample covariance matrix $Q := N^{-1}\sum_i \mathbf y_i \mathbf y_i^*$ is the simplest estimator for $\Sigma$. In fact, if $M$ is fixed, then $Q$ converges almost surely to $\Sigma$ as $N\to \infty$. However, in many modern applications, the advance of technology has led to high dimensional data where $M$ is comparable to or even larger than $N$. In this setting, $\Sigma$ cannot be estimated through $Q$ directly, but some properties of $\Sigma$ can be inferred from the eigenvalue and eigenvector statistics of $Q$. The large dimensional covariance matrices have more and more applications in various fields, such as statistics \cite{DT2011,IJ,IJ2,IJ2008}, economics \cite{Economics} and population genetics \cite{Genetics}. 

In this paper, we consider sample covariance matrices of the form $Q_1:=\Sigma^{1/2} X X^* \Sigma^{1/2}$, where $X=(x_{ij})$ is an $M\times N$ real or complex data matrix whose entries are independent (but not necessarily identically distributed) random variables satisfying
\begin{align}
& \mathbb{E} x_{ij} =0, \ \ \mathbb{E} | x_{ij} |^2  = N^{-1}, \ \ 1\le i \le M, \ 1\le j \le N, \label{entry_assm}
\end{align}
and the population covariance matrix $ \Sigma:=\text{diag}(\sigma_1,\sigma_2,\ldots, \sigma_M)$, {$\sigma_1\ge \cdots \ge \sigma_M \ge 0$,} is a deterministic positive-definite matrix. If the entries of $X$ are complex, then we assume in addition that 
\begin{equation}
\mathbb Ex_{ij}^2 = 0, \ \ 1\le i \le M, \ 1\le j \le N. \label{entry_assm2}
\end{equation}
Define the aspect ratio $d_N:={N}/{M}.$ We are interested in the high dimensional case with $\lim_{N \rightarrow \infty} d_N = d \in (0,\infty)$. %, i.e. the case with high dimensional data.
%case where the data dimension is comparable to the sample size. 
We will also consider the $N \times N$ matrix $Q_2:=X^* \Sigma X$, which share the same nonzero eigenvalues with $Q_1$.

A simple but important example is the sample covariance matrix with $\Sigma=\sigma^2 I$ (i.e. the null case). In applications of spectral analysis of large dimensional random matrices, one important problem is the convergence rate of the empirical spectral distributions (ESD). It is well-known that the ESD $F^{(M)}_{XX^*}$ of $XX^*$ converges weakly to the Mar{\v c}enko-Pastur (MP) law $F_{MP}$ \cite{MP}.
%of $Q_1$ and $Q_2$, which are denoted by $F^{(M)}_{Q_1}(x)$ and $F^{(N)}_{Q_2}(x)$, respectively. 
%Recall that the Kolmogorov distance between two distributions $F_1$ and $F_2$ is defined as
One way to measure the convergence rate of the ESD is to use the Kolmogorov distance
$$\| F^{(M)}_{XX^*} - F_{MP}\| :=\sup_{x}|F^{(M)}_{XX^*}(x) - F_{MP}(x)|.$$
%Then we use $\|F^{(M)}_{XX^*} - F_{d_N}\|$ 
%to measure the convergence rate of the ESD. 
The convergence rate for sample covariance matrices was first established in \cite{bai1993_2}, % with Berry-Esseen type inequalities for the difference of two distributions in terms of their Stieltjes transforms.
%The convergence rate of the ESD of sample covariance matrices was first studied in \cite{bai1993_2} with Berry-Esseen type inequalities for the difference of two distributions in terms of their Stieltjes transforms. 
%The Berry-Esseen type inequalities were 
and later improved in \cite{GT2004} to $O(N^{-1/2})$ in probability under the finite 8th moment condition. In \cite{PY}, the authors proved an almost optimal bound that $\|F^{(M)}_{XX^*} - F_{MP}\|=O(N^{-1+\epsilon})$ with high probability for any fixed $\epsilon>0$ under the sub-exponential decay assumption.

The research on the asymptotic properties of eigenvectors of large dimensional random matrices is generally harder and much less developed. However, the eigenvectors play an important role in high dimensional statistics. In particular, the principal component analysis (PCA) is now favorably recognized as a powerful technique for dimensionality reduction, and the eigenvectors corresponding to the largest eigenvalues are the directions of the principal components. 
%The earlier work on the properties of eigenvectors goes back to Anderson \cite{Anderson}, where the author obtained the asymptotic distribution for the eigenvectors of the Wishart matrix when $M$ is fixed and $N\to\infty$. 
The earlier work on the properties of eigenvectors goes back to Anderson \cite{Anderson}, where the author proved that the eigenvectors of the Wishart matrix are asymptotically normal and isotropic when $M$ is fixed and $N\to\infty$. For the high dimensional case, Johnstone \cite{IJ2} proposed the spiked model to test the existence of principal components. Then Paul \cite{Paul2007} studied the directions of eigenvectors corresponding to spiked eigenvalues. In \cite{Ma2013}, Ma proposed an iterative thresholding approach to estimate sparse principal subspaces in the setting of a high-dimensional spiked covariance model. Using a reduction scheme which reduces the sparse PCA problem to a high-dimensional multivariate regression problem, \cite{Cai2013} established the optimal rates of convergence for estimating the principal subspace for a large class of spiked covariance matrices. One can see the references in \cite{Cai2013,Ma2013} for more literatures on sparse PCA and spiked covariance matrices. 

For the test of the existence of spiked eigenvalues, we first need to study the properties of the eigenmatrices in the null case. If $\Sigma=\sigma^2 I$, then the eigenmatrix is expected to be asymptotically Haar distributed (i.e. uniformly distributed over the unitary group). However, formulating the terminology ``asymptotically Haar distributed" is far from trivial since the dimension $M$ is increasing. Following the approach in \cite{JS1989,JS1990,Bai2007,XZ2016,XYZ2013}, we will use the {\it{eigenvector empirical spectral distribution}} (VESD) to characterize the asymptotical Haar property. 
%\cite{JS1989,JS1990,Bai2007,XZ2016,XYZ2013} for the asymptotical Haar property of the eigenmatrix based on the VESD (defined in (\ref{defn_VESD}) below), 
%This paper is concerned with the {\it{eigenvector empirical spectral distribution}} (VESD) of sample covariance matrices, which we shall now define. 
%For later use, we define the VESD for sample covariance matrices for a general $\Sigma$. 
Suppose
\begin{equation}\label{SVD_X}
\Sigma^{1/2}X = \sum\limits_{1 \le k \le {N\wedge M} }{\sqrt {\lambda_k} \xi_k } \zeta_k^*
\end{equation}
is a singular value decomposition of $\Sigma^{1/2}X$, where
$$\lambda_1\ge \lambda_2 \ge \ldots \ge \lambda_{N\wedge M} \ge 0 = \lambda_{N\wedge M+1} = \ldots = \lambda_{N\vee M},$$
$\{\xi_{k}\}_{k=1}^{M}$ are the left-singular vectors, and $\{\zeta_k\}_{k=1}^{N}$ are the right-singular vectors.
%two orthonormal bases of $\mathbb R^{M}$ and $\mathbb R^{N}$, respectively. 
%Suppose $XX^*$ has the spectral decomposition
%\begin{equation}\label{SVD_X}
%X X^*= \sum\limits_{k = 1}^{M} \lambda_k(XX^*) \xi_k \xi _{k}^*,
%\end{equation}
%where $\xi_{k}$ are the eigenvectors. 
Then for deterministic unit vectors $\mathbf u\in\mathbb C^M$ and $\mathbf v\in\mathbb C^N$, we define the VESD of $Q_{1,2}$ as
\begin{equation}\label{defn_VESD}
F^{(M)}_{Q_1,\mathbf u}(x) = \sum_{k=1}^M |\langle \xi_k,\mathbf u\rangle|^2 \mathbf{1}_{\{\lambda_k \leq x\}}, \ \ F^{(N)}_{Q_2,\mathbf v}(x) = \sum_{k=1}^N |\langle \zeta_k,\mathbf v\rangle|^2 \mathbf{1}_{\{\lambda_k \leq x\}}.
\end{equation}
%where $\mathbf u$ and $\mathbf v$ are deterministic unit vectors in $\mathbb C^M$ and $\mathbb C^N$, respectively. 

Now we apply the above formulations to the null case. Adopting the ideas of \cite{JS1989,JS1990}, we define the stochastic process as
$$X_{M,\mathbf u}(t):=\sqrt{\frac{M}{2}}\sum_{k=1}^{\lfloor Mt \rfloor}\left( |\langle \xi_k,\mathbf u\rangle|^2 - M^{-1}\right).$$
If the eigenmatrix of $XX^*$ is Haar distributed, then the vector $\mathbf y:=(\langle \xi_k,\mathbf u\rangle)_{k=1}^M$ is uniformly distributed over the unit sphere, and $X_{M,\mathbf u}(t)$ would converge to a Brownian bridge by Donsker's theorem. Thus the convergence of $X_{M,\mathbf u}$ to a Brownian bridge characterizes the asymptotical Haar property of the eigenmatrix. For convenience, we can consider the time transformation 
$$X_{M,\mathbf u}(F^{(M)}_{XX^*}(x))= \sqrt{\frac{M}{2}}\left(F^{(M)}_{XX^*,\mathbf u}(x) - F^{(M)}_{XX^*}(x)\right).$$
Thus the problem is reduced to the study of the difference between the VESD and the ESD. It was already proved in \cite{Bai2007,isotropic} that $F^{(M)}_{XX^*,\mathbf u}$ also converges weakly to the 
MP law for any sequence of unit vectors $\mathbf u \in \mathbb R^M$. On the other hand, compared with ESD, much less has been known about the convergence rate of the VESD. The best result so far was obtained in \cite{XYZ2013}, where the authors proved that if $d_N>1$ and the entries of $X$ are $i.i.d.$ centered random variables, then $\|\mathbb E F^{(M)}_{XX^*,\mathbf u}-F_{MP}\|=O(N^{-1/2})$ under the finite 10th moment assumption, and $\|F^{(M)}_{XX^*,\mathbf u}-F_{MP}\|=O(N^{-1/4+\epsilon})$ almost surely under the finite 8th moment assumption. 
%(In fact, if $\Sigma = I$, then $F^{(M)}_{1c,\mathbf u}$ is independent of the choice of $\mathbf u$.) 
However, we find that both of these bounds are far away from being optimal, and can be improved with a different method. This is one of the purposes of this paper. 

We will also extend the above formulation to include sample covariance matrices with general population $\Sigma$. For a non-scalar $\Sigma$, the eigenmatrix of $Q_1$ is not asymptotically Haar distributed anymore. For its distribution, we conjecture that the eigenvectors of $Q_1$ are asymptotically independent, and each $\xi_k$ is asymptotically normal with covariance matrix given by some $\mathbf D_k$. In fact, our results in this paper suggest that $\mathbf D_k$ takes the form $\mathbf F_{1c}(\gamma_k)-\mathbf F_{1c}(\gamma_{k+1})$, where $\gamma_k$ is defined in (\ref{gammaj}) to denote the classical location for $\lambda_k$, and $\mathbf F_{1c}$ is a matrix-valued function defined in (\ref{F1cv}) with the property that $\langle \mathbf u,\mathbf F_{1c}\mathbf u\rangle$ is the asymptotic distribution of the VESD $F_{Q_1,\mathbf u}$ for any $\mathbf u \in \mathbb C^M$. Again, since the dimension $M$ increases to infinity, the above property is hard to formulate. One way is to consider the finite-dimensional restriction in the following sense:
% We conjecture that the eigenvectors of $Q_1$ are asymptotically normal in the following sense. 
given $m\in \mathbb N$, for any fixed unit vector $\mathbf u \in \mathbb C^M$ and $\{i_1,\cdots,i_m\}\subseteq \{1,\cdots, N\wedge M\}$, we should have asymptotically
\begin{equation}\label{sympQ1}
\left( \langle \xi_{i_1},  \mathbf u\rangle, \cdots, \langle \xi_{i_m}, \mathbf u\rangle \right) \sim \mathcal N_m\left(0, \langle \mathbf u,\mathbf D_{i_1}\mathbf u\rangle , \ldots, \langle \mathbf u,\mathbf D_{i_m}\mathbf u\rangle \right).
\end{equation}
%where $\mathcal N_m$ denotes an $m$-dimensional multivariate normal distribution and 
%Here the variances are $d_{j,\mathbf u}:=F_{1c,\mathbf u}(\gamma_{j})-F_{1c,\mathbf u}(\gamma_{j+1})$, ... define $\mathbf F_{1c}$ ..., $1\le j \le N\wedge M$, where $F_{1c,\mathbf u}$ is the asymptotic distribution of $F_{Q_1,\mathbf u}$ that is defined below (\ref{m1cv}), and $\gamma_j$'s are the classical locations of the eigenvalues that is defined in (\ref{gammaj}). 
(In fact, for a nice choice of $\Sigma$ in the sense of Definition \ref{def_regular}, $\langle \mathbf u,\mathbf D_{k}\mathbf u\rangle$ is typically of order $N^{-1}$.) We can also adopt the approach as above, that is to investigate the stochastic process
\begin{equation}\label{Xsigma}
X^\Sigma_{M,\mathbf u}(t):=\sqrt{\frac{M}{2}}\sum_{k=1}^{\lfloor Mt \rfloor}\left( |\langle \xi_k,\mathbf u\rangle|^2 - \langle \mathbf u,\mathbf D_{k}\mathbf u\rangle\right).
\end{equation}
If $M<N$, we conjecture that $X^\Sigma_{M,\mathbf u}(t)$ converges to the following Gaussian process for $0\le t \le 1$:
\begin{equation}\label{Bsigma}
{\mathbf B^\Sigma_{\mathbf u}(t)} := \int_0^t \langle \mathbf u,\mathbf F_{1c}\mathbf u \rangle \circ F_{1c}^{-1} \dd B_t \quad \text{conditioning on} \quad \mathbf B^\Sigma_{\mathbf u}(1)=0,
\end{equation}
where $B_t$ is a standard Brownian motion, $F_{1c}$ is the asymptotic ESD of $Q_1$ defined in (\ref{F1c}), and $F_{1c}^{-1}$ denotes the quantile function. As before, we can study the process (\ref{Xsigma}) through the time transformaton $X^\Sigma_{M,\mathbf u}(F_{Q_1}(x))$, where $F_{Q_1}$ is the ESD of $Q_1$. Due to the rigidity of eigenvalues (see Theorem \ref{thm_largerigidity}), we have for all $x$,
$$\sqrt{\frac{2}{M}}X^\Sigma_{M,\mathbf u}(F_{Q_1}(x))=F_{Q_1,\mathbf u}(x) - \langle \mathbf u,\mathbf F_{1c}(x)\mathbf u\rangle+O(N^{-1+\epsilon})$$ 
%up to a negligible error of order $O(N^{-1+\epsilon})$ 
with very high proability for any fixed $\epsilon>0$.  
%$$X_{M,\mathbf u}(F^{(M)}_{XX^*}(x))= \sqrt{\frac{M}{2}}\sum_{k=1}^{\lfloor Mt \rfloor}\left(F^{(M)}_{XX^*,\mathbf u}(x) - F^{(M)}_{XX^*}(x)\right).$$
%as a consequence of Theorem \ref{thm_largerigidity}
Thus we need to study the convergence rate of $F_{Q_1,\mathbf u}$ to $\langle \mathbf u,\mathbf F_{1c}\mathbf u\rangle$, and this is our main goal. %Also, we remark that $\mathbf F_{1c}$ depends on $\mathbf u$ in an explicit way as given by (\ref{m1cv}), which may be used to detect the variance structure of $\Sigma$; see Fig. \ref{fig_VESD} for an example. 
In fact, we will prove that the convergence rate of $\mathbb EF_{Q_1,\mathbf u}$ is $O(N^{-1+\epsilon})$ for any fixed $\epsilon>0$, which shows that the limiting process is centered, and the convergence rate of $F_{Q_1,\mathbf u}$ is $O(N^{-1/2+\epsilon})$, which partially verify the $\sqrt{M}$ scaling.

%We remark that great progress has been made in other directions of the research on eigenvector statistics. For example, one can refer to \cite{isotropic,Delocal_Wigner} for the delocalization and isotropic delocalization of eigenvectors, 
%\cite{Wigner_vector,TaoVu_vector} for the universality of eigenvectors, \cite{QUE_vector} for the local quantum unique ergodicity of eigenvectors and \cite{Principal} for the eigenvectors of principal components. %\cite{JS1989,JS1990,Bai2007,XZ2016,XYZ2013} for the asymptotical Haar property of the eigenmatrix based on the VESD (defined in (\ref{defn_VESD}) below), to name a few. %Note that some of these results are proved for Wigner matrices, but their generalizations to sample covariance matrices usually are straightforward.  

\subsection{Main results}
%In this paper, we consider sample covariance matrices of the form $Q_1:=\Sigma^{1/2} X X^* \Sigma^{1/2}$, where $X=(x_{ij})$ is an $M\times N$ real or complex data matrix whose entries are independent (but not necessarily identically distributed) random variables satisfying
%\begin{align}
%& \mathbb{E} x_{ij} =0, \ \ \mathbb{E} | x_{ij} |^2  = N^{-1}, \ \ 1\le i \le M, \ 1\le j \le N, \label{entry_assm}
%\end{align}
%and the population covariance matrix $\mathbb E Q_1 = \Sigma:=\text{diag}(\sigma_1,\sigma_2,\ldots, \sigma_M)$ is a real positive-definite deterministic matrix. If the entries of $X$ are complex, then we assume in addition that 
%\begin{equation}
%\mathbb Ex_{ij}^2 = 0, \ \ 1\le i \le M, \ 1\le j \le N. \label{entry_assm2}
%\end{equation}
%Define the aspect ratio $d_N:={N}/{M}.$ In this paper, we shall be interested in the regime with $\lim_{N \rightarrow \infty} d_N = d \in (0,\infty)$, which corresponds to the high dimensional case where the data dimension is comparable to the sample size. We will also consider the $N \times N$ matrix $Q_2:=X^* \Sigma X$, which share the same nonzero eigenvalues with $Q_1$.
We consider sample covariance matrices with a general diagonal $\Sigma$, %We first need to introduce the {\it{deformed Marchenko-Pastur (MP) law}}. 
whose empirical spectral distribution is denoted by
\begin{equation}\label{sigma_ESD}
\pi\equiv \pi_M := M^{-1} \sum_{1\le i\le M} \delta_{\sigma_i}.
\end{equation}
We assume that there exists a small constant $\tau>0$ such that 
\begin{equation}\label{assm3}
\sigma_1 \le \tau^{-1} \ \ \text{ and } \ \ \pi_M([0,\tau]) \le 1 - \tau \ \ \text{for all } M.
\end{equation}
The first condition means that the operator norm of $\Sigma$ is bounded, and the second condition means that the spectrum of $\Sigma$ cannot concentrate at zero. If $\pi_M$ converges weakly to some distribution $\hat \pi$ as $M\to \infty$, then it was shown in \cite{MP} that the ESD of $Q_{2}$ converges in probability to some deterministic distribution, which is called the {\it{deformed}} Mar{\v c}enko-Pastur law. For any $N$, we describe the deformed MP law $F_{2c}^{(N)}$ through its Stieltjes transform
$$m_{2c}(z)\equiv m^{(N)}_{2c}(z):=\int_{\mathbb R} \frac{\dd F^{(N)}_{2c}(x)}{x-z}, \ \ z = E+ i\eta \in \mathbb C_+.$$
%It is easy to observe that $m_{1,2c}$ are an analytic functions on $\mathbb C_+$ and $\Im\, m_{1,2c} \ge 0$ whenever $\Im\, z>0$. 
We define $m_{2c}$ as the unique solution to the self-consistent equation
\begin{equation}\label{deformed_MP21}
\frac{1}{m_{2c}(z)} = - z + {d_N^{-1}}\int\frac{t}{1+m_{2c}(z) t} \pi(\dd t), 
\end{equation}
subject to the conditions that $\Im \, m_{2c}(z) \ge 0$ and $\Im \, zm_{2c}(z) \ge 0$ for $z\in \mathbb C_+$. It is well known that the functional equation (\ref{deformed_MP21}) has a unique solution that is uniformly bounded on $\mathbb C_+$ under the assumption (\ref{assm3}) \cite{MP}. Letting $\eta \downarrow 0$, we can recover the asymptotic eigenvalue density $\rho_{2c}$ (which further gives $F_{2c}^{(N)}$) with the inverse formula
\begin{equation}\label{ST_inverse}
\rho_{2c}(E) = {\pi}^{-1}\lim_{\eta\downarrow 0} \Im\, m_{2c}(E+\ii\eta).
\end{equation}
Since $Q_1$ share the same nonzero eigenvalues with $Q_2$ and has $M-N$ more (or $N-M$ less) zero eigenvalues, we can obtain the asymptotic ESD for $Q_1$:
\begin{equation}
F^{(M)}_{1c}={d_N} {F}^{(N)}_{2c}+(1-{d_N})\mathbf{1}_{[0,\infty)}. \label{F1c}
\end{equation}
%The measure $\rho^{(N)}_{2c}$ is sometimes called the {\it{multiplicative free convolution}} of $\pi_M$ and the MP law, see e.g. \cite{AGZ,VDN}. 
In the rest of this paper, we will often omit the super-indices $N$ and $M$ from our notations. The properties of $m_{2c}$ and $\rho_{2c}$ have been studied extensively; see e.g. \cite{Bai1998,Bai2006,BPZ,HHN,Anisotropic,Silverstein1995,SC}. The following Lemma \ref{Structure_lem} describes the basic structure of $\rho_{2c}$. For its proof, one can refer to \cite[Appendix A]{Anisotropic}. %It is easy to observe that $m_{2c}(a_k)=b_k$ according to the definition of $f$. %In particular, we shall define the rightmost edge (i.e. the {\it{soft edge}}) of $\rho_{1,2c}$.
\begin{lem}[Support of the deformed MP law]\label{Structure_lem}
The density $\rho_{2c}$ is a disjoint union of connected components:
\begin{equation}\label{support_rho1c}
{\rm{supp}} \, \rho_{2c} \cap (0,\infty) = \bigcup_{ k=1}^L [a_{2k}, a_{2k-1}] \cap (0,\infty),
\end{equation}
where $L\in \mathbb N$ depends only on $\pi_M$. Moreover, $N\int_{a_{2k}}^{a_{2k-1}} \rho_{2c}(x)dx$ is an integer for any $k=1,\ldots, L$, which give the classical number of eigenvalues in the bulk component $[a_{2k},a_{2k-1}]$.
%Here $a_k$ are characterized as following: there exists a real sequence $\{b_k\}_{k=1}^{2L}$ such that $(x,m)=(a_k, b_k)$ are the real solutions to the equations
%\begin{equation}
%x = f(m), \ \ \text{and} \ \ f'(m) = 0. \label{equationEm2}
%\end{equation}
%Moreover, we have $b_1 \in (-\sigma_1^{-1}, 0)$. Finally, under assumptions (\ref{assm2}) and (\ref{assm3}), we have $a_1 \le C$ for some positive constant $C$. 
\end{lem}

We shall call $a_k$ the edges of $\rho_{2c}$. For any $1\le k\le 2L$, we define
\begin{equation}\label{Nk}
N_k := \sum_{l:2l\le k} N\int_{a_{2l}}^{a_{2l-1}} \rho_{2c}(x)\dd x.
\end{equation}
Then we define the classical locations $\gamma_j$ for the eigenvalues of $\mathcal Q_2$ through
\begin{equation}\label{gammaj}
 1 - F_{2c}(\gamma_j) = \frac{j-1/2}{N},  \ \ 1\le j \le K,
\end{equation}
where we abbreviate $K:=M\wedge N$. Note that (\ref{gammaj}) is well-defined since the $N_k$'s are integers. For convenience, we also denote $\gamma_0:=+\infty$ and $\gamma_{K+1}:=0$.

%In particular, we shall denote the leftmost and rightmost edges by $\lambda_-:=a_{2L}$ and $\lambda_+:=a_{1}$, which gives the classical location of the smallest (or largest) nontrivial singular value of $Q_2$
To establish our main result, we need to make some extra assumptions on $\Sigma$ and $\pi_M$, which takes the form of the following regularity conditions.

\begin{defn}[Regularity]\label{def_regular}
%(Regularity) Fix $\tau \le \left||z|^2-1\right| \le \tau^{-1}$.
(i) Fix a (small) constant $\tau>0$. We say that the edge $a_k$, $k=1, \ldots, 2L$, is $\tau$-regular if
\begin{equation}
a_k\ge \tau, \quad \min_{l\ne k} |a_k - a_l| \ge \tau, \quad \min_{i}|1+m_{2c}(a_k)\sigma_i| \ge \tau , \label{regular1}
\end{equation}
where $m_{2c}(a_k):= m_{2c}(a_k + \ii 0_+)$.

(ii) We say that the bulk components $[a_{2k}, a_{2k-1}]$ is regular if for any fixed $\tau'>0$ there exists a constant $c\equiv c_{\tau'}>0$ such that the density of $\rho_{2c}$ in $[a_{2k}+\tau', a_{2k-1}-\tau']$ is bounded from below by $c$.
\end{defn}

\begin{rem} The edge regularity conditions (i) has previously appeared (in slightly different forms) in several works on sample covariance matrices \cite{BPZ1, Karoui,HHN, Anisotropic,LS,Regularity4}. The condition (\ref{regular1}) ensures a regular square-root behavior of $\rho_{2c}$ near $a_k$. %, and in particular rules out outliers. 
The bulk regularity condition (ii) was introduced in \cite{Anisotropic}, and it imposes a lower bound on the density of eigenvalues away from the edges. These conditions are satisfied by quite general classes of $\Sigma$; see e.g. \cite[Examples 2.8 and 2.9]{Anisotropic}.%Without it, one can have points in the interior of $\text{supp}\, \rho_{1c}$ with an arbitrarily small density and our arguments would fail.
\end{rem}

For any $\mathbf u \in \mathbb C^M$ and $z\in \mathbb C_+$, we define
\begin{equation}\label{m1cv}
m_{1c,\mathbf u}(z):=- \langle \mathbf u, z^{-1}(1+m_{2c}(z)\Sigma)^{-1}\mathbf u\rangle . 
\end{equation}
Then $m_{1c,\mathbf u}$ is the Stieltjes transform of a distribution, which we shall denote by $F_{1c,\mathbf u}$. From (\ref{m1cv}), it is easy to see that there exists a matrix-valued function $\mathbf F_{1c}$ depending on $\Sigma$ such that %$\mathbf F_{1c}^{(M)}(x)$ is diagonal for each $x$ and 
$F_{1c,\mathbf u}=\langle \mathbf u,\mathbf F_{1c}\mathbf u\rangle$, i.e., we have%that is we define $\mathbf F_{1c}^{(M)}$ through
\begin{equation}\label{F1cv}
m_{1c,\mathbf u}(z) = \int_{\mathbb R} \frac{\dd F_{1c,\mathbf u}(x)}{x-z}=  \langle \mathbf u, \int_{\mathbb R} \frac{\dd\mathbf F_{1c}(x)}{x-z}\mathbf u\rangle.
\end{equation}
%(Note that if $\Sigma = \sigma^2 I$, then $\mathbf F_{1c}^{(M)}(x)$ is a scalar matrix for each $x$.) 
It was already proved in \cite{Anisotropic} that for any sequence of unit vectors $\mathbf u\in \mathbb C^M$ and $\mathbf v\in \mathbb C^N$, $F^{(M)}_{Q_1,\mathbf u}$ converges weakly to $F_{1c,\mathbf u}$ and $F^{(N)}_{Q_2,\mathbf v}(x)$ converges weakly to $F_{2c}$. Now we are ready to state our main results, i.e. Theorem \ref{main_thm}. %--- on the convergence rate of the VESD. 
We first give the main assumptions.

%{\color{red}On the other hand, compared with ESD, much less has been known about the convergence rate of the VESD of $Q_{1,2}$. %{\color{red}To the best of our knowledge, there is only one paper \cite{XYZ2013} studying on this topic.} In that paper, 
%The best result so far was obtained in \cite{XYZ2013}, where the authors proved that if $d_N>1$, $\Sigma=I$, and the entries of $X$ are $i.i.d.$ centered random variables, then $\|\mathbb E F^{(M)}_{Q_1,\mathbf u}-F^{(M)}_{1c,\mathbf u}\|=O(N^{-1/2})$ under the finite 10th moment assumption, and $\|F^{(M)}_{Q_1,\mathbf u}-F^{(M)}_{1c,\mathbf u}\|=O(N^{-1/4+\epsilon})$ almost surely under the finite 8th moment assumption. However, we find that both of these bounds are far away from being optimal, and can be improved with a different method. This is the main purpose of this paper. We first state our main assumptions.}

\begin{assm}\label{main_assm}
Fix a (small) constant $\tau>0$.

(i) $X=(x_{ij})$ is an $M\times N$ real or complex matrix whose entries are independent random variables that satisfy the following moment conditions: there exist constants $C_0,c_0>0$ such that for all $1\le i \le M$, $1\le j \le N$,
\begin{align}
\left|\mathbb{E} x_{ij}\right|  & \le C_0 N^{-2-c_0},\label{entry_assm0} \\
\left|\mathbb{E} | x_{ij} |^2  - N^{-1}\right|  & \le C_0N^{-2-c_0}, \label{entry_assm1}\\
\left| \mathbb Ex_{ij}^2\right| &\le C_0 N^{-2-c_0}, \ \ \text{if $x_{ij}$ is complex}, \label{entry_assmex}\\
\mathbb{E} | x_{ij} |^4 & \le C_0 N^{-2}. \label{conditionA3} 
\end{align}
Note that (\ref{entry_assm0})-(\ref{entry_assmex}) are slightly more general than (\ref{entry_assm}) and (\ref{entry_assm2}).

(ii) $\tau \le d_N \le \tau^{-1}$ and $|d_N - 1|\ge \tau$.

(iii) $\Sigma=\text{diag}(\sigma_1,\sigma_2,\ldots, \sigma_M)$ is a deterministic positive-definite matrix. We assume that (\ref{assm3}) holds, all the edges of $\rho_{2c}$ are $\tau$-regular, and all the bulk components of $\rho_{2c}$ are regular in the sense of Definition \ref{def_regular}.
\end{assm}

%Our main result is stated as the following theorem. 
%For a reason that will be clear later (when we prove Corollary \ref{main_cor}), we consider slightly more general random matrices $X=(x_{ij})$. More specifically, we define the following conditions for the entries of $X$: there exist constants $C_0,c_0>0$ such that for all $1\le i \le M$ and $1\le j \le N$,
%\begin{align}
%\left|\mathbb{E} x_{ij}\right|  & \le C_0 N^{-2-c_0},\label{entry_assm0} \\
%\left|\mathbb{E} | x_{ij} |^2  - N^{-1}\right|  & \le C_0N^{-2-c_0}, \label{entry_assm1}\\
%\mathbb{E} | x_{ij} |^4 & \le C_0 N^{-2}, \label{conditionA3} \\
%\left| \mathbb Ex_{ij}^2\right| &\le C_0 N^{-2-c_0}, \ \ \text{if $x_{ij}$ is complex.} \label{entry_assmex}
%\end{align}

\begin{thm}\label{main_thm}
%Let $X=(x_{ij})$ be an $M\times N$ random matrix whose entries are independent random variables satisfying (\ref{entry_assm0}), (\ref{entry_assm1}), (\ref{conditionA3}) and (\ref{entry_assmex}). 
Suppose $d_N$, $X$ and $\Sigma$ satisfy the Assumption \ref{main_assm}. Suppose there exist constants $C_1,\phi>0$ such that 
\begin{equation}\label{size_condition}
\max_{1\le i \le M, 1\le j \le N} |x_{ij}| \le C_1 N^{-\phi}.
\end{equation}
%for some fixed $\phi >0$. 
%Suppose $d_N\to d$ for some constant $d\ne 1$. 
Let $\mathbf u \equiv \mathbf u_M \in \mathbb C^M$ and $\mathbf v\equiv \mathbf v_N\in \mathbb C^N$ denote sequences of deterministic unit vectors. Then for any fixed (small) $\epsilon>0$ and (large) $D>0$, we have
\begin{equation}\label{boundE}
\|\mathbb E F^{(M)}_{Q_1,\mathbf u}-F^{(M)}_{1c,\mathbf u}\| + \|\mathbb E F^{(N)}_{Q_2,\mathbf v}-F^{(N)}_{2c}\|\le N^{-1+\epsilon}
\end{equation}
for sufficiently large $N$, and for $\mathfrak a:=\min(2\phi, 1/2)$, 
%for all deterministic unit vectors $\mathbf u\in \mathbb C^M$ and $\mathbf v\in \mathbb C^N$, provided that $N$ is large enough. Moreover, for any sequence of deterministic unit vectors $\mathbf u_M \in \mathbb C^M$ and $\mathbf v_N\in \mathbb C^N$, we have
\begin{equation}\label{boundp}
\mathbb P\left(\|F^{(M)}_{Q_1,\mathbf u}-F^{(M)}_{1c,\mathbf u}\| + \|F^{(N)}_{Q_2,\mathbf v}-F^{(N)}_{2c}\| \ge N^{-\mathfrak a+ \epsilon} \right) \le N^{-D}.
\end{equation}
%for all deterministic unit vectors $\mathbf u\in \mathbb C^M$ and $\mathbf v\in \mathbb C^N$ {\color{red}($\mathbf u,\mathbf v$ need to be a sequence of vectors)}, provided that $N$ is large enough. 
\end{thm}

As an immediate corollary of Theorem \ref{main_thm}, we have the following result.

\begin{cor}\label{main_cor}
Suppose $d_N$ and $\Sigma$ satisfy the Assumption \ref{main_assm}. Let $X=(x_{ij})$ be an $M\times N$ random matrix whose entries are independent and satisfy (\ref{entry_assm}) and (\ref{entry_assm2}). Suppose there exist constants $a,A>0$ such that
%\begin{equation}\label{tail_condition}
%\max_{1\le i \le M, 1\le j \le N} \mathbb E |\sqrt{N}x_{ij}|^a \le A
%\end{equation}
\begin{equation}\label{tail_condition}
\limsup_{s\to \infty} s^a \max_{ i, j }\mathbb P\left( |\sqrt{N}x_{ij}|\ge s\right)\le A
\end{equation}
for all $N$. Let $\mathbf u \equiv \mathbf u_M \in \mathbb C^M$ and $\mathbf v\equiv \mathbf v_N\in \mathbb C^N$ denote sequences of deterministic unit vectors. Then for any fixed $\epsilon>0$, if $a\ge 6$, we have %for any fixed $\epsilon>0$, %and deterministic unit vector $\mathbf u\in \mathbb C^M$ and $\mathbf v\in \mathbb C^N$,
\begin{equation}\label{optimal1}
\|\mathbb E F^{(M)}_{Q_1,\mathbf u}-F^{(M)}_{1c,\mathbf u}\| + \|\mathbb E F^{(N)}_{Q_2,\mathbf v}-F^{(N)}_{2c}\| \le N^{-1+\epsilon} 
\end{equation}
for sufficiently large $N$; if $a\ge 8$, we have
\begin{equation}\label{optimal2}
\mathbb P\left( \limsup_{N\to \infty} N^{1/2-\epsilon}\left(\|F^{(M)}_{Q_1,\mathbf u}-F^{(M)}_{1c,\mathbf u}\| + \|F^{(N)}_{Q_2,\mathbf v}-F^{(N)}_{2c}\|\right)\le 1\right) =1.
\end{equation}

%Then for any fixed $0<\epsilon<1/2$ and deterministic unit vector $\mathbf v\in \mathbb C^M$, we have
%\begin{equation}\label{boundE}
%\|\mathbb E F^{(M)}_{XX^*}(\mathbf v, \cdot) - F_{d_N}(\cdot)\| \le N^{-1+\epsilon} + N^{-a/2+2+\epsilon},
%\end{equation}
%%uniformly for all deterministic unit vectors $\mathbf v\in \mathbb C^M$, provided that $N$ is large enough. 
%%for sufficiently large $N$. Moreover, if $a>8$, then we have for any fixed $\epsilon>0$ and deterministic unit vector $\mathbf v\in \mathbb C^M$,
%and
%\begin{equation}\label{boundp}
%\mathbb P\left(\| F^{(M)}_{XX^*}(\mathbf v, \cdot) - F_{d_N}(\cdot)\|\ge N^{-1/2+\epsilon}\right) \le N^{2 - a(1+\epsilon)/4},
%\end{equation}
%%uniformly for all deterministic unit vectors $\mathbf v\in \mathbb C^M$, provided that $N$ is large enough. 
%for sufficiently large $N$. Moreover, if $a(1+\epsilon)>8$, then we have
%\begin{equation}\label{boundas}
%\mathbb P\left( \limsup_{N\to \infty} N^{1/2-\epsilon}\|F^{(M)}_{XX^*}(\mathbf v, \cdot) - F_{d_N}(\cdot)\|\le 1\right) =1,
%\end{equation}
%where we regard $X\equiv X_{N}$ as an increasing sequence of random matrices.
\end{cor}
\begin{proof}[Proof of Corollary \ref{main_cor}]
We use a standard cutoff argument. We fix $a>4$ and choose a constant $\phi>0$ small enough such that $\left(N^{1/2-\phi}\right)^{a} \ge N^{2+\omega}$ for some constant $\omega>0$. Then we introduce the following truncation  
$$\tilde X :=1_{\Omega} X, \ \ \Omega :=\left\{|x_{ij}|\le N^{-\phi} \text{ for all }  1\le i \le M,  1\le j \le N\right\}.$$
%where, for convenience, we again use $j$ to denote the column index of $X$. 
By the tail condition (\ref{tail_condition}), we have
\begin{equation}\label{XneX}
\mathbb P(\tilde X \ne X) =O ( N^{2-a/2+a\phi}).
\end{equation}
Moreover, we have
\begin{equation}\label{XneXio}
\begin{split}
& \mathbb P(\tilde X \ne X \ \text{i.o.}) = \lim_{k\to \infty} \mathbb P\left( \cup_{N=k}^\infty \cup_{i=1}^M \cup_{j=1}^N \left\{|x_{ij}|\ge N^{-\phi}\right\}\right)\\
&= \lim_{k\to \infty} \mathbb P\left( \cup_{t=k}^\infty \cup_{N\in [2^t,2^{t+1})} \cup_{i=1}^M \cup_{j=1}^N \left\{|x_{ij}|\ge N^{-\phi}\right\}\right) \\
&\le C \lim_{k\to \infty} \sum_{t=k}^\infty \left(2^{t+1}\right)^2 \left(2^{t(1/2-\phi)}\right)^{-a} \le C \lim_{k\to \infty} \sum_{t=k}^\infty 2^{-\omega t}=0,
\end{split}
\end{equation}
i.e. $\tilde X = X$ almost surely as $N\to \infty$. Here in the above derivation, we regard $M= N/d_N$ as a function depending on $N$.%, which, by the given condition on $d_N$, satisfies $M = O(N)$ for large enough $N$.
%Furthermore, by the rank inequality in \cite[Theorem A.43]{Bai_book}, we have

Using (\ref{tail_condition}) and integration by parts, it is easy to verify that %we can get that
\begin{align*}
\mathbb E  \left|x_{ij}\right|1_{|x_{ij}|> N^{-\phi}} =O(N^{-2-\omega/2}), \ \ \mathbb E \left|x_{ij}\right|^2 1_{|x_{ij}|> N^{-\phi}} =O(N^{-2-\omega/2}),
\end{align*}
which imply that
$$|\mathbb E  \tilde x_{ij}| =O(N^{-2-\omega/2}), \  \  \mathbb E |\tilde x_{ij}|^2 = N^{-1} + O(N^{-2-\omega/2}),$$
%and
$$\left| \mathbb E\tilde x_{ij}^2\right| =O( N^{-2-\omega/2}), \ \ \text{if $x_{ij}$ is complex.} $$
Moreover, we trivially have
$$\mathbb E  |\tilde x_{ij}|^4 \le \mathbb E  |x_{ij}|^4 =O(N^{-2}).$$
Hence $\tilde X$ is a random matrix satisfying Assumption \ref{main_assm}. Then using (\ref{boundE}) and (\ref{XneX}) with $a=6$ and $\phi=\epsilon/6$, we conclude (\ref{optimal1}); using (\ref{boundp}) and (\ref{XneXio}) with $\phi=(1-\epsilon)/4$ and $a=8$, we conclude (\ref{optimal2}). 
%Hence we can write $\tilde X = X_1 + B , $ where 
%%$\alpha=O(N^{-1-\omega/4})$, $\mathbf i_1$ and $\mathbf i_2$ are vectors in $\mathbb C^{\mathcal I}$ such that
%%$$\mathbf i_1(k)= N^{-1/2}, \ \ \mathbf i_1(\mu)=0, \ \ \mathbf i_2(k)=0, \ \ \mathbf i_2(\mu)=N^{-1/2}, \ \ k\in \mathcal I_1, \mu\in \mathcal I_2,$$
%%for $k\in \mathcal I_1, \mu\in \mathcal I_2$, 
%$X_1$ is a random matrix satisfying the assumptions in Theorem \ref{thm_large} and $B$ is a deterministic matrix such that 
%\begin{equation}\label{boundB}
%\max_{i,j}|B_{ij}|=O(N^{-2-\omega/4}).
%\end{equation}
\end{proof}

\begin{rem}
The estimates (\ref{optimal1}) and (\ref{optimal2}) improve the bounds obtained in \cite{XYZ2013}, and relax the assumptions on moments and $\Sigma$ as well. The convergence rates in (\ref{optimal1}) and (\ref{optimal2}) are optimal up to an $N^{\epsilon}$ factor. In fact, it was proved in \cite{Bai2007} that for an analytic function $f$,
\begin{equation}\label{CLT_bai}
\sqrt{N}\int f(x) d \left( F_{Q_1,\mathbf u}(x) - F_{1c,\mathbf u}(x) \right) \to \mathcal N(0,\sigma_{f,\mathbf u}),
\end{equation}
where $\mathcal N(0,\sigma_{f,\mathbf u})$ denotes the Gaussian distribution with mean zero and variance $\sigma_{f,\mathbf u}$. This shows that the fluctuation of $F_{Q_1,\mathbf u}(x)$ is of order $N^{-1/2}$ and suggests the bound in (\ref{optimal2}). Taking expectation of (\ref{CLT_bai}), one can see that the order of $|\mathbb EF_{Q_1,\mathbf u}(x) - F_{1c,\mathbf u}(x)|$ should be even smaller. 
%Note that
%$$\mathbb EF_{Q_1}(x) - F_{1c}(x) = \frac{1}{M}\sum_{i=1}^M \left(\mathbb EF_{Q_1,\mathbf e_i}(x) - F_{1c,\mathbf e_i}(x)\right),$$
%where $\mathbf e_i$ denotes the standard basis vector with $\mathbf e_i(j)=\delta_{ij}$. 
Moreover, the fluctuation of eigenvalues on the microscopic scale will lead to an error of order at least $N^{-1}$ by the universality of eigenvalues \cite{BPZ1,LS,PY}.
%it is known that the convergence rate of the expected ESD $\mathbb EF_{Q_1}$ is of order $N^{-1}$ by the universality of eigenvalues \cite{BPZ1,LS,PY}. 
This shows that %the convergence rate of the expected VESD is at least of order $N^{-1}$, and hence 
the bound (\ref{optimal1}) should be close to being optimal. We check the bounds (\ref{optimal1}) and (\ref{optimal2}) below with some numerical simulations; see Fig. \ref{fig_Kol}.
\end{rem}

\begin{rem}
In \cite{XYZ2013}, the authors only handle the $M < N$ (i.e. $d_N > 1$) case for $Q_1$, while our proof works for both the $d_N>1$ and $d_N<1$ cases. However, in the case with $d_N \to 1$, we will encounter some difficulties near the leftmost edge $a_{2L}$, which converges to $0$ as $N\to \infty$ and violates the regularity condition (\ref{regular1}). 
%Also the regularity assumption (\ref{regular1}) has ruled out the spiked models that have outliers. 
We will try to relax this assumption in the future.
%However, we can still prove weaker versions of (\ref{boundE}) and (\ref{boundp}) by restricting ourself to the region away from $0$. For instance, we have for any fixed $\tau>0$,
%\begin{equation*}
%\sup_{x\ge \tau}|\mathbb E F^{(M)}_{XX^*}(\mathbf v, \cdot) - F_{d_N}(\cdot)| \le N^{-1+\epsilon} 
%\end{equation*}
%under the assumptions in Theorem \ref{main_thm}. Similarly, the bound in (\ref{boundp}) also holds if we only take the $\sup$ over $x\ge \tau$. {\color{red}Also add some remarks on possible zero eigenvalues of $D$.}
\end{rem}

\begin{rem}\label{off_rem}
In Theorem \ref{main_thm}, we have assumed that $\Sigma$ is diagonal. But our results can be extended immediately to the case with a general non-diagonal population covariance matrix $\mathbf C$ for multivariate normal data. More precisely, let $X$ be a random matrix with $i.i.d.$ Gaussian entries and suppose $\mathbf C$ has eigendecomposition $\mathbf C= U^* \Sigma U$. Then we have
\begin{equation}\label{off_rem_eq}
\mathbf C^{1/2}XX^* \mathbf C^{1/2}=U^* ( \Sigma^{1/2}XX^*\Sigma^{1/2}) U \quad \text{in distribution}.
\end{equation}
Hence for any unit test vector $\mathbf u$, our results can be applied to the VESD of $\Sigma^{1/2}XX^*\Sigma^{1/2}$ with test vector $U\mathbf u$. 

For generally distributed data, under sufficiently strong moment assumptions, it is possible to prove the same results for the case with non-diagonal population covariance matrix $\mathbf C$. In particular, if the entries of $\sqrt{N}X$ have arbitrarily high moments, it can be proved that (\ref{optimal1}) and (\ref{optimal2}) hold for the VESD of $\mathbf C^{1/2}XX^* \mathbf C^{1/2}$. The main inputs for the proof will include: (a) the local law in \cite[Theorem 3.6]{Anisotropic} (which generalizes the one in Theorem \ref{lem_EG0} to the non-diagonal $\mathbf C$ case with generally distributed data), (b) Theorem \ref{main_thm} (proved for the diagonal $\mathbf C$ case),   (c) a comparison argument in \cite[Section 7]{Anisotropic} (which extends Theorem \ref{main_thm} to the non-diagonal case through comparison with the diagonal case), and (d) the Helffer-Sj{\"o}strand arguments in Section \ref{section_proof}. However, under weaker moment assumptions as in Corollary \ref{main_cor}, the proof will be much harder. For step (a), we need to use the local law proved in \cite{YF_separable}, which further generalizes the one in \cite{Anisotropic} to the heavy-tailed case. The main issue will be that the error bounds in steps (a) and (c) are not sharp enough, which does not give the optimal convergence rates as in (\ref{optimal1}) and (\ref{optimal2}). We would like to deal with this problem in the future, and focus on proving a sharp bound for the convergence rate of VESD in the diagonal $\mathbf C$ case in this article. 
\end{rem}

\begin{rem}
As discussed above, the convergence of the stochastic process $X^\Sigma_{M,\mathbf u}$ defined in (\ref{Xsigma}) to the Gaussian process $\mathbf B^\Sigma_{\mathbf u}$ in (\ref{Bsigma}) is also a very important question, which is complementary to the results in Corollary \ref{main_cor}. The convergence of $X^I_{M,\mathbf u}$ to the Brownian bridge was first proved in the null case $\Sigma=I$, for some special vectors of the form $\mathbf u=M^{-1/2}(\pm 1,\cdots, \pm 1)$ in \cite{JS1990}. The result was later extended to the case with a general fixed vector $\mathbf u$ in \cite{Bai2007}. More precisely, it was proved in \cite{Bai2007} that for any fixed vector $\mathbf u$ and analytic functions $g_1, \cdots, g_k$, the random vector
$$(\hat X_{M,\mathbf u}(g_1), \cdots, \hat X_{M,\mathbf u}(g_k)), \ \ \hat X_{M,\mathbf u}(g_i):=\int g_i(x) dX_{M,\mathbf u}^{I}(F_{Q_1}(x)), \ 1\le i \le k,$$
converges to a Gaussian vector with mean zero and certain covariance function. We expect that combining the method in \cite{Bai2007} and the new tools in this paper, one can prove a similar convergence result for $X^\Sigma_{M,\mathbf u}$ in the case with a non-scalar $\Sigma$. This will be studied in a future paper.
\end{rem}

The rest of this paper is organized as follows. In Section \ref{Sec_simulappl}, we check the results in Corollary \ref{main_cor} with some numerical simulations, and then introduce some applications of our results in high-dimensional statistical inference. We prove Theorem \ref{main_thm} in Section \ref{main_result} using Stieltjes transforms. In the proof, we mainly use Theorems \ref{lem_EG0}-\ref{thm_large}, which give the desired anisotropic local laws for the resolvents of $Q_1$ and $Q_2$. 
Theorem \ref{lem_EG} constitutes the main novelty of this paper, and its proof will be given in Section \ref{proof_lem_EG}. The proofs of Theorem \ref{lem_EG0} and Theorem \ref{thm_large} will be given in the supplementary material.

\section{Simulations and applications}\label{Sec_simulappl}

%{\color{red}
%\begin{rem}\label{rem_TX}
%It is possible to generalize our proof to more general %general random matrix models. For example, one may consider 
%sample covariance matrices of the form $Q:=(TX)(TX)^*$, with $T$ being a general deterministic (non-diagonal) rectangular matrix. However, this is quite nontrivial and requires more novel ideas and techniques. We will present it in another paper. 
%% generalized Wigner matrices (i.e. Wigner ensembles whose entries have non-identical variances), and deformed Wigner matrices of the form $H+A$ (where $H$ is a Wigner matrix and $A$ is a deterministic Hermitian matrix). The convergence of VESD of these models will be studied in future works. In particular, we expect that our proof applied to the Wigner matrices can improve the results obtained in \cite{XZ2016}. 
%\end{rem}
%}

In this section, we first check the convergence rate of the (expected) VESD to the deformed MP law with some numerical simulations. Then we will discuss briefly the applications of our results in high-dimensional statistical inference procedures.

\subsection{Simulations}\label{sec_appl}

The simulations are performed under the following setting: $M= 2N$, i.e. $d_N=0.5$; the entries $\sqrt{N}x_{ij}$ are drawn from a distribution $\xi$ with mean zero, variance 1 and tail $\mathbb P(|\xi| \ge s) \sim s^{-6}$ for large $s$; the unit vector $\mathbf v$ is randomly chosen for each $N$. In Fig. \ref{fig_Kol}, we plot the Kolmogorov distances $\|F_{Q_2,\mathbf v}-F_{2c}\|$ and $\|\mathbb E F_{Q_2,\mathbf v}-F_{2c}\|$ for the following two choices of $\Sigma$: $\Sigma=I$ with ESD $\pi=\delta_1$, and 
\begin{equation}\label{sigma12}
\Sigma = {\text{diag}}(\underbrace{1,\cdots, 1}_{M/2},\underbrace{4,\cdots, 4}_{M/2}), \ \ \text{ with ESD } \pi= 0.5\delta_1 + 0.5\delta_4.
\end{equation} 
For each $N$, we take an average over 10 repetitions to represent $F^{(N)}_{Q_2,\mathbf v}$ and an average over $4N^2$ repetitions to approximate $\mathbb E F^{(N)}_{Q_2,\mathbf v}$. Under each setting, we choose an appropriate function $f(x)$ to fit the simulation data. It is easy to observe that the convergence rate of the VESD is bounded by $O(N^{-1/2})$, while the convergence rate of the expected VESD has order $N^{-1}$. This verifies the results in Corollary \ref{main_cor}.
%expected VESD of $XX^*$ and the MP distribution for $M$ from 50 to 800. One can see that the convergence rate of the VESD is indeed of the order $M^{-1}$. %Moreover, the vector $\mathbf v$ indeed is randomly chosen for each $M$, which suggests that the results does not depend on $\mathbf v$.
%We choose the ESD of $\Sigma$ to be $\pi=\delta_1$ in (a) and $\pi= 0.5\delta_1 + 0.5\delta_2$ in (b).

\begin{figure}[htb]
\centering
\subfigure[$\pi= \delta_1$]{\includegraphics[width=6.2cm]{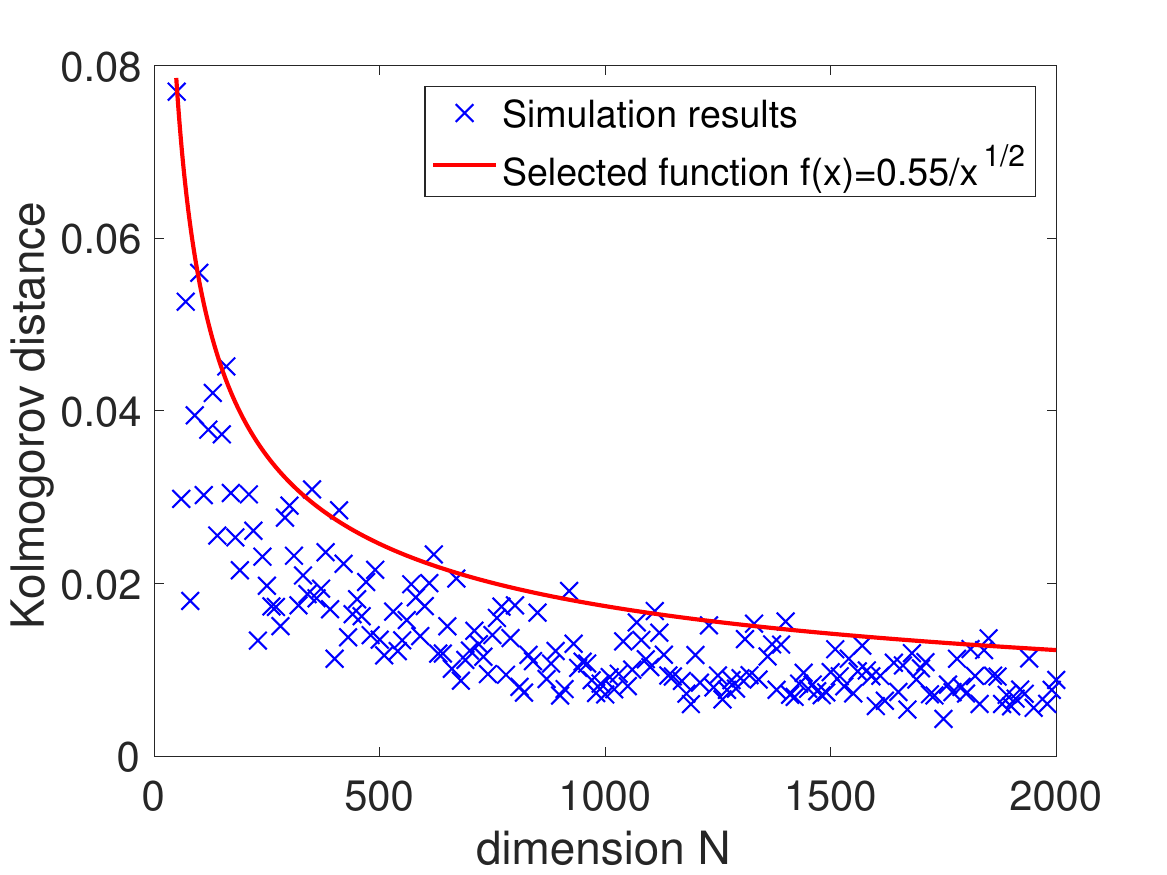}
\includegraphics[width=6.2cm]{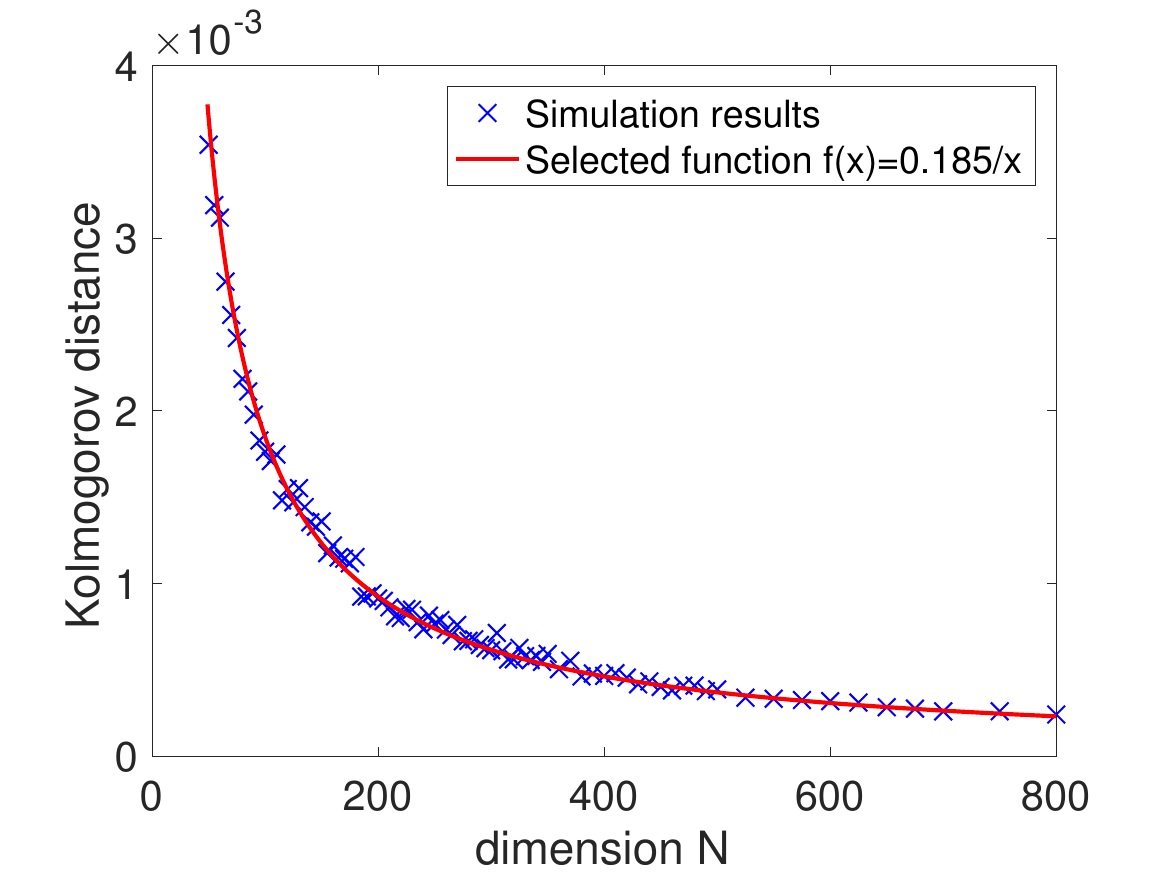}}
\subfigure[$\pi= 0.5\delta_1 + 0.5\delta_4$]{\includegraphics[width=6.2cm]{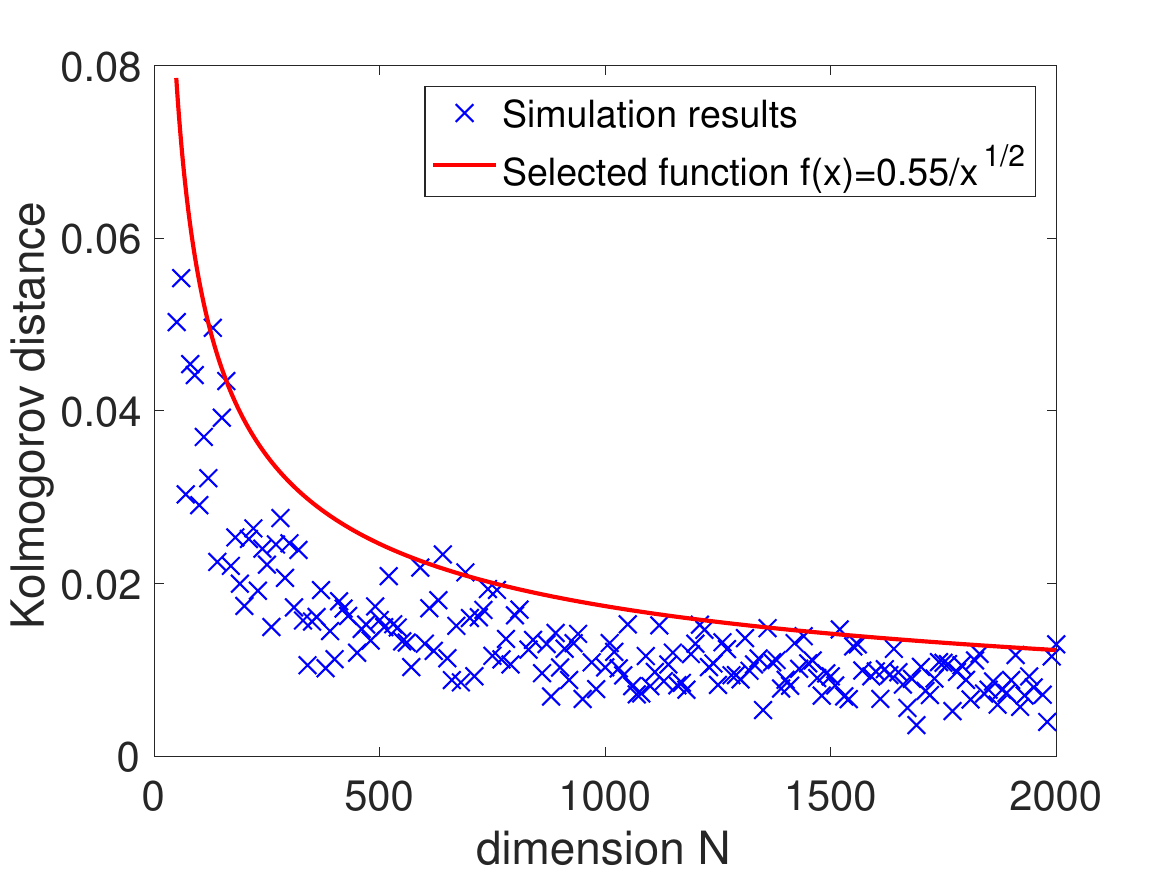}
\includegraphics[width=6.2cm]{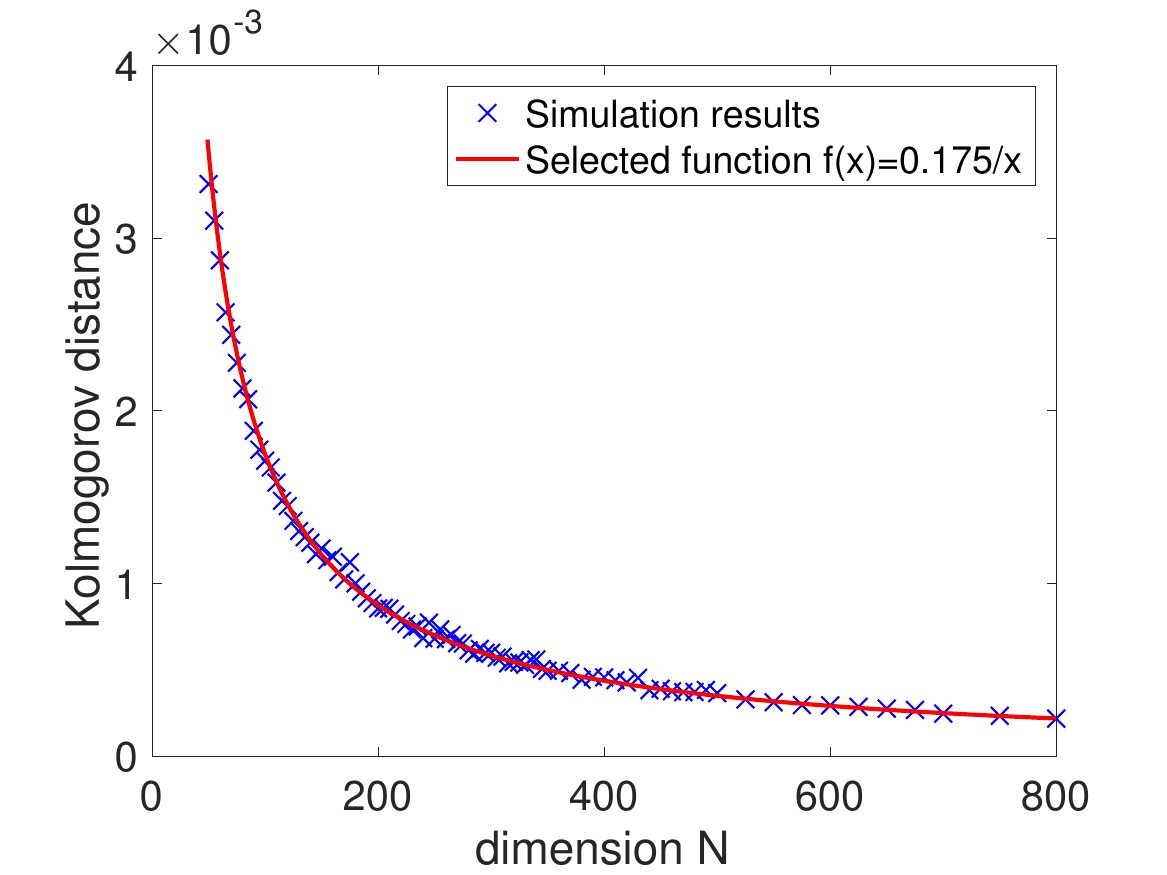}}
\caption{%The simulations are performed under the following setting: (i) $M= 2N$, i.e. $d_N=0.5$; (ii) the entries $\sqrt{N}x_{ij}$ are drawn from a distribution $\omega$ with mean zero, variance 1 and tail $\mathbb P(|\omega| \ge s) \sim s^{-6}$ for large $s$; (iii) the vector $\mathbf v$ is randomly chosen for each $N$. We choose the ESD of $\Sigma$ to be $\pi=\delta_1$ in (a) and $\pi= 0.5\delta_1 + 0.5\delta_2$ in (b). ... 
The left figures of (a) and (b) plot $\|F^{(N)}_{Q_2,\mathbf v} - F_{2c}^{(N)}\| $ as $N$ increases from 50 to 2000, and we choose $f$ to fit the upper envelope of the data. The right figures plot $\|\mathbb E F^{(N)}_{Q_2,\mathbf v} - F_{2c}^{(N)}\|$ as $N$ increases from 50 to 800.
%choose the ESD of $\Sigma$ to be $\pi=\delta_1$ and $\pi= 0.5\delta_1 + 0.5\delta_4$, respectively. ... We take an average over at least $4N^2$ repetitions to report the $\mathbb E F^{(N)}_{Q_2,\mathbf v}$ for each $N$. Then we calculate the Kolmogorov distance $\|\mathbb \mathbb E F^{(N)}_{Q_2,\mathbf v} - F_{2c}^{(N)}\| $, and choose the function $f(x)$ to fit the simulation data.
}
\label{fig_Kol}
\end{figure}

%... As demonstrated in \cite{JS1989,JS1990,XYZ2013}, 
As discussed before, the convergence of $F_{Q_2,\mathbf v}$ to $F_{2c}$ for any sequence of deterministic unit vectors $\mathbf v$ can be used to characterize the asymptotical Haar property of the eigenmatrix of $Q_2=X^*\Sigma X$ (which also implies the asymptotical Haar property of the eigenmatrix of $Q_1$ when $\Sigma=\sigma^2 I$). 
%Thus the bounds in Corollary \ref{main_cor} for the VESD of large sample covariance matrices can assist us in better studying spiked covariance matrices as assumed in \cite{Cai2013,Ma2013} among many others.  
%for the convergence rate will lead to a better understanding of the Haar properties of the eigenvectors of large sample covariance matrices.
On the other hand, for a general $\Sigma$, the eigenmatrix of $Q_1$ is not asymptotically Haar distributed anymore and the VESD of $Q_1$ will depend on $\mathbf v$. Moreover, (\ref{m1cv}) gives an explicit dependence of $\mathbf F_{1c}$ on $\Sigma$, which should be of interest to statistical applications. {(For more details on the application of this principle, the reader can refer to the discussions in Section \ref{section_appl3}.)} %For instance, one can use the VESD to detect the variance structure of $\Sigma$. 
In Fig.\,\ref{fig_VESD}(a), we plot $F_{Q_1,\mathbf v}$ for $\Sigma$ in (\ref{sigma12}) and different choices of $\mathbf v_i$, $i=1,2,3$. One can observe a transition of $F_{Q_1,\mathbf v}$ when $\mathbf v$ changes from the direction corresponding to the smaller eigenvalues of $\Sigma$ to the direction corresponding to the larger eigenvalues of $\Sigma$. {In Fig.\,\ref{fig_VESD}(b), we take $\Sigma=UDU^*$, where $D$ is as in (\ref{sigma12}), $U$ is a randomly chosen unitary matrix, and $\mathbf w_i = U\mathbf v_i$. 
One can see that even if $\Sigma$ is non-diagonal, the convergence of the VESD of $Q_1$ still holds (see Remark \ref{off_rem}). %{\color{red}need to add something?}}

\begin{figure}[htb]
\centering
\includegraphics[width=\columnwidth]{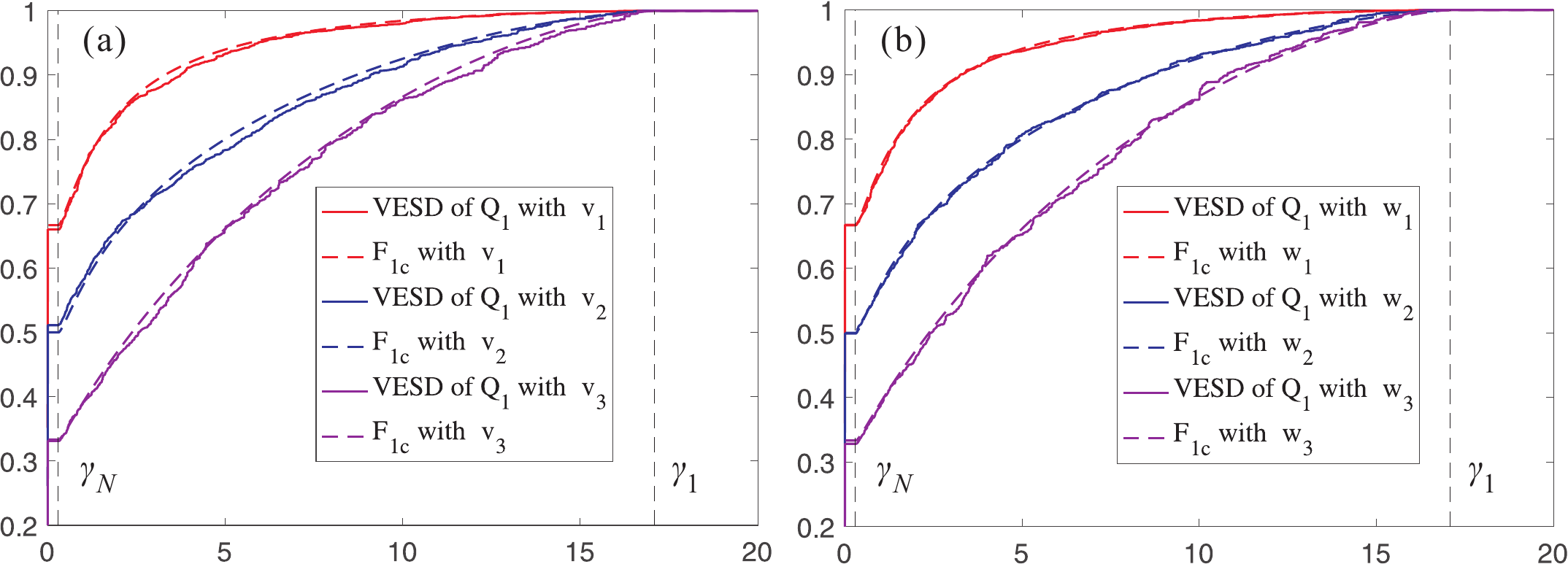}
\caption{ The plots for $F_{Q_1,\mathbf v}(x)$ and $F_{1c,\mathbf v}(x)$ with $N=2000$, $M=2N$ and under the settings in Fig. \ref{fig_Kol}. We take ${\mathbf v}_1 = \sqrt{{2}/{M}}({\protect\underbrace{1,\cdots,1}_{M/2}}, 0 , \cdots, 0), $ $\mathbf v_3=\sqrt{{2}/{M}}(0 , \cdots, 0,{\protect\underbrace{1,\cdots,1}_{M/2}}),$ $\mathbf v_2=({\mathbf v_1 + \mathbf v_3})/\sqrt{2}$, and $\mathbf w_i = U\mathbf v_i$, $i=1,2,3$. The dashed lines mark the places of the left edge $\gamma_N$ and the right edge $\gamma_1$ of the spectrum (recall (\ref{gammaj})).}
\label{fig_VESD}
\end{figure}

%\begin{equation}\label{v123}
%\begin{split}
%{\mathbf v}_1 = \sqrt{\frac{2}{M}}({\underbrace{1,\cdots,1}_{M/2}}, 0 , \cdots, 0), \ \ & \mathbf v_3=\sqrt{\frac{2}{M}}(0 , \cdots, 0,{\underbrace{1,\cdots,1}_{M/2}}), \\
%\mathbf v_2&=\frac{\mathbf v_1 + \mathbf v_3}{\sqrt{2}}. 
%\end{split}
%\end{equation}
%The simulations are performed under the following setting: (i) $M= 2N$, i.e. $d_N=0.5$; (ii) the entries $\sqrt{N}x_{ij}$ are drawn from a distribution $\omega$ with mean zero, variance 1 and tail $\mathbb P(|\omega| \ge s) \sim s^{-6}$ for large $s$; (iii) the vector $\mathbf v$ is randomly chosen for each $N$. We choose the ESD of $\Sigma$ to be $\pi=\delta_1$ in (a) and $\pi= 0.5\delta_1 + 0.5\delta_4$ in (b).

\subsection{Statistical applications}\label{sec_appl2}

% We now introduce some potential applications of our results in high-dimensional statistical inference. 
 %Our convergence result have potential applications in many areas, 

\subsubsection{Detection of signals in noise}

%We now give a few concrete examples of applications to multivariate statistics, empirical finance and signal processing. 
Consider the following model:
\begin{equation} \label{model_application}
\mathbf{x} = A\mathbf{s} + \mathbf z,
\end{equation}
where $A$ is an $M\times k$ deterministic matrix, $\mathbf{s}$ is a $k$-dimensional mean zero signal vector, and $\mathbf{z}$ is an $M$-dimensional noise vector with $i.i.d.$ centered entries. Moreover, the signal vector and the noise vector are assumed to be independent. In practice, suppose we observe $N$ such $i.i.d.$ samples and set the matrix $X=(\mathbf x_1, \cdots, \mathbf x_N)$. This signal-plus-noise model is a standard model in classic signal processing \cite{SMK}. A fundamental task is to detect the signals via observed samples, and the very first step is to know whether there exists any such signal, i.e., %Hence our hypothesis testing problem can be formulated as 
\begin{equation}\label{model_null0}
\mathbf{H}_0: \ k=0 \quad \text{vs.} \quad \mathbf{H}_1: \ k \geq 1.
\end{equation}
The model (\ref{model_application}) is also widely used in various other fields, such as multivariate statistics, wireless communications, bioinformatics, and finance. %The model (\ref{model_application}) has many applications in statistics, we motivate it using {\color{red} three} examples.
For example, in multivariate statistics one wants to determine whether there exists any relation between two sets of variables. To test the independence, we can adopt the multivariate multiple regression model (\ref{model_application}),  where $\mathbf{x}$ and $\mathbf{s}$ are the two sets of variables for testing \cite{JW}. Then we can test the null hypothesis that these regression coefficients are all zero:
\begin{equation} \label{model_null}
\mathbf{H}_0: \ A=0 \  \ \text{vs}. \ \ \mathbf{H}_1: \ A \neq 0.
\end{equation}
Another example is from financial studies \cite{FF, FFL, FLL}. In the empirical research of finance, (\ref{model_application}) is the factor model, where $\mathbf{s}$ is the common factor, $\Gamma$ is the factor loading matrix and $\mathbf{z}$ is the idiosyncratic component. In order to analyze the stock return $\mathbf{x},$ we first need to know if the factor $\mathbf{s}$ is significant for the prediction. Then a statistical test can be also constructed as (\ref{model_null}).

\begin{figure}[htb]
\centering
%\subfigure{
\includegraphics[width=6.8cm]{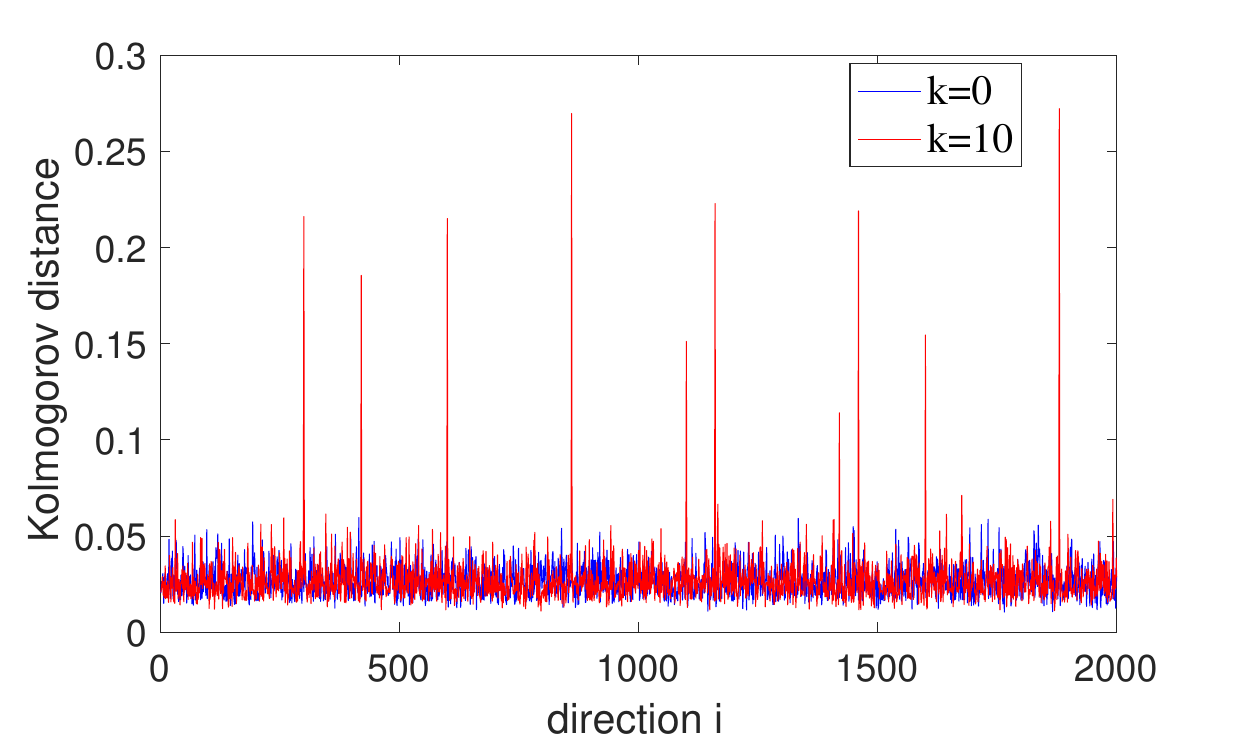}
%}
\caption{$\|F^{(M)}_{Q_1,\mathbf e_i} - F_{MP}\| $ for $Q_1= XX^*$, where $X$ is the sample matrix of $\mathbf x$ in (\ref{model_application}).}
\label{fig_signal}
\end{figure}

For the above hypothesis testing problems (\ref{model_null0}) and (\ref{model_null}), under the null hypothesis $\mathbf{H}_0$, we have that $F_{Q_1,\mathbf u}^{(M)}=F_{MP}+O(M^{-1/2+\epsilon})$ for any unit vector $\mathbf u$ independent of $X$  by our results. As an example, we perform a simulation under the following setting: $M=2000$, $N= 2M$; the entries $\sqrt{N}z_{ij}$ are $i.i.d.$ Gaussian with mean 0 and variance 1; the entries $\sqrt{N}s_{ij}$ are $i.i.d.$ Bernoulli $\pm 1$ random variables. We choose $A=DV$, where $V$ is a randomly chosen unitary matrix and $D$ is an $M\times k$ matrix satisfying the following: all the entries of $D$ are zero except $D_{n(i)i}$, and each $D_{n(i)i}$ is sampled uniformly from $[0.4,0.8]$. Here $n(i)$, $1\le i \le k$, are $k$ values sampled uniformly at random from the integers 1 to $M$. In Fig. \ref{fig_signal}, we plot the Kolmogorov distances $\|F_{Q_1,\mathbf e_i}- F_{MP}\|$ with respecto to $i$, where $\mathbf e_i$ denotes the standard unit vector along $i$-axis. Comparing the $k=10$ case with the null case, we observe 10 obvious peaks. Moreover, the positions of the peaks correspond to the values of $n(i)$, and the heights of the peaks give the strengths of the signals. Note that if one use the bound $M^{-1/4}$ in \cite{XYZ2013}, then the {estimated} noise would be of order $0.15$, which does not allow one to detect the smallest few signals.

For Gaussian noise, some classical statistical procedures to test the number of signals usually use the largest eigenvalue of the sample covariance matrix \cite{BDMN,NE,NS}. The key property is that the largest eigenvalue converges to the Gaussian distribution under the $N^{1/2}$ scaling if it is an outlier, and the Tracy-Widom distribution under the $N^{2/3}$ scaling otherwise. Onatski proposed to use the test statistic $R=(\lambda_1 - \lambda_2)/(\lambda_2-\lambda_3)$, which is asymptotically pivotal \cite{Economics}. Our method is more general in the sense that it can be also applied in the case without outliers. For example, one can check numerically that the sample covariance matrices in Fig.  \ref{fig_signal} has no outliers. %Moreover, our method not only detects the number of signals, but also can give some information of the directions of the signals. 

\subsubsection{Separable covariance matrices}

Consider data matrices of the form
\begin{equation}\label{separaY}
Y= \Sigma_1^{1/2} X \Sigma_2^{1/2},
\end{equation}
where $X$ is an $M\times N$ random matrix as in Corollary \ref{main_cor}, and $\Sigma_{1}$ and $\Sigma_2$ are $M\times M$ and $N\times N$ deterministic positive-definite matrices, respectively. Then $Q_1:= YY^*$ is called a {\it separable covariance matrix}, and it is widely used to model the spatio-temporal sampling data \cite{Karoui2009,  Separable, WANG2014}. Without loss of generality, we shall call $\Sigma_1$ the spatial covariance matrix and $\Sigma_2$ the temporal covariance matrix. Suppose we want to determine whether the spatial identity holds, i.e.,
$$ \mathbf{H}_0: \ \Sigma_1 \propto I \  \ \text{vs}. \ \ \mathbf{H}_1: \ \Sigma_1 \not\propto I.$$
For this hypothesis testing problem, under $\mathbf{H}_0,$ we have that $F_{Q_1,\mathbf u_1}^{(M)}=F_{Q_1,\mathbf u_2}^{(M)} +O(M^{-1/2+\epsilon})$ for any unit vectors $\mathbf u_{1,2}$ independent of $X$. More generally, we can test whether $\Sigma_1 = \Sigma_0$ for some given positive definite matrix $\Sigma_0$ by using $\Sigma_0^{-1/2}Y$. Similarly, the temporal identity can also be tested using the VESD of $Q_2:= Y^*Y$. Note that our error bound allows us to test very weak signals up to order $M^{-1/2}$ with one sample. The precision can be further improved if one can take average over many samples. 
\begin{figure}[htb]
\centering
%\subfigure{
\includegraphics[width=6.2cm]{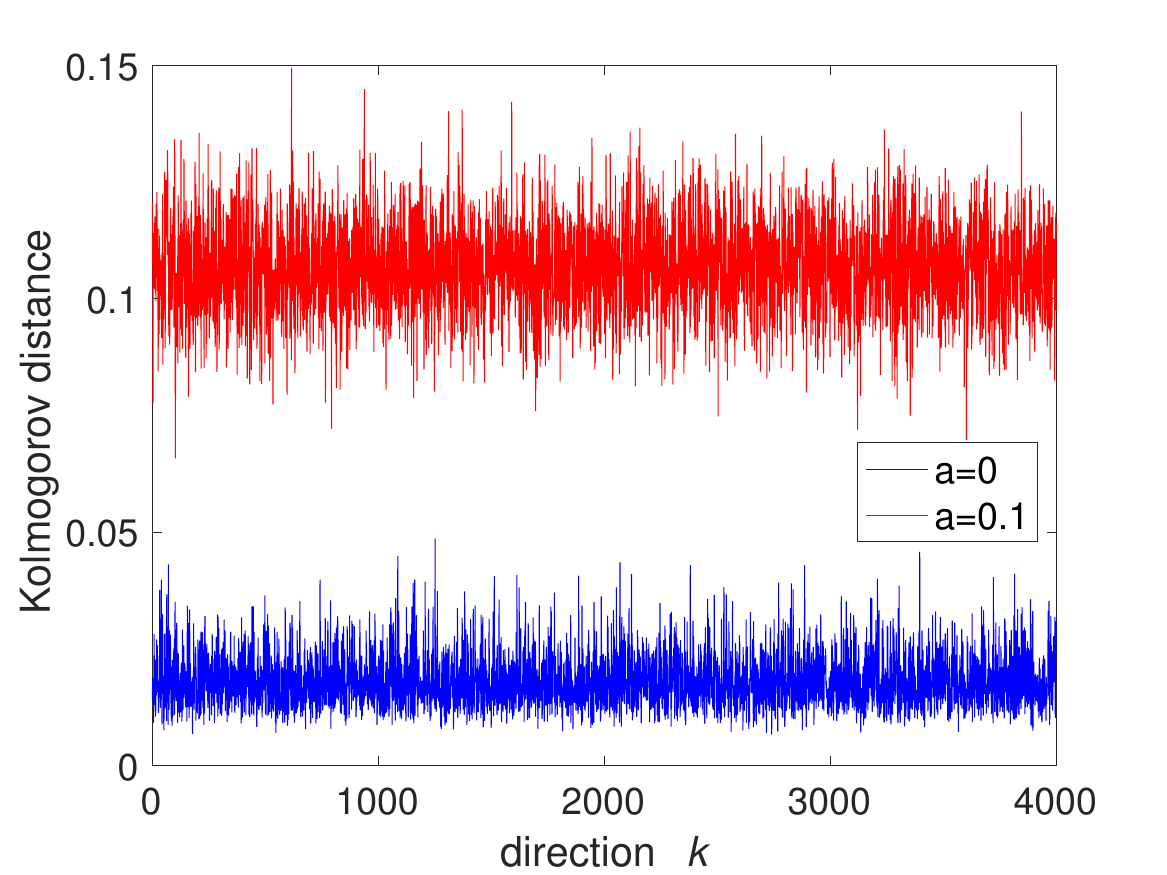}
\includegraphics[width=6.2cm]{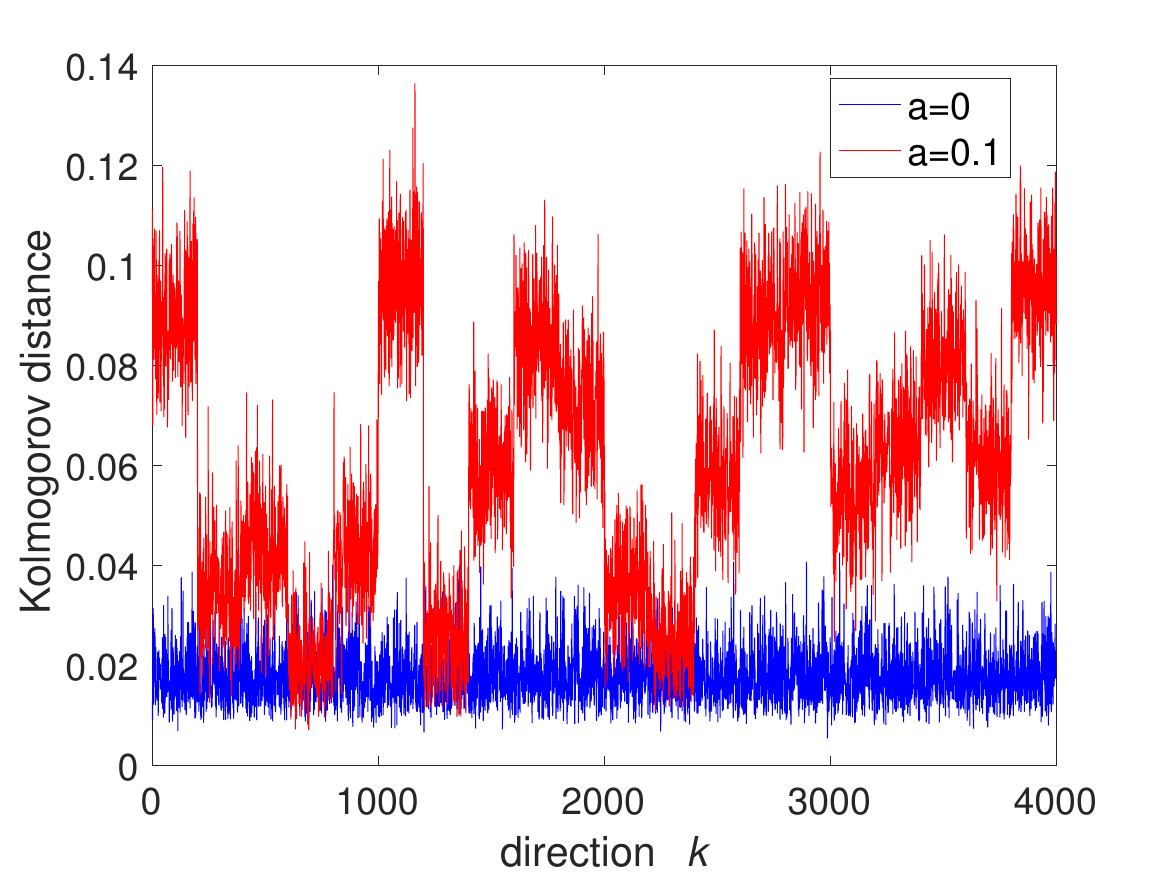}
%}
\caption{We plot $\|F^{(M)}_{Q_1,\mathbf e_k} - M^{-1}\sum_k F_{Q_1,\mathbf e_k}\| $ for $Q_1= YY^*$ with $\Sigma_{1,2}$ in (\ref{separaY}). We take $M=4000$, $N=2M$, $A$ as in (\ref{case1}) in the left figure, and $A$ as in (\ref{case2}) in the right figure.}
\label{fig_separable}
\end{figure}
We now illustrate this application with some numerical simulations. The simulations are performed under the following setting: $N= 2M$; the entries $\sqrt{N}x_{ij}$ are $i.i.d.$ Gaussian with mean 0 and variance 1. We consider separable covariance matrices of the form (\ref{separaY}) with 
\begin{equation}\label{separa12}
\Sigma_2^{1/2} = {\text{diag}}(\underbrace{1,\cdots, 1}_{N/2},\underbrace{2,\cdots, 2}_{N/2}),  \quad \Sigma_1^{1/2} = I + a A,
\end{equation} 
where
\begin{equation}\label{case1}
A={\text{diag}}(\underbrace{1,\cdots, 1}_{200},\underbrace{-1,\cdots, -1}_{200}, \underbrace{1,\cdots, 1}_{200}, \cdots) , 
\end{equation}
or for $i.i.d.$ sequence of random variables $b_1, b_2 , \cdots \sim \text{unif}(-1,1)$,
\begin{equation}\label{case2}
A={\text{diag}}(\underbrace{b_1,\cdots, b_1}_{200},\underbrace{b_2,\cdots, b_2}_{200}, \underbrace{b_3,\cdots, b_3}_{200}, \cdots).
\end{equation}
In Fig. \ref{fig_separable}, for $M=4000$, $a=0.1$ and the above two choices of $A$, we plot the Kolmogorov distances $\|F_{Q_1,\mathbf e_k}- M^{-1}\sum_k F_{Q_1,\mathbf e_k}\|$ with respecto to $k$. 
%, where $\mathbf e_k$ is the standard unit vector along $k$-axis. 
We compare them with the results in the null case with $a=0$, and observe very obvious signals. Note that if one use the bound $M^{-1/4}$ in \cite{XYZ2013}, then the {estimated} noise would be of order $0.126$, which does not allow one to detect such ``weak" signals. 

For this problem, \cite{BPZ1} proposed to use the largest eigenvalue $\lambda_1(YY^*)$ as a test statistic. But it has the disadvantage that the limiting distribution of $\lambda_1$ depends on the unknown matrices $\Sigma_1$ and $\Sigma_2$, and hence is not asymptotically pivotal. Moreover, it was proved in \cite{YF_separable} that the behavior of $\lambda_1$ in the non-identity $\Sigma_1$ case is similar to the one in the identity case, which is not good for test purpose. On the other hand, our procedure tests the isotropic property of $\Sigma_1$ directly. 

\subsubsection{Eigenvectors of population covariance matrices}\label{section_appl3}

Now we go back to consider the sample covariance matrices $Q_1 = \Sigma^{1/2}XX^* \Sigma^{1/2}$. By Corollary \ref{main_cor}, we know that the VESD $F_{Q_1,\mathbf u}^{(M)}$ converges to $F_{1c,\mathbf u}^{(M)}$, which is defined through the Stieltjes transform (\ref{m1cv}). It is easy to observe that the matrix $\mathbf F_{1c}^{(M)}$ is diagonal in the eigenbasis of $\Sigma$, and the diagonal entries depend on the eigenvalues of $\Sigma$ in an explicit way. This allows one to use the VESD of $Q_1$ to detect the leading eigenvectors (or eigenspaces) of $\Sigma$. More precisely, if $\mathbf u_i$ is the eigenvector of $\Sigma$ with eigenvalue $\sigma_i$, then with (\ref{m1cv}) and the inverse formula we can get that
%$$m^{(M)}_{1c,\mathbf u_i}(z):=- \frac{1}{z(1+\sigma_im_{2c}(z))}, \quad \sigma_1 \ge \sigma_2 \ge \cdots . $$
%Thus we have
$$\rho^{(M)}_{1c,\mathbf u_i}(E)= \frac1\pi\lim_{\eta\downarrow 0} \Im\, m^{(M)}_{1c,\mathbf u_i}(E+\ii\eta)= \frac{\rho_{2c}(E)}{E\left(\sigma_i^{-1} + 2 \Re\, m_{2c}(E) + |m_{2c}(E)|^2 \sigma_i\right)},$$
where $E\in \mathbb R$, $\rho^{(M)}_{1c,\mathbf u_i}$ is the density of $F_{1c,\mathbf u_i}^{(M)}$, and we abbreviate $m_{2c}(E)\equiv m_{2c}(E+\ii 0_+)$. Near the right edge $\gamma_1$, we know that $- \sigma_1^{-1} <m_{2c}(E) < 0$  (see \cite[Appendix A]{Anisotropic}). Hence it is easy to see that there exists a constant $c>0$ such that for $\gamma_1 - c < E < \gamma_1$, $\rho^{(M)}_{1c,\mathbf u_i}(E)$ is monotone with respect to $\sigma_i$. In particular, $\rho^{(M)}_{1c,\mathbf u}(E)$ is maximized if $\mathbf u = \mathbf u_1$. %and minimized if $\mathbf u = \mathbf u_M$. 
Thus our results shows that measuring the density (i.e. the slope) of $F_{Q_1,\mathbf u}^{(M)}$ allows one to make some inference on the overlaps between the test vectors and the population eigenvectors corresponding to the leading eigenvalues of $\Sigma$.

In Fig.\,\ref{fig_VESD3}, we give two examples of VESD of spiked covariance matrices. In the simulations, we take $M=1000$ and the entries $\sqrt{N}x_{ij}$ to be $i.i.d.$ Gaussian with mean 0 and variance 1. {One can take the population covariance matrix to be a general positive definite matrix, but for simplicity we assume that it is diagonal by properly rotating the test vectors; see Remark \ref{off_rem} and (\ref{off_rem_eq}).} In Fig.\,\ref{fig_VESD3}(a), we take $N=2M$, and 
\begin{equation*}
\Sigma^{1/2} = {\text{diag}}(\underbrace{1,\cdots, 1}_{0.9M},\underbrace{2,\cdots, 2}_{0.1M}).
\end{equation*}
In Fig.\,\ref{fig_VESD3}(b), we take $N=10M$, and 
\begin{equation*}
\Sigma^{1/2} = {\text{diag}}(\underbrace{1,\cdots, 1}_{0.6M},\underbrace{3,\cdots, 3}_{0.3M},\underbrace{5,\cdots, 5}_{0.1M}).
\end{equation*}
Moreover, we take the following test vectors (up to normalization):
\begin{equation*}
\begin{split}
& \mathbf v_1\propto (\underbrace{1,\cdots, 1}_{M/2},\underbrace{0,\cdots, 0}_{M/2}), \quad \mathbf v_2 \propto  (1,\cdots, 1), \quad \mathbf v_3\propto (\underbrace{0,\cdots, 0}_{M/2},\underbrace{1,\cdots, 1}_{M/2}),  \\
& \mathbf v_4\propto (\underbrace{0,\cdots, 0}_{0.7M},\underbrace{1,\cdots, 1}_{0.3M}),\quad \mathbf v_5\propto (\underbrace{0,\cdots, 0}_{0.9M},\underbrace{1,\cdots, 1}_{0.1M}). 
\end{split}
\end{equation*}
For each choice of $\mathbf v_i$, we take an average over 10 repetitions to get $F^{(M)}_{Q_1,\mathbf v_i}$.

\begin{figure}[htb]
\centering
\subfigure[Two components case]{\includegraphics[width=6.2cm]{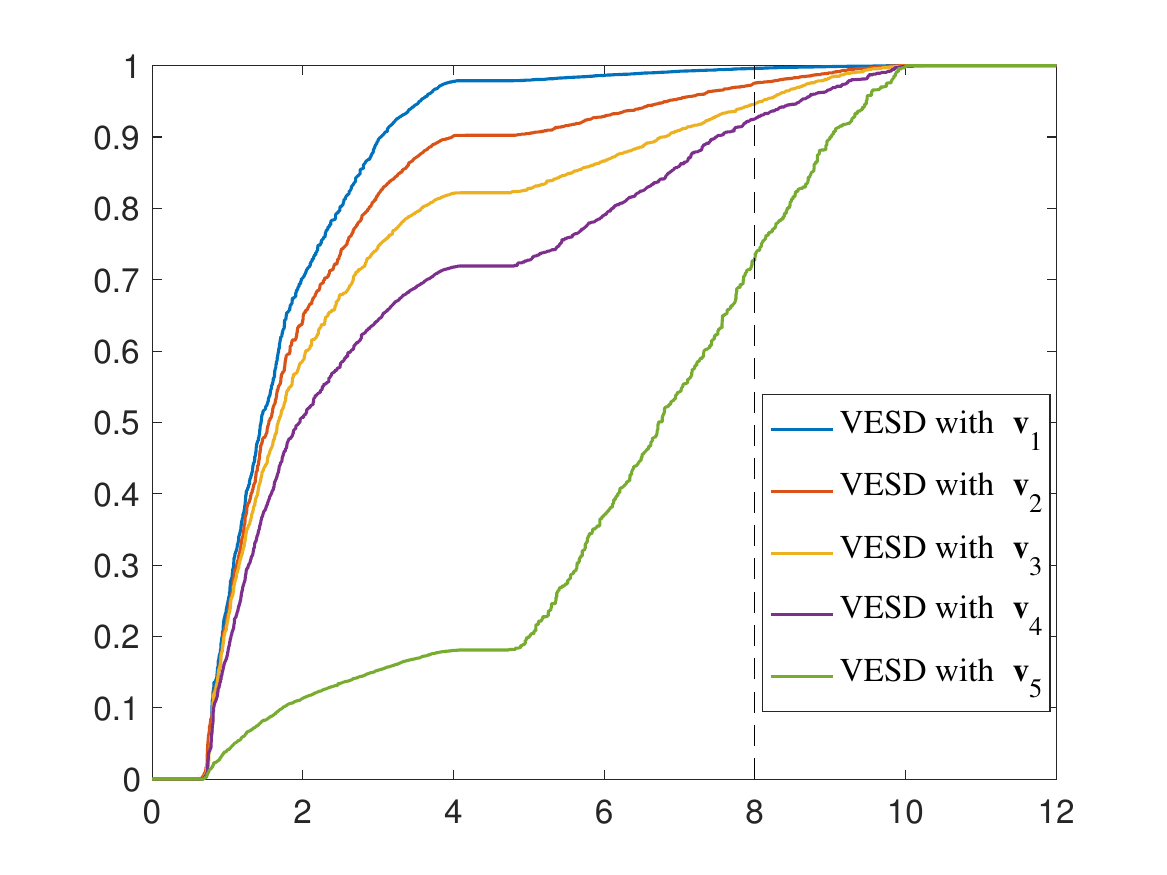}}
\subfigure[Three components case]{\includegraphics[width=6.2cm]{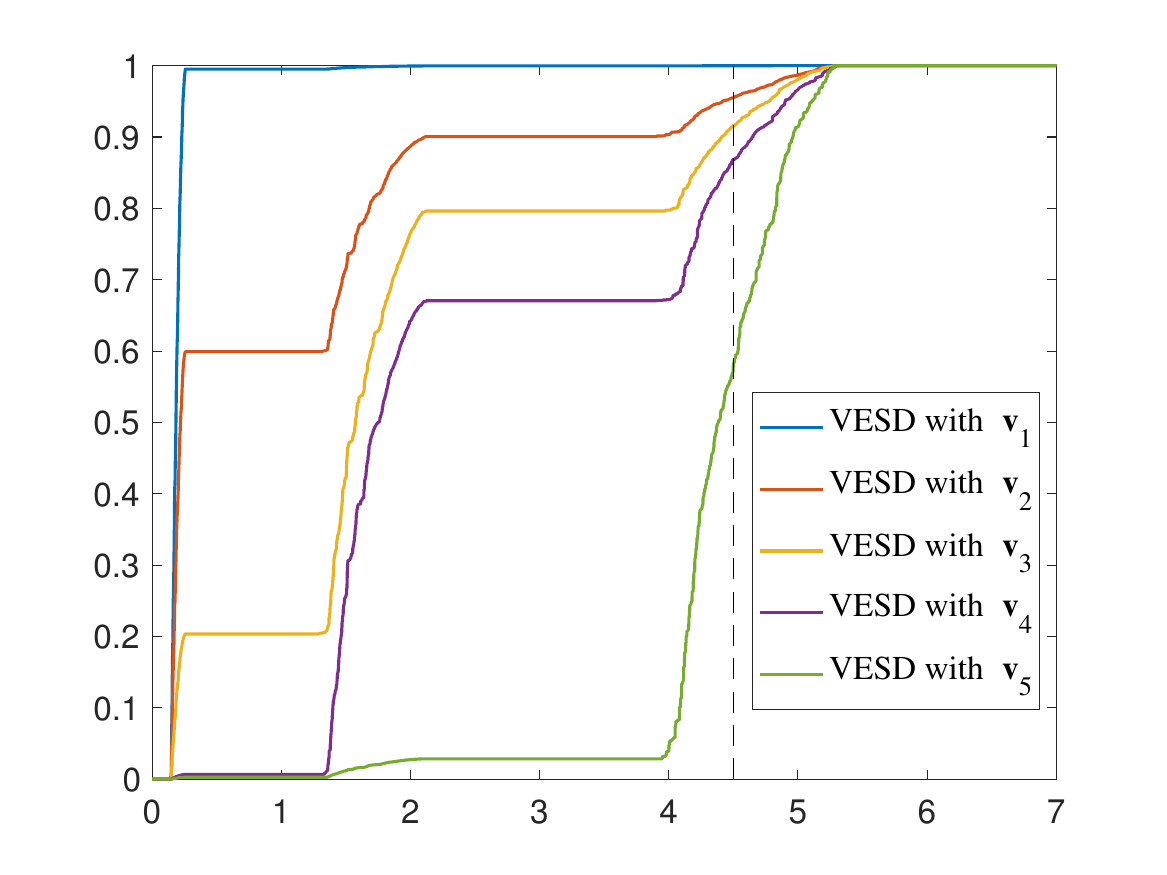}}
\caption{The plots for $F_{Q_1,\mathbf v_i}(x)$ with $M=1000$ and $i=1,\cdots, 5$.}
\label{fig_VESD3}
\end{figure}

Note that the flat parts of the curves in Fig.\,\ref{fig_VESD3} correspond to the gaps between different components of the eigenvalue spectrum of $Q_1$. Hence the spectral densities in Fig.\,\ref{fig_VESD3}(a) and \ref{fig_VESD3}(b) have two and three components, respectively. The rightmost components can be formally regarded as the outlier component  
caused by the large eigenvalues of $\Sigma$. It is easy to see that for $x$ near the right edge (e.g. the $x$ marked by the dashed line), the slope of the VESD $F_{Q_1,\mathbf v_i}(x)$ increases as $i$ changes from 1 to 5. This verifies our previous conclusion, i.e. the density $\rho_{1c,\mathbf v}$ increases if $\mathbf v$ has more overlap with the leading eigenvectors of $\Sigma$. Note that since all the VESD curves reach 1 at the right edge $\gamma_1$, the lower curves have larger densities.  

%\vspace{5pt}

%At the end of this section, 
%We would like to remark that 
Here we have only considered examples with diagonal $\Sigma$. However, our results is possible to be applied to more general and complicated sample covariance matrices with nonzero correlations between rows, i.e. non-diagonal population covariance matrix $\mathbf C$ (see Remark \ref{off_rem}). This gives much more insight into future applications of our results in high-dimensional statistical inference. {We also remark that in \cite{Peche}, the overlaps between sample eigenvectors and population eigenvectors are studied through certain functionals that are closely related to VESD (with test vectors being specified to be the population eigenvectors). Based on the results in \cite{Peche}, certain estimator was proposed to estimate the population covariance $\mathbf C$ \cite{RIE}. However, this estimator does not provide much information about the population eigenvectors since it uses the same eigenvectors as the sample covariance matrix $Q_1$. }

\section{Proof of Theorem \ref{main_thm}}\label{main_result}

%In this preliminary section, we collect some tools that will be used in the proof.
%\begin{rem}
For definiteness, we will focus on {\it{real}} sample covariance matrices during the proof. However, our proof also applies, after minor changes, to the {\it{complex}} case if we include the extra assumption (\ref{entry_assm2}) or (\ref{entry_assmex}). %Also, we will only use $d_N$ (instead of $d$) in the rest of this paper. Correspondingly, we will use the quantities $\rho_{1c}^{(N)}$, $m_{1,2c}^{(N)}$ and $\lambda_{\pm}^{(N)}$, which are obtained by replacing $d$ with $d_N$ in (\ref{mp_density})-(\ref{m2c}). For simplicity, we shall always omit the superscript and still call them $\rho_{1c}$, $m_{1,2c}$ and $\lambda_{\pm}$ in the proof.
%\end{rem}

%\subsection{Main ideas}
\subsection{Anisotropic local Mar{\v c}enko-Pastur law}

A basic tool for the proof is the Stieltjes transform. For any $z= E+i\eta \in \mathbb C_+$, we define the resolvents (the Green functions) of $Q_1$ and $Q_2$ as
\begin{equation}\label{def_green}
\mathcal G_1(X,z):=(Q_1 - z)^{-1} , \ \ \ \mathcal G_2 (X,z):=(Q_2 - z)^{-1} .
\end{equation}
Then the Stieltjes transforms of the ESD of $Q_{1,2}$ are equal to
\begin{equation*}
m_1(X,z):= {M}^{-1} \mathrm{Tr} \, \mathcal G_1(X,z), \ \ m_2(X,z):= N^{-1}\mathrm{Tr} \, \mathcal G_2(X,z), %\label{EEEE}
\end{equation*}
%Similarly, we also define $m_2(X,z):= N^{-1}\mathrm{Tr} \, \mathcal G_2(X,z)$.
and the Stieltjes transforms of $F^{(M)}_{Q_1, \mathbf u}$ and $F^{(N)}_{Q_2, \mathbf v}$ are equal to $\langle \mathbf u, \mathcal G_1(X,z)\mathbf u\rangle$ and $\langle \mathbf v, \mathcal G_2(X,z)\mathbf v\rangle$, respectively. The main goal of this subsection is to establish the following asymptotic estimate for $z\in \mathbb C_+$:%when $N$ is large:
\begin{equation}\label{iso_law}
\langle \mathbf u, \mathcal G_1(X,z)\mathbf u\rangle \approx m_{1c,\mathbf u}(z),  \ \ \ \langle \mathbf v, \mathcal G_2(X,z)\mathbf v\rangle \approx m_{2c}(z).
\end{equation}
%for $z\in \mathbb C_+$. % with $\Im\, z \ge \eta$ for some fixed $\eta>0$. 
By taking the imaginary part, it is easy to see that a control of the Stieltjes transforms $\langle \mathbf u, \mathcal G_1(X,z)\mathbf u\rangle$ and $\langle \mathbf v, \mathcal G_2(X,z)\mathbf v\rangle$ yields a control of the VESD on the scale of order $\Im\, z$ around $E$. An {\it{anisotropic local law}} is an estimate of the form (\ref{iso_law}) for all $\Im\, z \gg N^{-1}$. Such local law was first established in \cite{isotropic1,isotropic,Anisotropic} for sample covariance matrices, assuming that the matrix entries have arbitrarily high moments. In Section \ref{section_proof}, we will finish the proof of Theorem \ref{main_thm} with the (almost) optimal anisotropic local laws for $\mathcal G_1$ and $\mathcal G_2$.

Our anisotropic local law can be stated in a simple and unified fashion using the following $(N+M)\times (N+M)$ self-adjoint matrix $H$:%, which is a linear function of $X$. %It has been used previously in \cite{NeceSuff_sample,Anisotropic} to prove the local laws of sample covariance matrices.
%We define the following $(N+M)\times (N+M)$ block matrix, which is a linear function of $X$.
%\begin{defn}[Resolvents]\label{def_linearHG}%definiton of the Green function
%We define the $(N+M)\times (N+M)$ matrix
 \begin{equation}\label{linearize_block}
   H : = \left( {\begin{array}{*{20}c}
   { 0 } & \Sigma^{1/2}X  \\
   {(\Sigma^{1/2}X)^*} & {0}  \\
   \end{array}} \right).
 \end{equation}
We define the resolvent of $H$ as
 \begin{equation}\label{eqn_defG}
 G (X,z):= \left( {\begin{array}{*{20}c}
   { - I_{M\times M}} & \Sigma^{1/2}X  \\
   {(\Sigma^{1/2}X)^*} & { - zI_{N\times N}}  \\
\end{array}} \right)^{-1}, \quad z\in \mathbb C_+ . 
\end{equation}
%and the resolvents (or the Green functions) of $XX^*$ and $X^* X$ as
%%We define the Green functions for $XX^*$ and $X^* X$ as
%\begin{equation}\label{def_green}
%\mathcal G_1(X,z):=(XX^*-z)^{-1} , \ \ \ \mathcal G_2 (X,z):=(X^* X-z)^{-1} .
%\end{equation}
%The Stieltjes transform of the ESD of $XX^{*}$ is then given by
%\begin{equation*}
%m_1(X,z):=\int \frac{1}{x-z}dF^{(M)}_{X X^*}(x)=\frac{1}{M} \mathrm{Tr} \, \mathcal G_1(X,z). %\label{EEEE}
%\end{equation*}
%Similarly, we also define $m_2(X,z):= N^{-1}\mathrm{Tr} \, \mathcal G_2(X,z)$. During the proof, we often omit the arguments $X,z$ from our notations.
%\end{defn}
%\begin{rem}
%Since the nonzero eigenvalues of $X^* X$ and $XX^*$ are identical and $XX^*$ has $M-N$ more (or $N-M$ less) zero eigenvalues, we have
%\begin{equation*}
%F^{(M)}_{XX^*}={d_N} {F}^{(N)}_{X^* X}+(1-{d_N})\mathbf{1}_{[0,\infty)}, %\label{def21}
%\end{equation*}
%which implies that (see also (\ref{m2c}))
%\begin{equation}\label{barm}
%m_{2}(z)=\frac{d_N^{-1}-1}{z}+d_N^{-1}m_{1}(z) .%m_1 (z)= - \frac{1-d_N}{z}+d_N m_2 (z).
%\end{equation}
%\end{rem}
Using Schur complement formula, it is easy to check that
\begin{equation} \label{green2}
G = \left( {\begin{array}{*{20}c}
   { z\mathcal G_1} & \mathcal G_1 (\Sigma^{1/2}X)  \\
   {(\Sigma^{1/2}X)^*\mathcal G_1} & { \mathcal G_2 }  \\
\end{array}} \right)= \left( {\begin{array}{*{20}c}
   { z\mathcal G_1} & (\Sigma^{1/2}X)\mathcal G_2   \\
   {\mathcal G_2}(\Sigma^{1/2}X)^* & { \mathcal G_2 }  \\
\end{array}} \right).
\end{equation}
Thus a control of $G$ yields directly a control of the resolvents $\mathcal G_1$ and $\mathcal G_2$. For simplicity of notations, we define the index sets
$$\mathcal I_1:=\{1,...,M\}, \ \ \mathcal I_2:=\{M+1,...,M+N\}, \ \ \mathcal I:=\mathcal I_1\cup\mathcal I_2.$$
We shall consistently use the latin letters $i,j\in\mathcal I_1$, greek letters $\mu,\nu\in\mathcal I_2$, and $a,b\in\mathcal I$. Then we label the indices of $X$ as $X= (X_{i\mu}:i\in \mathcal I_1, \mu \in \mathcal I_2).$
% Moreover, we have
%$$m_1=\frac{1}{Mz}\sum_{i\in \mathcal I_1}G_{ii}, \ \ m_2=\frac{1}{N}\sum_{\mu \in \mathcal I_2}G_{\mu\mu}.$$

%For simplicity of presentation, 
We will use the following notion of stochastic domination, which was first introduced in \cite{Average_fluc} and subsequently used in many works on random matrix theory, such as \cite{isotropic,Anisotropic}. %\cite{isotropic,Principal,Anisotropic}. 
It simplifies the presentation of the results and their proofs by systematizing statements of the form ``$\xi$ is bounded with high probability by $\zeta$ up to a small power of $N$".

\begin{defn}[Stochastic domination]\label{stoch_domination}
%\begin{itemize}
%\item[(i)] 
(i) Let
$$\xi=\left(\xi^{(N)}(u):N\in\mathbb N, u\in U^{(N)}\right),\hskip 10pt \zeta=\left(\zeta^{(N)}(u):N\in\mathbb N, u\in U^{(N)}\right)$$
be two families of nonnegative random variables, where $U^{(N)}$ is a possibly $N$-dependent parameter set. We say $\xi$ is stochastically dominated by $\zeta$, uniformly in $u$, if for any (small) $\epsilon>0$ and (large) $D>0$, 
$$\sup_{u\in U^{(N)}}\mathbb P\left[\xi^{(N)}(u)>N^\epsilon\zeta^{(N)}(u)\right]\le N^{-D}$$
for large enough $N\ge N_0(\epsilon, D)$. %Throughout this paper the stochastic domination will always be uniform in all parameters that are not explicitly fixed (such as matrix indices, deterministic vectors, and spectral parameters $z\in \mathbf D$). Note that $N_0(\epsilon, D)$ may depend on quantities that are explicitly constant, such as $d$, $C_1$ and $\phi$ in Theorem \ref{main_thm}. 
%\item[(ii)] 

(ii) If $\xi$ is stochastically dominated by $\zeta$, uniformly in $u$, we use the notation $\xi\prec\zeta$. Moreover, if for some complex family $\xi$ we have $|\xi|\prec\zeta$, we also write $\xi \prec \zeta$ or $\xi=O_\prec(\zeta)$.

%(ii)
%%\item[(iii)] 
%We also extend the definition of $O_\prec(\cdot)$ to matrices in the weak operator sense as follows. Let $A$ be a family of complex random matrices and $\omega$ a family of nonnegative random variables. Then we use $A\prec \omega$ or $A=O_\prec(\omega)$ to mean $\left|\left\langle\mathbf v, A\mathbf w\right\rangle\right|\prec\omega \| \mathbf v\|_2 \|\mathbf w\|_2$ uniformly for all deterministic vectors $\mathbf v$ and $\mathbf w$.
%\item[(iv)] 

(iii) We say that an event $\Xi$ holds with high probability if for any constant $D>0$, $\mathbb P(\Xi)\ge 1- N^{-D}$ for large enough $N$.%$1-\mathbf 1(\Xi)\prec 0$.
%\end{itemize}
\end{defn}

The following lemma collects basic properties of stochastic domination, which will be used tacitly throughout the proof .

\begin{lem}[Lemma 3.2 in \cite{isotropic}]\label{lem_stodomin}
(i) Let $\xi$ and $\zeta$ be families of nonnegative random variables. Suppose that $\xi (u,v)\prec \zeta(u,v)$ uniformly in $u\in U$ and $v\in V$. If $|V|\le N^C$ for some constant $C$, then $\sum_{v\in V} \xi(u,v) \prec \sum_{v\in V} \zeta(u,v)$ uniformly in $u$.

(ii) If $0\le \xi_1 (u)\prec \zeta_1(u)$ and $0\le \xi_2 (u)\prec \zeta_2(u)$ uniformly in $u\in U$, then $\xi_1(u)\xi_2(u) \prec \zeta_1(u)\zeta_2(u)$ uniformly in $u\in U$.

(iii) Suppose that $\Psi(u)\ge N^{-C}$ is deterministic and $\xi(u)$ satisfies $\mathbb E\xi(u)^2 \le N^C$ for all $u$. Then if $\xi(u)\prec \Psi(u)$ uniformly in $u$, we have $\mathbb E\xi(u) \prec \Psi(u)$ uniformly in $u$.
\end{lem}

%{\color{blue}Do I need to include the large deviation lemma?}

%\begin{rem}\label{rem_highmoment}
%If the entries of $X$ satisfy (\ref{size_condition}), then $X$ trivially satisfies the bounded support condition with $q=N^{-\phi}$. If we assume that $\sqrt{N}X_{i\mu}$ has arbitrarily high moments, i.e. for any $p\in \mathbb N$ there is a constant $C_p$ such that 
%\begin{equation}\label{high_moments}
%\max_{i,\mu}\mathbb E|\sqrt{N}X_{i\mu}|^p \le C_p.
%\end{equation}
%Then by Markov's inequality, $X$ has support $N^{-1/2}$. 
%\end{rem}
Throughout the rest of this paper, we will consistently use the notation $z=E+i\eta$ for the spectral parameter $z$. In the following proof, we always assume that $z$ lies in the spectral domain
\begin{equation}\label{eq_domainD}
\mathbf D(\omega,N) := \{z\in \mathbb C_+:  \omega \le E \le 2\gamma_1, N^{-1+\omega} \le \eta\le \omega^{-1}\},
\end{equation}
for some small constant $\omega>0$, unless otherwise indicated. 
%Note that if $d_N\to d$ for some constant $d\ne 1$, then by (\ref{mp_density}) 
Recall the condition (\ref{regular1}), we can take $\omega$ to be sufficiently small such that $\omega\le \gamma_K/2$. Define the distance to the spectral edges as
%\begin{equation}
$\kappa:= \min_{1\le k \le 2L}\vert E -a_k\vert.$
% \label{KAPPA}
%\end{equation}
%The next lemma gives some basic properties of $m_{2c}$.%, which can be proved through direct calculations using (\ref{m1c}) and (\ref{m2c}).
Then we have the following estimates for $m_{2c}$:
%\begin{lem}[Lemma ]\label{lem_mbehavior}
%and $\delta_N \le (\log N)^{-1}$, 
%we have
\begin{equation}\label{Immc}
\vert m_{2c}(z) \vert \sim 1, \ \  \Im \, m_{2c}(z) \sim \begin{cases}
    {\eta}/{\sqrt{\kappa+\eta}}, & \text{ if } E \notin \text{supp}\, \rho_{2c}\\
    \sqrt{\kappa+\eta}, & \text{ if } E \in \text{supp}\, \rho_{2c}\\
  \end{cases},
\end{equation}
and 
\begin{equation}\label{Piii}
\max_{i\in \mathcal I_1} \vert (1 + m_{2c}(z)\sigma_i)^{-1} \vert = O(1).
\end{equation}
for $z\in \mathbf D$. The reader can refer to \cite[Appendix A]{Anisotropic} for the proof.
%and
%\begin{equation}\label{perturb}
%\vert m_{1,2c}(z) - m_{1,2c}(z + \delta_N) \vert = O\left( \frac{\delta_N}{\sqrt{\kappa+\eta}}\right).
%\end{equation}
%\end{lem}

We define the deterministic limit
\begin{equation}\label{defn_pi}
\Pi (z): = \left( {\begin{array}{*{20}c}
   { -(1+m_{2c}(z)\Sigma)^{-1}} & 0  \\
   0 & {m_{2c}(z)I_{N\times N}}  \\
\end{array}} \right) ,
\end{equation}
and the control parameter
\begin{equation}\label{eq_defpsi}
\Psi (z):= \sqrt {\frac{\Im \, m_{2c}(z)}{{N\eta }} } + \frac{1}{N\eta}.
\end{equation}
Note that by (\ref{Immc}) and (\ref{Piii}), we have for $z\in \mathbf D$,
\begin{equation}\label{psi12}
\|\Pi\|=O(1), \quad \Psi \gtrsim N^{-1/2} , \quad \Psi^2 \lesssim (N\eta)^{-1}.
\end{equation}

\begin{defn}[Bounded support condition] \label{defn_support}
We say a random matrix $X$ satisfies the {\it{bounded support condition}} with $q$, if
\begin{equation}
\max_{i\in \mathcal I_1, \mu \in \mathcal I_2}\vert X_{i\mu}\vert \prec q. \label{eq_support}
\end{equation}
Here $q\equiv q(N)$ is a deterministic parameter and usually satisfies $ N^{-{1}/{2}} \leq q \leq N^{- \phi} $ for some (small) constant $\phi>0$. Whenever (\ref{eq_support}) holds, we say that $X$ has support $q$. Obviously, if the entries of $X$ satisfy (\ref{size_condition}), then $X$ trivially satisfies the bounded support condition with $q=N^{-\phi}$.
\end{defn}

Now we are ready to state the local laws for the resolvent $G(X,z)$. Here and throughout the following, whenever we say ``uniformly in any deterministic vectors", we mean that ``uniformly in any deterministic vectors belonging to some fixed set of cardinality $N^{O(1)}$". %We first state a result under the stronger moment assumption (\ref{high_moments}).

\begin{thm}[Local MP law]\label{lem_EG0}
Suppose $d_N$, $X$ and $\Sigma$ satisfy the Assumption \ref{main_assm}. Suppose $X$ is real and satisfies (\ref{eq_support}) with $q\le N^{-\phi}$ for some constant $\phi>0$.
%Let $X$ be an $M\times N$ real random matrix whose entries are independent random variables satisfying (\ref{entry_assm0}), (\ref{entry_assm1}), (\ref{conditionA3}) and 
%the bounded support condition (\ref{eq_support}) with $q\le N^{-\phi}$ for some constant $\phi>0$.
%and the moment assumption (\ref{high_moments}). Then %for any deterministic unit vector $\mathbf v \in \mathbb C^{\mathcal I_1}$, 
Then the following estimates hold for $z\in \mathbf D$:

(1) the averaged local law:
\begin{equation}\label{aver_law}
| m_2(X,z)-m_{2c}(z) | + \Big|M^{-1}\sum_{i\in \mathcal I_1} \sigma_i (G_{ii} - \Pi_{ii}) \Big|\prec ({N\eta})^{-1};
\end{equation}

(2) the anisotropic local law: for deterministic unit vectors $\mathbf u, \mathbf v \in \mathbb C^{\mathcal I}$,
\begin{equation}\label{aniso_law}
\left| \langle \mathbf u, G(X,z) \mathbf v\rangle - \langle \mathbf u, \Pi (z)\mathbf v\rangle \right| \prec q + \Psi(z);
\end{equation}

(3) for deterministic unit vectors $\mathbf u,\mathbf v \in \mathbb C^{\mathcal I_1}$ or $\mathbf u,\mathbf v \in \mathbb C^{\mathcal I_2}$,
%\begin{equation}\label{ANISO_LAW2}
%%\begin{split}
%\left| \langle \mathbf v, \mathcal G_1(X,z) \mathbf v\rangle - m_{1c,\mathbf v}(z)\rangle \right| + \left| \langle \mathbf w, \mathcal G_2(X,z) \mathbf w\rangle - m_{2c}(z)\rangle \right| 
%\prec q^2 +(N\eta)^{-1/2}.
%%\end{split}
%\end{equation}
\begin{equation}\label{ANISO_LAW2}
\left| \langle \mathbf u, G(X,z) \mathbf v\rangle - \langle \mathbf u, \Pi (z)\mathbf v\rangle \right| \prec q^2 +(N\eta)^{-1/2}.
%\end{split}
\end{equation}
All of the above estimates are uniform in the spectral parameter $z$ and the deterministic vectors $\mathbf u, \mathbf v$.
\end{thm}

The proof for Theorem \ref{lem_EG0} will be given in the Supplementary material \ref{chap_supp}. Here we make some brief comments on it.
%{\color{red}The averaged local law (\ref{aver_law}) has been proved in \cite[Theorem 3.14]{NeceSuff_sample} assuming mean zero and variance $N^{-1}$ instead of (\ref{entry_assm0}) and (\ref{entry_assm1}). However, the proof for the more general case with (\ref{entry_assm0}) and (\ref{entry_assm1}) is essentially the same, and we omit the details.} 
If we assume (\ref{entry_assm}) (instead of (\ref{entry_assm0}) and (\ref{entry_assm1})) and $q=N^{-1/2}$, then (\ref{aver_law}) and (\ref{aniso_law}) have been proved in \cite{Anisotropic}. If we have (\ref{entry_assm}) and $q\le N^{-\phi}$, then it was proved in Lemma 3.11 and Theorem 3.14 of \cite{NeceSuff_sample} that the averaged local law (\ref{aver_law}) and the entrywise local law
\begin{equation}\label{entry_law}
\max_{a,b\in \mathcal I}\left|G_{ab}(X,z) - \Pi_{ab}(z)\right| \prec q+\Psi(z)
\end{equation}
hold uniformly in $z\in \mathbf D$. %Moreover, the averaged local law (\ref{aver_law}) was also proved in \cite[Theorem 3.14]{NeceSuff_sample}. 
With (\ref{entry_law}) and the moment assumption (\ref{conditionA3}), one can repeat the arguments in \cite[Section 5]{isotropic} or \cite[Section 5]{XYY} to get the anisotropic local law (\ref{aniso_law}). The main novelty of this theorem is the bound (\ref{ANISO_LAW2}), which is the main focus in the proof in supplementary material. %The bound (\ref{ANISO_LAW2}) is relatively easier to prove. 
%In fact, if we only consider the upper left and lower right blocks of $G(X,z)$, 
%In the proof, we will first establish the following version of the entrywise local law for the upper left and lower right blocks of $G(X,z)$:
%\begin{equation}\label{entry_law02}
%\max_{r =1,2}\max_{a,b\in \mathcal I_r}\left|G_{ab}(X,z) - \Pi_{ab}(z)\right| \prec q^2+(N\eta)^{-1/2},
%\end{equation}
%which can be proved with the help of (\ref{entry_law}). Then using (\ref{entry_law02}) and (\ref{conditionA3}), we will extend the arguments in \cite[Section 5]{isotropic} to conclude the anisotropic local law (\ref{ANISO_LAW2}). %(see Lemma \ref{ANISO_LEM}). 
Finally, if the variance assumption in (\ref{entry_assm}) is relaxed to the one in (\ref{entry_assm1}), we can repeat the previous arguments to get the desired estimates (\ref{aver_law})-(\ref{ANISO_LAW2}). In fact, it is easy to check that the $O(N^{-2-c_0})$ term leads to a negligible error at each step, and the whole proof remains unchanged. The relaxation of the mean zero assumption in (\ref{entry_assm}) to the assumption (\ref{entry_assm0}) can be handled with the centralization Lemma \ref{comp_claim}.

After taking expectation, we have the following crucial improvement from (\ref{ANISO_LAW2}) to (\ref{Eaniso_law0}), which is the main reason why we can improve the bound in \cite{XYZ2013} to the almost optimal one in (\ref{boundE}). In fact, the leading order terms of $(\langle \mathbf u, \mathcal G_1 \mathbf u\rangle - m_{1c,\mathbf u})$ and $(\langle \mathbf v, \mathcal G_2 \mathbf v\rangle - m_{2c})$ vanish after taking expectation, and hence leads to a bound that is one order smaller than the one in (\ref{ANISO_LAW2}). The proof of Theorem \ref{lem_EG} will be given in Sections \ref{proof_lem_EG}, which constitutes the main novelty of this paper. 

\begin{thm}\label{lem_EG}
%Let $X$ be an $M\times N$ real random matrix whose entries are independent random variables satisfying (\ref{entry_assm0}), (\ref{entry_assm1}), (\ref{conditionA3}) and the bounded support condition (\ref{eq_support}) with $q\le N^{-\phi}$ for some constant $\phi>0$.
%and the moment assumption (\ref{high_moments}). Then %for any deterministic unit vector $\mathbf v \in \mathbb C^{\mathcal I_1}$, 
Suppose the assumptions in Theorem \ref{lem_EG0} hold. Then we have %the following estimates hold for all $z\in \mathbf D$:
%\begin{equation}\label{Eaniso_law0}
%%\begin{split}
%%\left| \mathbb E \langle \mathbf v, \mathcal G_1(X,z) \mathbf v\rangle - m_{2c}(z) \right| 
%\left| \mathbb E\langle \mathbf v, \mathcal G_1(X,z) \mathbf v\rangle - m_{1c,\mathbf v}(z)\rangle \right| + \left| \mathbb E\langle \mathbf w, \mathcal G_2(X,z) \mathbf w\rangle - m_{2c}(z)\rangle \right|
%\prec q^4 + (N\eta)^{-1}
%%\end{split}
%\end{equation}
\begin{equation}\label{Eaniso_law0}
\left| \mathbb E\langle \mathbf u, G(X,z) \mathbf v\rangle - \langle \mathbf u, \Pi (z)\mathbf v\rangle \right| \prec q^4 + (N\eta)^{-1}
%\end{split}
\end{equation}
%All of the above estimates are 
uniformly in $z\in \mathbf D$ and deterministic unit vectors $\mathbf u,\mathbf v \in \mathbb C^{\mathcal I_1}$ or $\mathbf u,\mathbf v \in \mathbb C^{\mathcal I_2}$.
\end{thm}

%By Theorem \ref{lem_EG}, 
If $q=N^{-1/4}$, then (\ref{ANISO_LAW2}) and (\ref{Eaniso_law0}) already give that
\begin{align*}
& \left| \langle \mathbf u, \mathcal G_1\mathbf u\rangle - m_{1c,\mathbf u} \right| + \left| \langle \mathbf v, \mathcal G_2 \mathbf v\rangle - m_{2c} \right|  \prec (N\eta)^{-1/2},\\
%\end{equation*}
%%and
%\begin{equation*}
& \left|\mathbb E \langle \mathbf u, \mathcal G_1\mathbf u\rangle - m_{1c,\mathbf u} \right| + \left| \mathbb E\langle \mathbf v, \mathcal G_2 \mathbf v\rangle - m_{2c} \right|  \prec (N\eta)^{-1},
\end{align*}
which are sufficient to conclude Theorem \ref{main_thm}. However, we find that the second bound on the expected VESD is still valid under a much weaker support assumption. More specifically, we have the following theorem, whose proof will be given in the supplementary material.
%Section \ref{Section_comparison}. 
%The main strategy is a resolvent comparison method that was developed in \cite{LY}.

\begin{thm} \label{thm_large} 
Suppose the assumptions in Theorem \ref{lem_EG0} hold. Then we have
\begin{equation}\label{Eaniso_law}
\left| \mathbb E\langle \mathbf u, G(X,z) \mathbf v\rangle - \langle \mathbf u, \Pi (z)\mathbf v\rangle \right|\prec (N\eta)^{-1},
\end{equation}
uniformly in $z\in \mathbf D$ and deterministic unit vectors $\mathbf u,\mathbf v \in \mathbb C^{\mathcal I_1}$ or $\mathbf u,\mathbf v \in \mathbb C^{\mathcal I_2}$.
\end{thm}

%We define the classical location $\gamma_j$ of the $j$-th eigenvalue of $XX^{*}$ as
%\begin{equation*}
%\int_{\gamma_j}^{+\infty} \rho_{1c}(x)dx=\frac{j}{M}, \ \ 1\le j \le K,
%\end{equation*}
%where $\rho_{1c}$ is defined in (\ref{mp_density}) and $K:=\min\{M,N\}$. 
As a corollary of (\ref{aver_law}), we have the following rigidity result for the eigenvalues. The reader can refer to \cite[Theorem 3.12]{Anisotropic} for the proof. Recall the notations in (\ref{Nk}) and (\ref{gammaj}).%{\color{red}In fact, we can get a stronger result as the one in \cite[Theorem 3.12]{Anisotropic}, which gives the rigidity of all the eigenvalues $\lambda_j$, $1\le j \le K$. However, the estimate (\ref{rigidity2}) is already sufficient for our use.}
%For its proof, one can refer to the arguments in \cite[Section 5]{EYY}, \cite[Section 7]{Semicircle} and \cite[Section 8]{PY}. %In fact, we can almost copy the proof verbatim except for some notation differences.

\begin{thm}[Rigidity of eigenvalues] \label{thm_largerigidity}
Suppose Theorem \ref{lem_EG0} and the regularity condition (\ref{regular1}) hold. 
%$\lambda_- \ge c$ for some constant $c>0$. 
Then for $\gamma_j \in [a_{2k},a_{2k-1}]$, we have
\begin{equation}\label{rigidity2} 
| \lambda_j - \gamma_j| \prec [(N_{2k}+1-j)\wedge (j+1-N_{2k-1})]^{-1/3}N^{-{2}/{3}}.
\end{equation}
%where $\lambda_{1,K}$ are defined in ... and $\lambda_{\pm}$ are defined in...
%\begin{equation}\label{rigidity2} 
%\vert \lambda_j (XX^*)-\gamma_j \vert \prec (\min\{j, K+1-j\})^{-1/3}K^{-{2}/{3}}, \ \ 1 \le j \le K.
%\end{equation}
\end{thm}

\subsection{Convergence rate of the VESD}\label{section_proof}

In this subsection, we finish the proof of Theorem \ref{main_thm} using Theorems \ref{lem_EG0}-\ref{thm_largerigidity}. The following arguments have been used previously to control the Kolmogorov distance between the ESD of a random matrix and the limiting law. For example, the reader can refer to \cite[Lemma 6.1]{Wigner_Bern} and \cite[Lemma 8.1]{PY}. By the remark below (\ref{eq_domainD}), we can choose the constant $\omega>0$ such that $\gamma_K/2>\omega$. Also for simplicity, we will only prove the bounds for $\|\mathbb E F_{Q_2,\mathbf v}-F_{2c}\|$ and $\|F_{Q_2,\mathbf v}-F_{2c}\|$. The bounds for $\|\mathbb E F_{Q_1,\mathbf u}-F_{1c,\mathbf u}\|$ and $\|F_{Q_1,\mathbf u}-F_{1c,\mathbf u}\|$ can be proved in the same way. 

\begin{proof}[Proof of (\ref{boundE})]
The key inputs are the bounds (\ref{Eaniso_law}) and (\ref{rigidity2}). Suppose $\langle \mathbf v, \mathcal G_2( X,z) \mathbf v\rangle$ is the Stieltjes transform of $\hat \rho_{\mathbf v}$. Then we define
\begin{equation}\label{nv}
\hat n_{\mathbf v}(E) :=\int \mathbf 1_{[0,E]}(x) \hat \rho_{\mathbf v} dx, \ \   n_{c}(E):= \int \mathbf 1_{[0,E]}(x) \rho_{2c}dx,
\end{equation}
and $\rho_{\mathbf v}:= \mathbb E \hat \rho_{\mathbf v}$, $n_{\mathbf v} : = \mathbb E\hat n_{\mathbf v}$. Hence we would like to bound
%$$\| F^{(M)}_{\tilde X\tilde X^*}(\mathbf v, x) - F_{d_N}(x)\| = \sup_E \left| \hat n_{\mathbf v}(E) - n_{c}(E)\right|$$
%and 
$$\|\mathbb E F_{ Q_2,\mathbf v} - F_{2c} \| = \sup_E \left| n_{\mathbf v}(E) - n_{c}(E)\right|.$$
%where we define
%$$n_{\mathbf v}(E) :=\int 1_{[0,E]}(x) \rho_{\mathbf v} (x)dx, \ \ n_{c}(E):= \int1_{[0,E]}(x) \rho_{1c}(x)dx .$$
For simplicity, we denote $\Delta \rho:=\rho_{\mathbf v} - \rho_{2c}$ and its Stieltjes transform by
$$\Delta m (z):= \mathbb E \langle \mathbf v, \mathcal G_2(X,z) \mathbf v\rangle - m_{2c}(z).$$ 

Let $\chi(y)$ be a smooth cutoff function with support in $[-1,1]$, with $\chi(y)=1$ for $|y| \le 1/2$ and with bounded derivatives. Fix $\eta_0 = N^{-1+\omega}$ and $3\gamma_K/4 \le E_1 < E_2 \le 3\gamma_1/2 $. Let $f\equiv f_{E_1,E_2,\eta_0}$ be a smooth function supported in $[E_1-\eta_0,E_2+\eta_0]$ such that $f(x) = 1$ if $x\in [E_1+\eta_0,E_2-\eta_0]$, and $|f'|\le C\eta_0^{-1}$, $|f''|\le C\eta_0^{-2}$ if $|x-E_i|\le \eta_0$. Using the Helffer-Sj{\"o}strand calculus (see e.g. \cite{functional_calc}), we have
$$f(E)=\frac{1}{2\pi}\int_{\mathbb R^2} \frac{iy f''(x)\chi(y) + i(f(x)+iyf'(x))\chi'(y)}{E-x-iy}dxdy.$$
Then we obtain that 
\begin{align}
& \left| \int f(E) \Delta \rho(E)dE\right| \nonumber\\ 
& \le C\int_{\mathbb R^2} \left( |f(x)|+|y||f'(x)|\right)|\chi'(y)||\Delta m(x+iy)|dxdy \label{term1}\\
& + C\sum_i \left| \int_{|y|\le \eta_0} \int_{|x-E_i|\le \eta_0} yf''(x)\chi(y)\Im\, \Delta m(x+iy)dxdy\right| \label{term2}\\
& + C\sum_i \left| \int_{|y|\ge \eta_0}\int_{|x-E_i|\le \eta_0} yf''(x)\chi(y)\Im\, \Delta m(x+iy)dxdy\right|. \label{term3}
\end{align}
By (\ref{Eaniso_law}) with $\eta = \eta_0$, we have
\begin{equation}\label{upper}
\eta_0 \Im \, \mathbb E \langle \mathbf v, \mathcal G_2(X,E+i\eta_0) \mathbf v\rangle \prec N^{-1+\omega}.
\end{equation}
Since $\eta \Im \, \mathbb E \langle \mathbf v, \mathcal G_2(X,E+i\eta) \mathbf v\rangle$ and $\eta \Im \, m_{2c}(E+i\eta)$ are increasing with $\eta$, we obtain that
\begin{equation}\label{lower}
\eta\left| \Im \, \Delta m(E+i\eta)\right|\prec N^{-1+\omega} \ \ \text{ for all } 0\le \eta \le \eta_0.
\end{equation}
Moreover, since $G(X,z)^* = G(X, \bar z)$, the estimates (\ref{Eaniso_law}) and (\ref{lower}) also hold for $z\in \mathbb C_-$.
%Here we also used that the estimate (\ref{EANISO_LAW2}) for $z\in \mathbb C_-$ with $\bar z\in \mathbf D$. Hence we can obtain a similar bound for the integral with $-\eta_0 \le y \le 0$. By symmetry, 

Now we bound the terms (\ref{term1}), (\ref{term2}) and (\ref{term3}). Using (\ref{Eaniso_law}) and that the support of $\chi'$ is in $1\ge |y| \ge 1/2$, the term (\ref{term1}) can be bounded by
\begin{align}\label{term1b}
\int_{\mathbb R^2} \left( |f(x)|+|y||f'(x)|\right)|\chi'(y)||\Delta m(x+iy)|dxdy \prec N^{-1}.
\end{align}
Using $|f''| \le C\eta_0^{-2}$ and (\ref{lower}), we can bound the terms in (\ref{term2}) by
\begin{align}\label{term2b}
\left| \int_{|y| \le \eta_0} \int_{|x-E_i|\le \eta_0} yf''(x)\chi(y)\Im\, \Delta m(x+iy)dxdy\right|  \prec N^{-1+\omega}.
\end{align}
Finally, we integrate the term (\ref{term3}) by parts first in $x$, and then in $y$ (and use the Cauchy-Riemann equation $\partial \Im(\Delta m)/\partial x = - \partial \Re (\Delta m)/\partial y$) to get %that
\begin{align}
& \int_{y\ge \eta_0}\int_{|x-E_i|\le \eta_0} yf''(x)\chi(y)\Im\, \Delta m(x+iy)dxdy \nonumber\\
%&= \int_{y\ge \eta_0}\int_{|x-E_i|\le \eta_0} yf'(x)\chi(y)\frac{\partial \Re\, \Delta m(x+iy)}{\partial y}dxdy \nonumber\\
&=  -\int_{|x-E_i|\le \eta_0} \eta_0 \chi(\eta_0) f'(x)\Re\, \Delta m(x+i\eta_0) dx \label{term4}\\
& \quad \, - \int_{y\ge \eta_0}\int_{|x-E_i|\le \eta_0} \left(y\chi'(y) +\chi(y) \right)f'(x)  \Re\, \Delta m(x+iy) dxdy\label{term5} .
\end{align}
We bound the term in (\ref{term4}) by $O_\prec(N^{-1})$ using (\ref{Eaniso_law}) and $|f'|\le C\eta_0^{-1}$. The first term in (\ref{term5}) can be estimated by $O_\prec(N^{-1})$ as in (\ref{term1b}). For the second term in (\ref{term5}), we again use (\ref{Eaniso_law}) and $|f'|\le C\eta_0^{-1}$ to get that
\begin{align*}
\left|\int_{y\ge \eta_0}\int_{|x-E_i|\le \eta_0} \chi(y) f'(x)  \Re\, \Delta m(x+iy) dxdy\right| \prec \int_{\eta_0}^1 \frac{1}{Ny}dy \prec N^{-1}.
\end{align*}
Combining the above estimates, we obtain that
\begin{equation*}
\left| \int_{y \ge \eta_0}\int_{|x-E_i|\le \eta_0} yf''(x)\chi(y)\Im\, \Delta m(x+iy)dxdy\right| \prec N^{-1}. %\label{term3b}
\end{equation*}
Obviously, the same estimate also holds for the $y\le -\eta_0$ part. Together with (\ref{term1b}) and (\ref{term2b}), we conclude that
\begin{align}\label{estimatef}
\left| \int f(E) \Delta \rho(E)dE\right| \prec N^{-1+\omega}.
\end{align}

For any interval $I:=[E-\eta_0,E+\eta_0]$ with $E \in [\gamma_K/2, 2\gamma_1]$, we have
\begin{equation}\label{estimaten}
\begin{split}
& \hat n_{\mathbf v}(E+\eta_0) - \hat n_{\mathbf v} (E-\eta_0) = \sum\limits_{\lambda_k \in (E-\eta_0,E+\eta_0]} |\langle\zeta_k,\mathbf v \rangle|^2 \\
& \le 2\eta_0 \sum\limits_{k = 1}^{N} \frac{|\langle\zeta_k,\mathbf v \rangle|^2 \eta_0}{(\lambda_k-E)^2+\eta_0^2} = 2\eta_0  \Im\,  \langle \mathbf v, \mathcal G_2(X,E+i\eta_0) \mathbf v\rangle, %\prec N^{-1+\omega},
\end{split}
\end{equation}
%\begin{equation}\label{estimaten}
%\begin{split}
%&n_{\mathbf v}(E+\eta_0) - n_{\mathbf v} (E-\eta_0) = \mathbb E\sum\limits_{\lambda_k \in (E-\eta_0,E+\eta_0]} |\langle\xi_k,\mathbf v \rangle|^2 \\
%& \le 2\eta_0 \mathbb E\sum\limits_{k = 1}^{M} \frac{|\langle\xi_k,\mathbf v \rangle|^2 \eta_0}{(\lambda_k-E)^2+\eta_0^2} = 2\eta_0  \Im\, \mathbb E \langle \mathbf v, \mathcal G_1(\tilde X,E+i\eta_0) \mathbf v\rangle \\
%& \prec N^{-1+\omega},
%\end{split}
%\end{equation}
where in the last step we used the spectral decomposition 
$$\mathcal G_2(X,E+i\eta)=\sum_{k=1}^N \frac{\zeta_k \zeta_k^*}{\lambda_k - E - i\eta},$$
which follows from (\ref{SVD_X}). Then by (\ref{upper}) and Lemma \ref{lem_stodomin}, we get that
\begin{equation}\label{estimaten0}
n_{\mathbf v}(E+\eta_0) - n_{\mathbf v} (E-\eta_0) \prec N^{-1+\omega}.
\end{equation}
On the other hand, since $\rho_{2c}$ is bounded, we trivially have 
\begin{equation}\label{estimaten2}
n_c(E+\eta_0)-n_c(E-\eta_0)\le C\eta_0 = CN^{-1+\omega}.
\end{equation}

Now we set $E_2=3\gamma_1/2$. With (\ref{estimatef}), (\ref{estimaten0}) and (\ref{estimaten2}), we get that for any $E\in [3\gamma_K/4,E_2]$,
\begin{equation}\label{pointE}
\left| \left(n_{\mathbf v}(E_2) - n_{\mathbf v}(E)\right) - \left(n_{c}(E_2)-n_{c}(E)\right) \right| \prec N^{-1+\omega} .
\end{equation}
Note that by (\ref{rigidity2}), the eigenvalues of $Q_2$ are inside $\{0\}\cup [3\gamma_K/4,E_2]$ with high probability. Hence we have that with high probability, 
\begin{equation}\label{nhat_small}
\hat n_{\mathbf v}(E_2) = n_c(E_2)=1, \ \  \hat n_{\mathbf v}(3\gamma_K/4)=\hat n_{\mathbf v}(0).
\end{equation}
Together with (\ref{pointE}), we get that
\begin{equation}\label{supE}
\sup_{E\ge 0} \left| n_{\mathbf v}(E) - n_{c}(E)\right| \prec N^{-1+\omega}.
\end{equation}
This concludes (\ref{boundE}) since $\omega$ can be arbitrarily small. 
\end{proof}

\begin{proof}[Proof of (\ref{boundp})]
The proof for (\ref{boundp}) is similar except that we shall use the estimate (\ref{ANISO_LAW2}) instead of (\ref{Eaniso_law}). By (\ref{ANISO_LAW2}), we have for any $\mathbf v \in \mathbb C^{\mathcal I_2}$,
\begin{equation}\label{EANISO_LAW2tilde}
\left| \langle \mathbf v, \mathcal G_2(X,z) \mathbf v\rangle - m_{2c}(z) \right| \prec N^{-2\phi} + (N\eta)^{-1/2} 
\end{equation}
uniformly in $z\in \mathbf D$. Then we would like to bound (recall (\ref{nv}))
$$\| F^{(M)}_{Q_2,\mathbf v} - F_{2c}\| = \sup_E \left| \hat n_{\mathbf v}(E) - n_{c}(E)\right|,$$
where $\hat n_{\mathbf v}$ is defined in (\ref{nv}). We denote 
$$\Delta \hat \rho:=\hat \rho_{\mathbf v} - \rho_{1c}, \ \ \Delta \hat m:= \langle \mathbf v, \mathcal G_2(X,z) \mathbf v\rangle - m_{2c}(z).$$
Then for $f_{E_1,E_2,\eta_0}$ defined above, we can repeat the Helffer-Sj{\"o}strand argument with the estimate (\ref{EANISO_LAW2tilde}) to get that
\begin{align}
\sup_{E_1 , E_2} \left| \int f_{E_1,E_2,\eta_0}(E) \Delta \hat \rho(E)dE\right| \prec N^{-2\phi}+N^{-1/2},
\end{align}
which, together with (\ref{estimaten}) and (\ref{nhat_small}), implies that
\begin{equation*}
\sup_{E\ge 0} \left| \hat n_{\mathbf v}(E) - n_{c}(E)\right| \prec N^{-2\phi}+N^{-1/2}.
\end{equation*}
This concludes (\ref{boundp}) by the Definition \ref{stoch_domination}.
\end{proof}

\section{Proof of Theorem \ref{lem_EG}}\label{proof_lem_EG}

%\subsection{Resolvent estimates}

%In this subsection, 

We first collect some useful identities from linear algebra and some simple resolvent estimates. For simplicity, we denote $Y:=\Sigma^{1/2}X$.%that follow from Theorem \ref{lem_EG}.

\begin{defn}[Minors]
For $\mathbb T \subseteq \mathcal I$, we define the minor $H^{(\mathbb T)}:=(H_{ab}:a,b \in \mathcal I\setminus \mathbb T)$ obtained by removing all rows and columns of $H$ indexed by $a,b\in \mathbb T$. Note that we keep the names of indices when defining $H^{(\mathbb T)}$, i.e. $(H^{(\mathbb{T})})_{ab}=\mathbf{1}_{ \{a,b \notin \mathbb{{T}}\}} H_{ab}$. Correspondingly, we define the Green function 
$$G^{(\mathbb T)}:=(H^{(\mathbb T)})^{-1}= \left( {\begin{array}{*{20}c}
   { z\mathcal G_1^{(\mathbb T)}} & \mathcal G_1^{(\mathbb T)} Y^{(\mathbb T)}  \\
   {(Y^{(\mathbb T)})^*\mathcal G_1^{(\mathbb T)}} & { \mathcal G_2^{(\mathbb T)} }  \\
\end{array}} \right)= \left( {\begin{array}{*{20}c}
   { z\mathcal G_1^{(\mathbb T)}} & Y^{(\mathbb T)}\mathcal G_2^{(\mathbb T)}   \\
   {\mathcal G_2^{(\mathbb T)}}(Y^{(\mathbb T)})^* & { \mathcal G_2^{(\mathbb T)} }  \\
\end{array}} \right),$$
and the partial traces
$$m_1^{(\mathbb T)}:=\frac{1}{M}{\rm{Tr}}\, \mathcal G_1^{(\mathbb T)} = \frac{1}{Mz}\sum_{i\in \mathcal I_1}G_{ii}^{(\mathbb T)},\ \ m_2^{(\mathbb T)}:=\frac{1}{N}{\rm{Tr}}\, \mathcal G_2^{(\mathbb T)} = \frac{1}{N}\sum_{\mu\in \mathcal I_2}G_{\mu\mu}^{(\mathbb T)},$$
where we adopt the convention that $G^{(\mathbb T)}_{ab}=0$ if $a\in \mathbb T$ or $b\in \mathbb T$. For simplicity, we will abbreviate $(\{a\})\equiv (a)$ and $(\{a, b\})\equiv (ab)$.
%and $$\sum_{a\notin (\mathbb T)} \equiv \sum_{a}^{(\mathbb T)} \ , \ \ \sum_{a,b\notin (\mathbb T)} \equiv \sum_{a,b}^{(\mathbb T)}\ .$$
\end{defn}

\begin{lem}[Resolvent identities]
\begin{itemize}

 \item[(i)]
 For $a \in \mathcal I$ and $b, c \in \mathcal I \setminus \{a\}$,
\begin{equation}
G_{bc}=G_{bc}^{\left( a \right)} + \frac{G_{ba} G_{ac}}{G_{aa}}, \ \ \frac{1}{{G_{bb} }} = \frac{1}{{G_{bb}^{(a)} }} - \frac{{G_{ba} G_{ab} }}{{G_{bb} G_{bb}^{(a)} G_{aa} }}. \label{resolvent8}
\end{equation}

\item[(ii)]
For $i\in \mathcal I_1$ and $\mu\in \mathcal I_2$, we have
\begin{equation}
\frac{1}{{G_{ii} }} =  - 1 - \left( {YG^{\left( i \right)} Y^*} \right)_{ii} ,\ \ \frac{1}{{G_{\mu \mu } }} =  - z  - \left( {Y^*  G^{\left( \mu  \right)} Y} \right)_{\mu \mu }.\label{resolvent2}
\end{equation}

 \item[(iii)]
 For $i\ne j \in \mathcal I_1$ and $\mu \ne \nu \in \mathcal I_2$, we have
%\begin{equation}
%G_{ij}   = G_{ii} G_{jj}^{\left( i \right)} \left( {YG^{\left( {ij} \right)} Y^* } \right)_{ij}, \label{resolvent3}
%\end{equation}
%and 
\begin{equation}
G_{ij}   = G_{ii} G_{jj}^{\left( i \right)} \left( {YG^{\left( {ij} \right)} Y^* } \right)_{ij},  \ \  G_{\mu \nu }  = G_{\mu \mu } G_{\nu \nu }^{\left( \mu  \right)} \left( {Y^*  G^{\left( {\mu \nu } \right)} Y} \right)_{\mu \nu }.\label{resolvent4}
\end{equation}
%For $i\in \mathcal I_1$ and $\mu\in \mathcal I_2$, we have
%%\begin{align}
%%& G_{i\mu }  =  - w^{1/2} G_{ii} \left( {YG^{\left( i \right)} } \right)_{i\mu }  =  - w^{1/2} G_{\mu \mu } \left( {G^{\left( \mu  \right)} Y} \right)_{i\mu } ,\label{resolvent5_1}\\
%%& G_{\mu i}  =  - w^{1/2} G_{ii} \left( {G^{\left( i \right)} Y^\dag  } \right)_{\mu i}  =  - w^{1/2} G_{\mu \mu } \left( {Y^\dag  G^{\left( \mu  \right)} } \right)_{\mu i}.\label{resolvent5_2}
%%\end{align}
%%Furthermore we have
%\begin{equation}
%G_{i\mu } = G_{ii} G_{\mu \mu }^{\left( i \right)} \left( { - X_{i\mu }  +  {\left( {XG^{\left( {i\mu } \right)} X} \right)_{i\mu } } } \right), \ \ G_{\mu i}  = G_{\mu \mu } G_{ii}^{\left( \mu  \right)} \left( { - X_{\mu i}^*  + \left( {X^*  G^{\left( {\mu i} \right)} X^*  } \right)_{\mu i} } \right)\label{resolvent6}.
%\end{equation}

\item[(iv)]
All of the above identities hold for $G^{(\mathbb T)}$ instead of $G$ for $\mathbb T \subset \mathcal I$.
\end{itemize}
\label{lemm_resolvent}
\end{lem}
\begin{proof}
These identities can be proved using Schur complement formula. The reader can refer to e.g. \cite[Lemmas 3.6 and 3.8]{isotropic} or \cite[Lemma 4.4]{Anisotropic}.
\end{proof}

\begin{lem}
%Suppose (\ref{aniso_lawweak}) holds. 
Suppose $\tilde \Phi(z)$ is a deterministic function on $\mathbf D$ satisfying $N^{-1/2} \le \tilde \Phi(z) \le N^{-c}$ for some constant $c>0$. Suppose $\left|  G_{ab}(z) - \Pi_{ab} (z) \right| \prec \tilde \Phi(z)$ uniformly in $a,b\in \mathcal I$ and $z\in \mathbf D$. %Fix an $l\in \mathbb N$. 
Then for any $\mathbb T \subseteq \mathcal I$ with $|\mathbb T|=O(1)$, we have uniformly in $z\in \mathbf D$,
\begin{equation}
\max_{a,b \in \mathcal I \setminus \mathbb T}\left| {G_{ab} (z) - G_{ab}^{\left( \mathbb T \right)} } (z)\right| \prec \tilde \Phi^2 (z). \label{G_T}
\end{equation}
%and 
%\begin{equation}\label{m_T}
%\left|m_2(z) - m_2^{\left( \mathbb T \right)} (z)\right| \prec \tilde \Phi^2(z),
%\end{equation}
%uniformly in $z\in \mathbf D$.
\end{lem}
\begin{proof}
The bound (\ref{G_T}) can be proved by repeatedly applying the first resolvent expansion in (\ref{resolvent8}) with respect to the indices in $\mathbb T$. %The bound (\ref{m_T}) is a trivial consequence of (\ref{G_T}).
\end{proof}

For $X$ satisfying the assumptions in Theorem \ref{lem_EG0}, we write $X=X_1 + B,$ where $X_1:= X-\mathbb EX$ is a real random matrix satisfying (\ref{entry_assm1}), (\ref{conditionA3}) and
\begin{equation}\label{entry_assmX1}
\mathbb E (X_1)_{i\mu} = 0, \quad i\in \mathcal I_1,\ \mu\in \mathcal I_2,
\end{equation}
and $B:=\mathbb EX$ is a deterministic matrix such that 
\begin{equation}\label{boundB}
\max_{i,\mu}|B_{i\mu}|\le C_0 N^{-2-c_0}.
\end{equation}
The next lemma shows that $G(X,z)$ is very close to $G(X_1,z)$ in the sense of anisotropic local law. Its proof will be given in the supplementary material.%and hence we can assume that the entries of $X$ are centralized for the rest of the proof.

\begin{lem}\label{comp_claim} 
If (\ref{aniso_law}) holds for $G(X_1,z)$, then we have 
\begin{equation}\label{aniso_central}
\left| \langle \mathbf u, G(X,z) \mathbf v\rangle - \langle \mathbf u, G(X_1,z) \mathbf v\rangle \right| \prec (N\eta)^{-1}
\end{equation}
uniformly in $z\in \mathbf D$ and deterministic unit vectors $\mathbf u, \mathbf v \in \mathbb C^{\mathcal I}$.
%it also holds for $X$. 
\end{lem}

\subsection{Sketch of the proof for Theorem \ref{lem_EG}}\label{sketch}

In this subsection, we start proving our main resolvent estimate (\ref{Eaniso_law0}). For simplicity, we denote $\Phi:= q^2 + (N\eta)^{-1/2}$. By Lemma \ref{comp_claim}, we can assume that the entries of $X$ are centered without loss of generality. %we can assume that the entries of $X$ are centered for the rest of the proof. 
We will only prove (\ref{Eaniso_law0}) for $\mathbf u,\mathbf v \in \mathbb C^{\mathcal I_2}$, while the proof in the case of $\mathbf u,\mathbf v \in \mathbb C^{\mathcal I_1}$ is exactly the same. Also by polarization, it suffices to prove the following estimate
\begin{equation}\label{Eaniso_law1}
\left| \mathbb E\langle \mathbf v, \mathcal G_2(X,z) \mathbf v\rangle - m_{2c}(z) \right| \prec q^4 + (N\eta)^{-1}, \ \ \mathbf v\in \mathbb C^{\mathcal I_2}.
\end{equation}
We can obtain the more general bound (\ref{Eaniso_law0}) by applying (\ref{Eaniso_law1}) to the vectors $\mathbf u + \mathbf v$ and $\mathbf u + i\mathbf v$, respectively. 
%... We want to estimate $|\mathbb E \langle \mathbf v, \mathcal G_1 \mathbf v\rangle - m_{2c}|$ for any deterministic unit vector $\mathbf v \in \mathbb C^{\mathcal I_1}$. 
Note that (\ref{ANISO_LAW2}) gives the a priori bound
$$\Big|\sum_{\mu, \nu} \bar v_\mu v_\nu \mathbb E \left(\mathcal G_2\right)_{\mu\nu} - m_{2c}\Big| \prec \Phi .$$
We will show that after taking expectation, the leading order term in $\left(\mathcal G_2\right)_{\mu\nu} - m_{2c}\delta_{\mu\nu}$ vanishes and leads to the better estimate (\ref{Eaniso_law1}). We deal with the diagonal and off-diagonal parts separately:
\begin{align*}
\sum_{\mu}|v_\mu|^2 \left[\mathbb E(\mathcal G_2)_{\mu\mu} - m_{2c}\right], \quad \sum_{\mu\ne \nu} \bar v_\mu v_\nu \mathbb E \left(\mathcal G_2\right)_{\mu\nu} .
\end{align*}
 
%, we can verify that
%\begin{equation}\label{Zi}
%Z_i =m_2 - \left( {XG^{\left( i \right)} X^*} \right)_{ii},\ \ Z_\mu = d_N^{-1}m_1 - \left( {X^*  G^{\left( \mu  \right)} X} \right)_{\mu \mu }.
%\end{equation}
%By (\ref{aniso_law}) and (\ref{Immc}), we have that
%\begin{equation}\label{bound_Z}
%  Z_{i} :=(1-\mathbb E_{i}) \left(G_{ii}^{-1} - \Pi_{ii}^{-1}\right) \prec \Psi.
%\end{equation}
%where we used (\ref{Immc}). 

For any $\mathbb T\subseteq \mathcal I$, we define the $Z$ variables
\begin{equation}\label{Z_variable}
  Z^{(\mathbb T)}_{\mu} :=(1-\mathbb E_{\mu}) (G^{(\mathbb T)})_{\mu\mu}^{-1} = \frac{1}{N}\sum_{i\in \mathcal I_1}\sigma_i G_{ii}^{(\mathbb T \mu)} - ( {Y^*G^{\left( \mathbb T \mu \right)} Y} )_{\mu\mu}, \ \ \mu\notin \mathbb T,
\end{equation}
where $\mathbb E_{\mu}[\cdot]:=\mathbb E[\cdot| H^{(\mu)}],$ i.e. it is the partial expectation in the randomness of the $\mu$-th row and column of $H$, and we used (\ref{resolvent2}) in the second step. If $\mathbb T=\emptyset$, we shall abbreviate $ Z_{i} \equiv  Z^{(\emptyset)}_{i} $. Note that by (\ref{ANISO_LAW2}), (\ref{G_T}) (with $\tilde \Phi= q+\Psi$ by (\ref{aniso_law})), and Lemma \ref{lem_stodomin}, we have
\begin{equation}\label{Z_bound}
Z^{(\mathbb T)}_{\mu} :=(1-\mathbb E_{\mu}) \left[(G^{(\mathbb T)})_{\mu\mu}^{-1} - m_{2c}^{-1}\right] \prec \Phi,
\end{equation}
for any $\mathbb T \subseteq \mathcal I$ with $|\mathbb T|=O(1)$. Then using (\ref{resolvent2}) we get that 
\begin{align*}
\mathbb EG_{\mu\mu} - m_{2c}  & = \mathbb E\frac{1}{-z- N^{-1}\sum_i \sigma_i \Pi_{ii} - N^{-1}\sum_i \sigma_i (G_{ii}^{(\mu)}-\Pi_{ii}) + Z_\mu} - m_{2c}\\
&= - m_{2c}^2\mathbb E Z_\mu + O_\prec\left(\Phi^2 + (N\eta)^{-1}\right) = O_\prec\left(\Phi^2\right),
\end{align*}
%\begin{align*}
%&\mathbb EG_{\mu\mu} - m_{2c}  = \mathbb E\frac{1}{-z- m_{2c} - (m_2^{(i)}-m_{2c}) + Z_i} - zm_{2c}\\
%&=\frac{1}{-1-m_{2c}}- zm_{2c} - \frac{1}{(1+m_{2c})^2}\mathbb E Z_{i} + O_\prec\left(\Phi^2 + \frac{1}{N\eta}\right) = O_\prec\left(\Phi^2\right),
%\end{align*}
where in the second step we used (\ref{aver_law}), (\ref{G_T}), (\ref{Z_bound}), and 
\begin{equation}\label{self_use}
-z- N^{-1}\sum_i \sigma_i \Pi_{ii} = m_{2c}^{-1},
\end{equation}
which follows from (\ref{defn_pi}) and (\ref{deformed_MP21}). %, and in the third step we used $\mathbb E Z_\mu = 0$. 
So we can bound the diagonal part by
\begin{equation}\label{diagonal_iso}
\sum_{\mu}|v_\mu|^2 \left[\mathbb E(\mathcal G_2)_{\mu\mu} - m_{2c}(z)\right]= \sum_{\mu}|v_\mu|^2 [\mathbb EG_{\mu\mu} - m_{2c}(z)] \prec q^4+\frac{1}{N\eta}.
\end{equation}
%for $z\in \mathbf D$ (recall that $|z|\ge E \sim 1$ by (\ref{eq_domainD})).

For the off-diagonal part, we claim that for $\mu\ne \nu \in \mathcal I_2$,
\begin{equation}\label{off_small}
\left|\mathbb E \left(\mathcal G_2\right)_{\mu\nu}\right| \prec N^{-1}\Phi^2.
\end{equation}
Then using (\ref{off_small}) and $\|\mathbf v\|_1 \le \sqrt{N}$, we obtain that
$$\Big|\sum_{\mu\ne \nu} \bar v_\mu v_\nu \mathbb E \left(\mathcal G_2\right)_{\mu\nu}\Big| \prec \|\mathbf v\|_1^2N^{-1}\Phi^2 \le C\left(q^4+\frac{1}{N\eta}\right).$$
This concludes (\ref{Eaniso_law1}) together with (\ref{diagonal_iso}).

To prove (\ref{off_small}), we extend the arguments in \cite[Section 5]{isotropic} and \cite[Section 5]{XYY}. We illustrate the basic idea with some simplified calculations. Using the resolvent identities (\ref{resolvent4}) and (\ref{resolvent8}), we get
\begin{align}
\mathbb E G_{\mu\nu} &= \mathbb E G_{\mu\mu}G_{\nu\nu}^{(\mu)}\left( Y^* G^{(\mu\nu)} Y\right)_{\mu\nu} \nonumber\\
&= \mathbb E G^{(\nu)}_{\mu\mu}G_{\nu\nu}^{(\mu)}\left( Y^*G^{(\mu\nu)}Y\right)_{\mu\nu} + \mathbb E \frac{G_{\mu\nu}G_{\nu\mu}}{G_{\nu\nu}}G_{\nu\nu}^{(\mu)}\left( Y^* G^{(\mu\nu)} Y\right)_{\mu\nu} . \label{Gexpansion1}
\end{align}
%\begin{align}
%\mathbb E G_{ij} &= \mathbb E G_{ii}G_{jj}^{(i)}\left( XG^{(ij)}X^*\right)_{ij} \nonumber\\
%&= \mathbb E G^{(j)}_{ii}G_{jj}^{(i)}\left( XG^{(ij)}X^*\right)_{ij} + \mathbb E \frac{G_{ij}G_{ji}}{G_{jj}}G_{jj}^{(i)}\left( XG^{(ij)}X^*\right)_{ij}. \label{Gexpansion1}
%\end{align}
We now focus on the first term. Applying (\ref{resolvent2}) gives that
\begin{align}
\mathbb E G^{(\nu)}_{\mu\mu}G_{\nu\nu}^{(\mu)}\left( Y^*G^{(\mu\nu)}Y\right)_{\mu\nu}  & = \mathbb E \frac{\left( Y^*G^{(\mu\nu)}Y\right)_{\mu\nu} }{\left[-z-(Y^* G^{(\mu\nu)}Y)_{\mu\mu}\right]\left[-z-(XG^{(\mu\nu)}X^*)_{\nu\nu}\right]} \nonumber\\
& = \mathbb E \frac{\left( Y^*G^{(\mu\nu)}Y\right)_{\mu\nu}}{\left(m_{2c}^{-1} + \epsilon_\mu\right)\left(m_{2c}^{-1} + \epsilon_\nu \right)}. \label{expansion1}
\end{align}
%\begin{align}
%\mathbb E G^{(j)}_{ii}G_{jj}^{(i)}\left( XG^{(ij)}X^*\right)_{ij} & = \mathbb E \frac{\left( XG^{(ij)}X^*\right)_{ij}}{\left[1+(XG^{(ij)}X^*)_{ii}\right]\left[1+(Y^*G^{(\mu\nu)}Y)_{\nu\nu}\right]} \nonumber\\
%& = \mathbb E \frac{\left( XG^{(ij)}X^*\right)_{ij}}{\left[(1+m_{2c}) - \epsilon_i\right]\left[(1+m_{2c}) - \epsilon_j \right]}. \label{expansion1}
%\end{align}
where we have %$|({1+m_{2c}})^{-1}|= |zm_{2c}| \sim 1$ and 
\begin{equation}\label{bound_epsilon}
\epsilon_\mu :=  \frac{1}{N}\sum_{i\in \mathcal I_1}\sigma_i \Pi_{ii} - (Y^*G^{(\mu\nu)}Y)_{\mu\mu} = \frac{1}{N}\sum_{i\in \mathcal I_1}\sigma_i (\Pi_{ii}-G_{ii}^{(\mu\nu)}) + Z_\mu^{(\nu)} \prec \Phi
\end{equation} 
by (\ref{self_use}), (\ref{aver_law}), (\ref{G_T}) (with $\tilde \Phi=q+\Psi$) and (\ref{Z_bound}). We now expand the fractions in (\ref{expansion1}) in order to take the expectation. Note that the $G^{(\mu\nu)}$ entries are independent of the $X$ entries in the $\mu,\nu$-th rows and columns. Thus to attain a nonzero expectation, each $X$ entry must appear at least twice in the expression. Due to this reason, the leading and next-to-leading order terms in the expansion vanish. The ``real" leading order term is 
\begin{align}
\mathbb E {m_{2c}^4} \epsilon_\mu \epsilon_\nu \left( Y^*G^{(\mu\nu)}Y\right)_{\mu\nu} & ={m_{2c}^4}\mathbb E (Y^*G^{(\mu\nu)}Y)_{\mu\mu}(Y^*G^{(\mu\nu)}Y)_{\nu\nu} ( Y^*G^{(\mu\nu)}Y)_{\mu\nu} \nonumber\\
&= {m_{2c}^4}\sum_{\mu, \nu}\frac{C_{i,j}}{N^3} \mathbb E G^{(\mu\nu)}_{ii}G^{(\mu\nu)}_{jj} G^{(\mu\nu)}_{ij}  \nonumber\\
&= {m_{2c}^4} \sum_{i\ne j}\frac{C_{i,j}}{N^3} \Pi_{ii}\Pi_{jj} \mathbb E G^{(\mu\nu)}_{ij} + O_\prec(N^{-1}\Phi^2) ,\label{bound_EG1}
\end{align}
where the constants $C_{i,j}$ depend on $\sigma_i$, $\sigma_j$ and the 3rd moments of $X_{i\mu}$ and $X_{j\mu}$ (recall (\ref{conditionA3})). Here in the last step, we used $|G^{(\mu\nu)}_{ii}-\Pi_{ii}| \prec \Phi$ (by (\ref{ANISO_LAW2}) and (\ref{G_T})) and $|\Pi_{ii}|=O(1)$ (by (\ref{Piii})), and bounded the $i=j$ terms by $O_\prec(N^{-2})=O_\prec(N^{-1}\Phi^2)$. Now applying (\ref{resolvent4}) to $G^{(\mu\nu)}_{ij}$, we get that
\begin{equation}\label{EG_smaller}
\begin{split}
\mathbb E G_{ij}^{(\mu\nu)} &= \mathbb EG_{ii}^{(\mu\nu)}G^{(i\mu\nu)}_{jj} \left( Y G^{(ij\mu\nu)} Y^*\right)_{ij} \\
&=  \Pi_{ii} \Pi_{jj}\mathbb E\left( Y G^{(ij\mu\nu)} Y^*\right)_{ij} + O_\prec(\Phi^2) = O_\prec(\Phi^2),
\end{split}
\end{equation}
where in the second step we used $|G_{ii}^{(\mu\nu)}-\Pi_{ii}| + |G^{(i\mu\nu)}_{jj} - \Pi_{jj}| \prec \Phi$ and
$$\left( Y G^{(ij\mu\nu)} Y^* \right)_{ij} = G_{ij}^{(\mu\nu)} \left(G_{ii}^{(\mu\nu)}G_{jj}^{(i\mu\nu)}\right)^{-1} \prec \Phi, $$
which follow easily from (\ref{ANISO_LAW2}) and (\ref{G_T}), and in the last step the leading order term vanishes since the two $X$ entries are independent for $i\ne j$. Then with (\ref{EG_smaller}), the terms in (\ref{bound_EG1}) can be bounded by $O_\prec(N^{-1}\Phi^2)$.

In general, after the expansion of the two fractions in (\ref{expansion1}), we get a summation of terms of the form
$$A_{m,n}:=\mathbb E \epsilon_\mu^m \epsilon_\nu^n ( Y^* G^{(\mu\nu)}Y)_{\mu\nu}, \ \ \mu\ne \nu,$$
up to some deterministic coefficients of order $O(1)$. Since $|\epsilon_{\mu,\nu}| \prec \Phi\lesssim N^{-\omega/2}$ for $z\in \mathbf D$ (we can take $\omega$ small enough such that $N^{-\omega/2} \ge q^2$), we only need to include the terms with $m+n \le 2+2/\omega$ and the tail terms will be smaller than $N^{-1}\Phi^2$. Note that in $A_{m,n}$, the $X_{*\mu}$ entries, $X_{*\nu}$ entries and $G^{(\mu\nu)}$ entries are mutually independent. Moreover, both the number of $X_{*\mu}$ entries and the number of $X_{*\nu}$ entries are odd. Thus to attain a nonzero expectation, we must pair the $X$ entries such that there are products of the forms $X_{i\mu}^{n_1}$ and $X_{j\nu}^{n_2}$ for some $n_1,n_2\ge 3$. As a result, we lose $(n_1-2)/2 + (n_2-2)/2\ge 1$ free indices, and this contributes an $N^{-1}$ factor. On the other hand, for the product of $G$ entries, we have the following three cases: (1) if there are at least $2$ off-diagonal $G$ entries, then we bound them with $O_\prec(\Phi^2)$; (2) if there is only $1$ off-diagonal $G$ entry, then we can use the trick in (\ref{bound_EG1}) and the bound (\ref{EG_smaller}); (3) if there is no off-diagonal $G$ entry, then we lose one more free index and get an extra $N^{-1}$ factor. This leads to the estimate (\ref{off_small}) for the term in (\ref{expansion1}).

For the second term in (\ref{Gexpansion1}), we again use Lemma \ref{lemm_resolvent} to expand the $G_{\mu\nu}$, $G_{\nu\mu}$ and $G_{\nu\nu}^{-1}$ entries. Our goal is to expand all the $G$ entries into polynomials of the random variables
\begin{equation}\label{Ssymbol}
S_{\alpha\beta}:=(Y^*G^{(\mu\nu)}Y)_{\alpha\beta}, \ \ \alpha,\beta\in \{\mu,\nu\},
\end{equation}
so that the $X$ entries and $G^{(\mu\nu)}$ entries are independent in the resulting expression. In particular, the {\it{maximally expanded}} terms (see (\ref{Amax})) can be expanded into $S_{\alpha\beta}$ variables directly through (\ref{resolvent2}) and (\ref{resolvent4}). However, {\it{non-maximally expanded}} terms are also created along the expansions in (\ref{resolvent4}) and (\ref{resolvent8}). Then we need to further expand these newly appeared terms. In general, this process will not terminate. However, we will show in Lemma \ref{iso_lem_3} that after sufficiently many expansions, the resulting expression either has enough off-diagonal terms, or is maximally expanded. In the former case, it suffices to bound each off-diagonal term by $O_\prec(\Phi)$. In the latter case, the expression will only consist of $S_{\alpha\beta}$ variables. Following the argument in the previous paragraph, the expectation over the $X$ entries produces an $N^{-1}$ factor, while the expectation over the $G$ entries produces a $\Phi^2$ factor. 

Next we give a rigorous proof based on the above arguments.

\subsection{Resolvent expansion}\label{subsection_exp}

To perform the resolvent expansion in a systematic way, we introduce the following notions of {\it{string}} and {\it{string operator}}. %Recall the definition of $S_{kl}$ in (\ref{Ssymbol}).

\begin{defn}[Strings]
Let $\mathfrak A$ be the alphabet containing all symbols that will appear during the expansion: 
%$$\mathfrak A=\left\{G_{ij}, G_{ji}, G_{ii},G_{jj}, G_{ii}^{-1},G_{jj}^{-1},G_{ii}^{(j)}, G_{jj}^{(i)}, (G_{ii}^{(j)})^{-1},(G_{jj}^{(i)})^{-1},S_{ii},S_{jj},S_{ij},S_{ji}\right\}.$$
$$\mathfrak A=\left\{G_{\alpha\beta}, G_{\alpha\alpha}^{-1},S_{\alpha\beta} \text{ with } \alpha,\beta \in \{\mu,\nu\}\right\} \cup \left\{G_{\mu\mu}^{(\nu)}, G_{\nu\nu}^{(\mu)}, (G_{\mu\mu}^{(\nu)})^{-1},(G_{\nu\nu}^{(\mu)})^{-1}\right\}.$$ 
We define a string $\mathbf s$ to be a concatenation of the symbols from $\mathfrak A$, and we use $\left\llbracket\bf s\right\rrbracket$ to denote the random variable represented by $\mathbf s$. %Let $\mathfrak M$ be the collection of all possible strings. 
We denote an empty string by $\emptyset$ with value $\left\llbracket\emptyset\right\rrbracket = 0$. 
\end{defn}
\begin{rem}
It is important to distinguish a string $\mathbf s$ from its value $\left\llbracket\bf s\right\rrbracket$. For example, $``G_{\mu\nu}"$ and $``G_{\mu\mu}G_{\nu\nu}^{(\mu)}S_{\mu\nu}"$ are different strings, but they represent the same random variable by (\ref{resolvent4}). 
\end{rem}

%We shall call a $G^{(\mathbb T)}_{kl}$ symbol to be {\it{maximally expanded}} if $\{k,l\}\cup \mathbb T =\{i,j\}$. or $S_{kl}$ In other words, 
We shall call the following symbols the {\it{maximally expanded}} symbols:
\begin{equation}\label{Amax}
\mathfrak A_{\max}=\left\{G_{\mu\nu},G_{\nu\mu},G_{\mu\mu}^{(\nu)},G_{\nu\nu}^{(\mu)},(G_{\mu\mu}^{(\nu)})^{-1},(G_{\nu\nu}^{(\mu)})^{-1},S_{\mu\mu},S_{\nu\nu},S_{\mu\nu},S_{\nu\mu}\right\}.
\end{equation}
%In our alphabet $\mathfrak A$, we say the symbols $G_{ij}$, $G_{ji}$, $G_{ii}^{(j)}$, $G_{jj}^{(i)}$, $(G_{ii}^{(j)})^{-1}$, $(G_{jj}^{(i)})^{-1}$ and $S_{ij}$, $S_{ji}$ are $maximally\ expanded$, and $G_{ii}$, $G_{jj}$, $G_{ii}^{-1}$, $G_{jj}^{-1}$ are not $maximally\ expanded$. 
A string $\mathbf s$ is said to be maximally expanded if all of its symbols are in $\mathfrak A_{\max}$. We shall call $G_{\mu\nu}, G_{\nu\mu}, S_{\mu\nu}, S_{\nu\mu}$ the {\it{off-diagonal}} symbols and all the other symbols {\it{diagonal}}. By (\ref{ANISO_LAW2}) and (\ref{G_T}), we have $\left\llbracket\mathbf a_o\right\rrbracket \prec \Phi$ if $\mathbf a_o$ is off-diagonal (we have $S_{\mu\nu}\prec \Phi$ using (\ref{resolvent4})) and $\left\llbracket\mathbf a_d\right\rrbracket \prec 1$ if $\mathbf a_d$ is diagonal. We use ${\cal F}_{n{\text{-}}max}(\mathbf s)$ and ${\cal F}_{\rm{off}}(\mathbf s)$ to denote the number of non-maximally expanded symbols and the number of off-diagonal symbols, respectively, in $\mathbf s$.

\begin{defn}[String operators]\label{defn_stringoperator}
%In the following definitions, $s,t$ are arbitrary distinct indices chosen from $\{i,j\}$.
Let $\alpha\ne \beta \in \{\mu,\nu\}$.

\begin{itemize}
\item[(i)] We define an operator $\tau_0$ acting on a string $\bf s$ in the following sense. Find the first $G_{\alpha\alpha}$ or $G_{\alpha\alpha}^{-1}$ in $\bf s$. If $G_{\alpha\alpha}$ is found, replace it with $G_{\alpha\alpha}^{(\beta)}$; if $G_{\alpha\alpha}^{-1}$ is found, replace it with $(G_{\alpha\alpha}^{(\beta)})^{-1}$; if neither is found, set $\tau_0(\bf s) = \bf s$ and we say that $\tau_0$ is trivial for $\bf s$.
\item[(ii)] We define an operator $\tau_1$ acting on a string $\bf s$ in the following sense. Find the first $G_{\alpha\alpha}$ or $G_{\alpha\alpha}^{-1}$ in $\bf s$. If $G_{\alpha\alpha}$ is found, replace it with $\frac{G_{\alpha\beta} G_{\beta\alpha} }{G_{\beta\beta} }$; if  $G_{\alpha\alpha}^{-1}$ is found, replace it with $-\frac{G_{\alpha\beta} G_{\beta\alpha}}{G_{\alpha\alpha} G_{\alpha\alpha}^{(\beta)} G_{\beta\beta}}$; if neither is found, set $\tau_1(\bf s)=\emptyset$ and we say that $\tau_1$ is null for $\bf s$.
\item[(iii)] The operator $\rho$ %acting on a string $\bf s$ in the following sense. 
replaces each $G_{\alpha\beta}$ in the string $\bf s$ with $G_{\alpha\alpha}G_{\beta\beta}^{(\alpha)}S_{\alpha\beta}$.
\end{itemize}
\end{defn}

By Lemma \ref{lemm_resolvent}, it is clear that for any string $\bf s$,
\begin{equation}\label{resolvent_string}
\llbracket\tau_0(\mathbf s)\rrbracket + \llbracket\tau_1(\mathbf s)\rrbracket = \llbracket\bf s\rrbracket, \ \ \llbracket\rho(\mathbf s)\rrbracket = \llbracket\mathbf s\rrbracket.
\end{equation}
Moreover, a string $\mathbf s$ is trivial under $\tau_0$ and null under $\tau_1$ if and only if $\mathbf s$ is maximally expanded.
Given a string $\bf s$, we abbreviate ${\mathbf s}_0 := \tau_0(\mathbf s)$ and ${\bf s}_1 := \rho(\tau_1(\bf s))$. For any sequence $w=a_1a_2\ldots a_m $ with $a_i\in \{0,1\}$, we denote 
$${\mathbf s}_{w}:=\rho^{a_m}\tau_{a_m}\ldots \rho^{a_2}\tau_{a_2}\rho^{a_1}\tau_{a_1}(\mathbf s), \ \ \text{ where }\rho^0:=1.$$
Then by (\ref{resolvent_string}) we have
\begin{equation}\label{resolvent_string2}
\sum_{|w|=m}\llbracket{\mathbf s}_w\rrbracket = \llbracket\mathbf s\rrbracket,
\end{equation}
where the summation is over all binary sequences $w$ with length $|w|=m$.

\begin{lem}\label{iso_lem_3} 
Consider the string $\mathbf s = ``G_{\mu\mu}G_{\nu\nu}^{(\mu)}S_{\mu\nu}"$. Let $w$ be any binary sequence with $|w|=4l_0$ and such that $\mathbf s_w \ne \emptyset$. Then either ${\cal F}_{\rm{off}}(\mathbf s_{w})\ge 2l_0$ or $\mathbf s_{w}$ is maximally expanded. 
\end{lem}
\begin{proof}
It suffices to show that any nonempty string $\mathbf s_w$ with ${\cal F}_{\rm{off}}(\mathbf s_{w})< 2l_0$ is maximally expanded. By Definition \ref{defn_stringoperator}, a nontrivial $\tau_0$ reduces the number of non-maximally expanded symbols by $1$, and keeps the number of off-diagonal symbols the same; a $\rho\tau_1$ increases the number of non-maximally expanded symbols by $2$ or $3$, and increases the number of off-diagonal symbols by $2$. Hence ${\cal F}_{\rm{off}}(\mathbf s_{w})< 2l_0$ implies that there are at most $(l_0-1)$ $1$'s in $w$. Those $\rho\tau_1$ operators increase ${\cal F}_{n{\text{-}}max}$ at most by $3(l_0 - 1)$ in total.
%$3(l_0 - 1) + 2 = 3l_0 - 1$ in total, where we used that the first $\rho\tau_1$ can only increase ${\cal F}_{n{\text{-}}max}$ by $2$. 
On the other hand, there are at least $3l_0$ $0$'s in $w$, which is sufficient to eliminate all the non-maximally expanded symbols (whose number is at most $3(l_0 - 1) + 1 = 3l_0-2$ in total since ${\cal F}_{n{\text{-}}max}(\mathbf s)=1$ for the initial string). 
%By Definition \ref{defn_stringoperator}, applying $\tau_1$(non-trivially) reduces the number of non-maximally expanded symbols by $1$, and keep the number of off-diagonal symbols the same. By definition \ref{defn_stringoperator} $(ii), (iii)$, applying $\rho\tau_2$(non-trivially) increase the number of non-maximally expanded symbols at most by $4$, and increase the number of off-diagonal symbols by $2$. Therefore ${\cal F}_{\rm{off}}(\mathbf s_{w})< l_0$ implies that we have applied $\rho\tau_1$ at most $l_0/2$ times, i.e. there are at most $l_0/2$ $1$'s in $w$. And at the same time, these operators increase the number of non-maximally expanded symbols by at most $2l_0$. Now we must have at least $4l_0-l_0/2 >3l_0$ $0$'s in $w$, which means we have applied $\rho\tau_0$ at least $3l_0$ times, which is sufficient to reduce all the non-maximally expanded symbols (the number is at most $2l_0+1$ since we start with $1$ non-maximally expanded symbol in $\bf s$). 
\end{proof}

Now we choose $l_0 = 1+1/\omega$. Then using $\Phi=O(N^{-\omega/2})$, we have
$$\sum_{|w|=4l_0}\llbracket\mathbf s_w\rrbracket \cdot \mathbf 1({\cal F}_{\rm{off}}(\mathbf s_{w})\ge 2l_0)\prec 2^{4l_0}\Phi^{2l_0}\prec N^{-1}\Phi^2.$$
 %by (\ref{eq_domainD}) and (\ref{eq_defpsi}). 
By Lemma \ref{iso_lem_3}, we see that to prove (\ref{off_small}), it suffices to show that
\begin{equation}\label{maximal_bound}
\left|\mathbb E \llbracket\mathbf s_w\rrbracket\right| \prec N^{-1}\Phi^2
\end{equation}
for any maximally expanded string $\mathbf s_w$ with $|w|=4l_0$. Note that the maximally expanded string $\mathbf s_w$ thus obtained consists only of the symbols 
$$G_{\alpha\alpha}^{(\beta)}, \ (G_{\alpha\alpha}^{(\beta)})^{-1}, \ S_{\alpha\beta},\ \ \text{ with } \alpha\ne \beta \in \{\mu,\nu\}.$$
By (\ref{resolvent2}), we can replace $(G_{\alpha\alpha}^{(\beta)})^{-1}$ with
\begin{equation}\label{diag_exp1}
(G_{\alpha\alpha}^{(\beta)})^{-1} = - z - S_{\alpha\alpha}.
\end{equation}
Note that $|S_{\alpha\alpha}-N^{-1}\sum_i\sigma_i \Pi_{ii}|\prec\Phi$ by (\ref{bound_epsilon}). Then we can expand $G_{\alpha\alpha}^{(\beta)}$ as
\begin{align}
G_{\alpha\alpha}^{(\beta)}= {m_{2c}}\sum_{k=0}^{2l_0} m_{2c}^k \left(S_{\alpha\alpha}- N^{-1}\sum_i \sigma_i \Pi_{ii}\right)^k+O_{\prec}(N^{-1}\Phi^2).\label{diag_exp2}
\end{align}
We apply the expansions (\ref{diag_exp1}) and (\ref{diag_exp2}) to the $G$ symbols in $\mathbf s_w$, disregard the sufficiently small tails, and denote the resulting polynomial (in terms of the symbols $S_{\alpha\beta}$) by $P_w$. Then $P_w$ can be written as a finite sum of maximally expanded strings (or monomials) consisting of the $S_{\alpha\beta}$ symbols. Moreover, the number of such monomials depends only on $l_0$.
%Moreover, the number of such strings depends only on $l_0$.
%multiply out (by regarding the expanded $\mathbf s_w$ as a polynomial in terms of $S_{ii}, S_{jj}, S_{ij}$). Notice that the tail term is already bounded by $N^{-1}\Psi^2$; all coefficients are of order $1$; the number of monomials only depend on $l_0$. 
Hence we only need to prove that for any such monomial $M_w$, 
\begin{equation}\label{maximal_bound2}
|\mathbb E \llbracket M_w \rrbracket | \prec N^{-1}\Phi^2.
\end{equation}

Let $N_\mu$ ($N_\nu$) be the number of times that $\mu$ $(\nu)$ appears as a lower index of the $S$ symbols in $M_w$. We have $N_\mu=N_\nu=3$ for the initial string $\mathbf s =``G_{\mu\mu}G_{\nu\nu}^{(\mu)}S_{\mu\nu}"$. From Definition \ref{defn_stringoperator}, it is easy to see that the operators $\tau_0, \tau_1$ and $\rho$ do not change the parity of $N_\mu$ and $N_\nu$. The expansions (\ref{diag_exp1}) and (\ref{diag_exp2}) also do not change the parity of $N_\mu$ and $N_\nu$. This leads to the following key observation: 
\begin{equation}\label{parity}
\text{\it{both $N_\mu$ and $N_\nu$ are odd in $M_w$}}.
\end{equation}

\subsection{A graphical proof}\label{graph}

In this subsection, we finish the proof of (\ref{maximal_bound2}). Suppose $M_w=C(z)(S_{\mu\mu})^{m_1}(S_{\nu\nu})^{m_2}(S_{\mu\nu})^{m_3}(S_{\nu\mu})^{m_4}$, where $C(z)$ denotes a deterministic function of order $1$ for all $z\in \mathbf D$. Then we write
\begin{equation}\label{iso_sum1}
\begin{split}
\llbracket M_w \rrbracket \sim \sum_{i_*^{(*)},j_*^{(*)} \in \mathcal I_1} & \prod_{a=1}^{m_1} X_{i_a^{(1)}\mu}G^{(\mu\nu)}_{i_a^{(1)} j_a^{(1)}}X_{j_a^{(1)} \mu}\prod_{b=1}^{m_2} X_{i_b^{(2)}\nu}G^{(\mu\nu)}_{i_b^{(2)} j_b^{(2)}}X_{j_b^{(2)} \nu} \\
& \prod_{c=1}^{m_3} X_{i_c^{(3)}\mu}G^{(\mu\nu)}_{i_c^{(3)} j_c^{(3)}}X_{j_c^{(3)} \nu}\prod_{d=1}^{m_4} X_{i_d^{(4)} \nu}G^{(\mu\nu)}_{i_d^{(4)} j_d^{(4)}}X_{j_d^{(4)} \mu}.
\end{split}
\end{equation}
%The notation in the above espression is heavy, that is the reason we want to introduce the graph representation of it. 
To avoid heavy expressions, we introduce the following graphical notations. We use a connected graph $(V,E)$ to represent the string $M_w$, where the vertex set $V$ consists of the indices in (\ref{iso_sum1}) and the edge set $E$ consists of the $X$ and $G$ variables. The indices $\mu,\nu$ are represented by the black vertices in the graph, while the $i,j$ indices are represented by the white vertices. The $X$ edges are represented by the zig-zag lines and the $G$ edges are represented by the straight lines. One can refer to Fig. \ref{M_1} for an example of such a graph.

\begin{figure}[htb]
\centering
\scalebox{0.8}{
\begin{tikzpicture}
 \tikzset{dot/.style={circle,fill=#1,inner sep=3,minimum size=0.5pt}}
 \tikzset{wdot/.style={circle,draw,inner sep=3,minimum size=0.5pt}}
 \tikzset{zig/.style={decoration={
    zigzag,
    segment length=4,
    amplitude=.9,post=lineto,
    post length=2pt}}}
 \node (i) [label=right:{$\mu$}, dot] {};
 \node (j) [label=left:{$\nu$}, right of=i, xshift = 20mm, dot] {};
 \node (c1) [left of=i, xshift=-6mm, yshift = 6mm, wdot] {} edge [decorate,zig]  (i);
 \node (c2) [left of=i, xshift=-7mm, yshift = -4mm, wdot] {} edge[decorate,zig] (i) edge (c1);
 \node (b1) [above of=i, xshift=3mm, yshift = 8mm, wdot] {} edge[decorate,zig] (i);
 \node (b2) [above of=j, xshift=-3mm, yshift = 8mm, wdot] {} edge[decorate,zig] (j) edge (b1);
  \node (b3) [above of=i, xshift=5mm, yshift = 0mm, wdot] {} edge[decorate,zig] (i);
 \node (b4) [above of=j, xshift=-5mm, yshift = 0mm, wdot] {} edge[decorate,zig] (j) edge (b3);
  \node (b5) [below of=i, xshift=3mm, yshift = -8mm, wdot] {} edge[decorate,zig] (i);
 \node (b6) [below of=j, xshift=-3mm, yshift = -8mm, wdot] {} edge[decorate,zig] (j) edge (b5);
 \node (d1) [right of=j, xshift=2mm, yshift = 16mm, wdot] {} edge [decorate,zig]  (j);
 \node (d2) [right of=j, xshift=7mm, yshift = 10mm, wdot] {} edge[decorate,zig] (j) edge (d1);
  \node (e1) [right of=j, xshift=2mm, yshift = -16mm, wdot] {} edge [decorate,zig]  (j);
 \node (e2) [right of=j, xshift=7mm, yshift = -10mm, wdot] {} edge[decorate,zig] (j) edge (e1);
 \node (l) [draw=black,thick,rounded corners=2pt,below left=10mm, minimum width = 75pt, minimum height = 60pt, right of=j,xshift = 40mm, yshift = 5mm]  {};
 \node (l1) [wdot, above of=l, xshift = -30pt, yshift = -12pt] {};\node (l2) [wdot, right of= l1, xshift = 5mm, label=right:{$G$}] {} edge (l1) ;
 \node (l3) [dot, below of = l1] {};\node (l4) [wdot, right of= l3, xshift = 5mm] {} edge [decorate, zig, label=right:{$X$}] (l3) ;
\end{tikzpicture}
}
\caption{The graph representing $S_{\mu\mu}(S_{\mu\nu})^3(S_{\nu\nu})^2$.}
\label{M_1}
\end{figure}
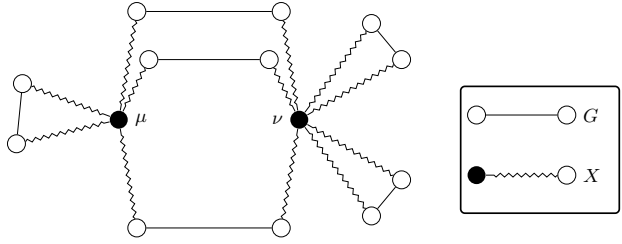

We organize the summation in (\ref{iso_sum1}) in the following way. We first partition the white vertices into blocks by requiring that any pair of white vertices take the same value if they are in the same block, and take different values otherwise. Then we take the summation over the white blocks which take values in $\mathcal I_2$. Finally, we sum over all possible partitions. Note that the number of different partitions depends only on the total number of $S$ variables in $M_w$, which in turn depends only on $l_0$.

%Note that the white vertices represent the summation indices taking values in $\mathcal I_2$. 

Fix a partition $\Gamma$ of the white vertices. We denote its blocks by $b_1,...,b_{k}$, where $k$ gives the number of distinct blocks in $\Gamma$. We denote by $n^\mu_l$ ($n_l^\nu$) the number of white vertices in $b_l$ that are connected to the vertex $\mu$ ($\nu$). Let $G(\Gamma)$ be the product of all the $G$ edges in the graph. Then we have
\begin{equation}\label{iso_sum2}
\llbracket M_w\rrbracket \sim \sum_{\Gamma}\sum_{b_1,...,b_{k}}^*G(\Gamma)\prod_{l=1}^{k}(X_{b_l \mu})^{n^\mu_l}(X_{b_l \nu})^{n_l^\nu},
\end{equation}
where $\sum^*$ denotes the summation subject to the condition that $b_1,...,b_{k}$ all take distinct values. Note that $k$, $b_l$, $n_l^\mu$ and $n_l^\nu$ all depend on $\Gamma$, and we have omitted the $\Gamma$ dependence for simplicity of notations. 

From (\ref{iso_sum1}), it is easy to observe that the $X$ edges are independent of $G(\Gamma)$. Thus taking expectation of (\ref{iso_sum2}) gives that
\begin{align}
|\mathbb E \llbracket M_w\rrbracket|\le & C\sum_{\Gamma}\sum_{b_1,...,b_k}^*|\mathbb E G(\Gamma)|\prod_{l=1}^k|\mathbb E(X_{b_l \mu})^{n_l^\mu}||\mathbb E(X_{b_l \nu})^{n_l^\nu}|\nonumber\\
\le & C\sum_{\Gamma}\sum_{b_1,...,b_k}^*|\mathbb E G(\Gamma)|\prod_{l=1}^k\mathbb E|X_{b_l \mu}|^{n_l^\mu}\mathbb E|X_{b_l \nu}|^{n_l^\nu}  \mathbf 1(n_l^\mu\ne 1,n_l^\nu\ne 1).\label{iso_sum30}
\end{align}
Note that we must have $n_l^\mu+n_l^\nu \ge 2$ for $1\le l \le k$, because we only consider nonempty blocks. On the other hand, if all $n_l^\mu$ are even, then $N_\mu = \sum_{l=1}^k n_l^\mu$ must be even, which contradicts (\ref{parity}). Hence we can find some $1\le l_1 \le k$ such that $n^\mu_{l_1}$ is odd and $n^\mu_{l_1}\ge 3$. Similarly, we can also find some $1\le l_2 \le k$ such that $n^\nu_{l_2}$ is odd and $n^\nu_{l_2}\ge 3$. We abbreviate $\hat n_l^\mu:=n_l^\mu\wedge3$ and $\hat n_l^\nu:=n_l^\nu\wedge3.$
%Define
%$$\hat n_l^\mu:=\begin{cases}n_l^\mu\wedge2,\text{ if } l \text{ is even}\\n_l^\mu\wedge3,\text{ if } l \text{ is odd}\end{cases}, \ \ \  \hat n_l^\nu:=\begin{cases}n_l^\nu\wedge2,\text{ if } l \text{ is even}\\n_l^\nu\wedge3,\text{ if } l \text{ is odd}\end{cases}.$$
From the above discussions, we see that
\begin{equation}\label{num_est}
\frac{1}{2}\sum_{l=1}^k \left(\hat n_l^\mu+\hat n_l^\nu\right) \ge \frac{1}{2}\sum_{l\ne l_1,l_2}^k (\hat n_l^\mu+\hat n_l^\nu)+\frac{3}{2}+\frac{3}{2} \ge (k-2)+3 = k+1.
\end{equation}
Now using the moment assumption (\ref{conditionA3}), we can bound (\ref{iso_sum30}) by
%\begin{align}
%|\mathbb E \llbracket M_w\rrbracket|\le &  C\sum_{\Gamma}\sum_{b_1,...,b_k}^*|\mathbb E G(\Gamma)|\prod_{l=1}^k\mathbb E|X_{b_l \mu}|^{\hat n_l^\mu}\mathbb E|X_{b_l \nu}|^{\hat n_l^\nu} \nonumber\\
%\le &C\sum_{\Gamma}\sum_{b_1,...,b_k}^*|\mathbb E G(\Gamma)|N^{- \sum_{l=1}^k(\hat n_l^\mu+\hat n_l^\nu)/2} .\label{iso_sum3}
%\end{align}
\begin{align}
|\mathbb E \llbracket M_w\rrbracket|\le C\sum_{\Gamma}\sum_{b_1,...,b_k}^*|\mathbb E G(\Gamma)|N^{- \sum_{l=1}^k(\hat n_l^\mu+\hat n_l^\nu)/2} .\label{iso_sum3}
\end{align}
%where in the first step we used $\mathbb E|X_{b_l \mu}|^{n_l^\mu}\le q^{n_l^\mu-\hat n_l^\mu}\mathbb E|X_{b_l \mu}|^{\hat n_l^\mu}\le \mathbb E|X_{b_l \mu}|^{\hat n_l^\mu}$, in the second step we used the up-to-4th moment assumption. 

Next we deal with $|\mathbb E G(\Gamma)|$. We consider the following $3$ cases separately: (i) there are at least $2$ off-diagonal $G$-edges in $G(\Gamma)$; (ii) there is only $1$ off-diagonal $G$-edge in $G(\Gamma)$; (iii) there is no off-diagonal $G$-edge in $G(\Gamma)$.

In case (i), we trivially have $|\mathbb E G(\Gamma)|\prec \Phi^2$. %, because the diagonal edges are of order $O_\prec(1)$, while the off-diagonal edges are of order $O_{\prec}(\Phi)$.
In case (ii), we use the same trick as in (\ref{bound_EG1}). Let the off-diagonal $G$-edge be $G_{ij}^{(\mu\nu)}$. For each diagonal $G^{(\mu\nu)}_{kk}$, we replace it with $(G^{(\mu\nu)}_{kk}-\Pi_{kk})+\Pi_{kk} = \Pi_{kk} + O_\prec(\Phi).$
Plugging these expansions into $\mathbb EG(\Gamma)$, we obtain that $|\mathbb E G(\Gamma)|\prec \Phi^2+|\mathbb E G_{ij}^{(\mu\nu)}| \prec \Phi^2,$ where we used (\ref{EG_smaller}) in the second step. Finally, in case (iii), we have $|\mathbb EG(\Gamma)| \prec 1$. Moreover, $n_l^\mu + n_l^\nu$ is even for any $1\le l \le k$. Take $1\le l_1 , l_2 \le k$ such that $n^\mu_{l_1},n^\nu_{l_2}$ are odd and $n^\mu_{l_1}, n^\nu_{l_2}\ge 3$. If $l_1\ne l_2$, then we must have $\hat n^\mu_{l_1}+\hat n^\nu_{l_1}\ge 4,\ \hat n^\mu_{l_2} + \hat n^\nu_{l_2}\ge 4$, and hence
$$\frac{1}{2}\sum_{l=1}^k \left(\hat n_l^\mu+\hat n_l^\nu\right) \ge \frac{1}{2}\sum_{l\ne l_1,l_2}^k (\hat n_l^\mu+\hat n_l^\nu)+4 \ge k+2.$$
Otherwise, if $l_1=l_2$, then
$$\frac{1}{2}\sum_{l=1}^k \left(\hat n_l^\mu+\hat n_l^\nu\right) \ge \frac{1}{2}\sum_{l\ne l_1}^k (\hat n_l^\mu+\hat n_l^\nu)+3 \ge k+2.$$

Now applying the above estimates and (\ref{num_est}) to (\ref{iso_sum3}), we obtain that
\begin{align*}
|\mathbb E \llbracket M_w\rrbracket| &\prec \sum_{\Gamma\text{ in Case (1), (2)}}\Phi^2 N^{k - \sum_{l=1}^k(\hat n_l^\mu+\hat n_l^\nu)/2}+\sum_{\Gamma\text{ in Case (3)}}N^{k-\sum_{l=1}^k(\hat n_l^\mu+\hat n_l^\nu)/2}\\
&\le C(N^{-1}\Phi^2+N^{-2})\le CN^{-1}\Phi^2.
\end{align*}
%where in the last step we used (\ref{psi12}). 
This concludes the proof of (\ref{maximal_bound2}), and hence finishes the proof of (\ref{off_small}).

\section*{Acknowledgements}
The authors would like to thank Zongming Ma for valuable suggestions on statistical applications, which have significantly improved this paper. We are also grateful to the editors and referees for carefully reading our manuscript and suggesting several improvements.

\appendix
\section{Supplementary Material}\label{chap_supp}
In the supplementary material, we would like to give the proof of Theorem \ref{lem_EG0}, Theorem \ref{thm_large} and Lemma \ref{comp_claim}. %We first record some basic resolvent identities and estimates. 

For $\mathbf v,\mathbf w \in \mathbb C^{\mathcal I}$, $a\in \mathcal I$ and an $\mathcal I\times \mathcal I$ matrix $A$, we abbreviate
\begin{equation}\nonumber
A_{\mathbf{vw}}:=\langle \mathbf v,A\mathbf w\rangle, \ \ A_{\mathbf{v}a}:=\langle \mathbf v,A\mathbf e_a\rangle, \ \ A_{a\mathbf{w}}:=\langle \mathbf e_a,A\mathbf w\rangle,
\end{equation}
where $\mathbf e_a$ denotes the standard unit vector in the coordinate direction $a$. We shall call them the generalized matrix entries. We sometimes identify vectors $\mathbf v\in \mathbb C^{\mathcal I_1}$ and $\mathbf w\in \mathbb C^{\mathcal I_2}$ with their natural embeddings $\left( {\begin{array}{*{20}c}
   {\mathbf v}  \\
   0 \\
\end{array}} \right)$ and $\left( {\begin{array}{*{20}c}
   0  \\
   \mathbf w \\
\end{array}} \right)$ in $\mathbb C^{\mathcal I}$. The exact meanings will be clear from the context.

\begin{lem}\label{lemma_Im}
Given any $M\times N$ matrix $Y$, the following estimates and identities hold for $G\equiv G(Y,z)$:
\begin{equation}\label{eq_gbound}
\left\|G \right\| \le C\eta ^{ - 1} ,\ \ \left\|\partial _z G \right\|\le C\eta ^{ - 2},
\end{equation}
%Moreover, 
 for some constant $C>0$, and for $\mathbf v\in \mathbb C^{\mathcal I_1}$ and $\mathbf w\in \mathbb C^{\mathcal I_2}$, %we have the following identities
\begin{align}
&\sum\limits_{i \in \mathcal I_1 }  \left| {G_{\mathbf v i} } \right|^2 = \sum\limits_{i \in \mathcal I_1 }  \left| {G_{i\mathbf v} } \right|^2  = \frac{|z|^2}{\eta}\Im\left(\frac{G_{\mathbf v\mathbf v}}{z}\right)\label{eq_gsq2} \\
& \sum\limits_{i \in \mathcal I_1 } {\left| {G_{\mathbf w i} } \right|^2 } = \sum\limits_{i \in \mathcal I_1 } {\left| {G_{i\mathbf w} } \right|^2 } = {G}_{\mathbf w\mathbf w}  + \frac{\bar z}{\eta} \Im \, G_{\mathbf w\mathbf w} , \label{eq_gsq3} \\
&\sum\limits_{\mu  \in \mathcal I_2 } {\left| {G_{\mathbf w \mu } } \right|^2 } = \sum\limits_{\mu  \in \mathcal I_2 } {\left| {G_{\mu \mathbf w} } \right|^2 }  = \frac{{\Im \, G_{\mathbf w\mathbf w} }}{\eta }, \label{eq_gsq1} \\
&\sum\limits_{\mu \in \mathcal I_2 } {\left| {G_{\mathbf v \mu} } \right|^2 } = \sum\limits_{\mu \in \mathcal I_2 } {\left| {G_{\mu \mathbf v} } \right|^2 } =  \frac{{G}_{\mathbf v\mathbf v}}{z}  + \frac{\bar z}{\eta} \Im\left(\frac{{G_{\mathbf v\mathbf v} }}{z}\right) .\label{eq_gsq4} 
 \end{align}
These estimates remain true for $G^{(\mathbb T)}$ instead of $G$ for any $\mathbb T \subseteq \mathcal I$. 
%Finally, suppose $\{\mathbf v_{i}\}_{i=1}^{N}$ and $\{\mathbf w_{\mu}\}_{\mu=1}^{N}$ are orthonormal bases of $\mathbb C^{\mathcal I_1}$ and $\mathbb C^{\mathcal I_2}$, respectively, then the above estimates remain true if we replace $\mathbf e_i$ with $\mathbf v_i$ and $\mathbf e_\mu$ with $\mathbf v_{\mu}$.
\end{lem}
\begin{proof}
These estimates and identities can be proved through simple calculations using the spectral decomposition of $G$. The reader can also refer to, for example, \cite[Lemma 4.6]{Anisotropic}, \cite[Lemma 3.5]{XYY} and \cite[Lemma A.3]{NeceSuff_sample}.
\end{proof}

\subsection{Proof of Lemma \ref{comp_claim}}\label{proof_lem_EG_center}

For $z\in \mathbf D $, we have
\begin{equation}\nonumber
 G(X,z) :=  \left( {\begin{array}{*{20}c}
   { - I_{M\times M}}  & X_1+B  \\
   {X_1^*+B^*} & {- zI_{N\times N}}  \\
   \end{array}} \right)^{-1} = \left(G_1^{-1} + V\right)^{-1}, 
 \end{equation}
where we abbreviate $G_1(z) :=G(X_1,z)$ and $V := \left( {\begin{array}{*{20}c}
   {0} & B  \\
   B^* & {0}  \\
   \end{array}} \right)$.
%By our assumption, (\ref{aniso_law}) holds for $G_1$. 
Then we expand $G$ using the resolvent expansion
\begin{equation}\label{rsexp1}
G = G_1 - G_1 V G_1 + (G_1 V )^2 G_1 - (G_1 V )^3 G.
\end{equation}
We need to estimate the last three terms of the right-hand side.
%Suppose the following entrywise local law holds:
%\begin{equation}
%\left| \langle \mathbf u, G(X_1,z) \mathbf v\rangle - \langle \mathbf u, \Pi (z)\mathbf v\rangle \right| \prec q+\Psi(z),
%\end{equation}
First, note that by (\ref{eq_gsq2})-(\ref{eq_gsq4}) and (\ref{aniso_law}), we have for $z\in \mathbf D$,
\begin{equation}\label{sum_bound0}
\begin{split}
\max\Big\{\sum_{i}\left|(G_1)_{\mathbf v i}\right|^2, \sum_{i}\left|(G_1)_{i \mathbf v}\right|^2,\sum_{\mu}\left|(G_1)_{\mathbf v \mu}\right|^2,\sum_{\mu}\left|(G_1)_{\mu \mathbf v}\right|^2\Big\} \\ \prec \eta^{-1},
\end{split}
\end{equation}
for any $\mathbf v \in \mathbb C^{\mathcal I}$ and $\mathbb T \subseteq \mathcal I$ with $|\mathbb T|=O(1)$.
%\end{lem}

For any unit vectors $\mathbf u, \mathbf v \in \mathbb C^{\mathcal I}$, we have
\begin{equation}\label{exp1}
\begin{split}
\left|\langle \mathbf u, G_1 V G_1\mathbf v\rangle \right| & \le \sum_{b\in \mathcal I} \Big| \sum_{a\in \mathcal I}\left(G_1\right)_{\mathbf u a} V_{ab}\Big| |\left(G_1\right)_{b\mathbf v}| \\
& \prec \max_{b} \Big( \sum_{a\in \mathcal I} |V_{ab}|^2 \Big)^{1/2} \sum_{b\in \mathcal I}  |\left(G_1\right)_{b\mathbf v}| \\
& \prec N^{-1-c_0} \Big(\sum_{b\in \mathcal I}  |\left(G_1\right)_{b\mathbf v}|^2\Big)^{1/2} \prec N^{-1-c_0}\eta^{-1/2},
\end{split}
\end{equation}
where in the second step we used (\ref{aniso_law}) for $G_1$, in the third step the Cauchy-Schwarz inequality and (\ref{boundB}), and in the last step (\ref{sum_bound0}). With a similar argument, we obtain that
\begin{equation}\label{exp2}
\begin{split}
\left|\langle \mathbf u, (G_1 V )^2 G_1 \mathbf v\rangle \right| \prec N^{-2-2c_0}\eta^{-1}.
\end{split}
\end{equation}
Combining (\ref{exp2}) with the rough bound (\ref{eq_gbound}) for $G$, we get that
\begin{equation}\label{exp3}
\begin{split}
& \left|\langle \mathbf u, (G_1 V)^3 G\mathbf v\rangle \right| =\Big| \sum_{a,b}\left((G_1 V )^2 G_1\right)_{\mathbf u a} V_{ab} G_{b\mathbf v}\Big| \\
&\prec \left(N^{-2-2c_0}\eta^{-1}\right) \eta^{-1} \sum_a\Big(\sum_{b} |V_{ab}|^2\Big)^{1/2} \le CN^{-3/2-3c_0}\eta^{-1},\end{split}
\end{equation}
where we used $\eta \ge N^{-1}$ for $z\in \mathbf D$ in the last step. Plugging the estimates (\ref{exp1})-(\ref{exp3}) into (\ref{rsexp1}), we conclude that
\begin{equation}\label{truncate_compare}
\left|\langle \mathbf u, G \mathbf v\rangle - \langle \mathbf u, G_1 \mathbf v\rangle\right| \prec N^{-1-c_0}\eta^{-1/2}\le (N\eta)^{-1}.
\end{equation}
for all deterministic unit vectors $\mathbf u,\mathbf v \in \mathbb C^{\mathcal I}$. %We can then easily conclude the lemma with this estimate.

\subsection{Proof of Theorem \ref{lem_EG0}}

By Lemma \ref{comp_claim}, we can assume that the entries of $X$ are centered without loss of generality. According to the comments below Theorem \ref{lem_EG0}, we can repeat the proof in \cite{NeceSuff_sample} to get the entrywise local law \eqref{entry_law} 
%\begin{equation}\label{entry_law}
%\max_{a,b\in \mathcal I}\left|G_{ab}(X,z) - \Pi_{ab}(z)\right| \prec q+\Psi(z),
%\end{equation}
and the averaged local law \eqref{aver_law}. 
%\begin{equation}\label{aver_law}
%| m_2(X,z)-m_{2c}(z) | + \Big|M^{-1}\sum_{i\in \mathcal I_1} \sigma_i (G_{ii} - \Pi_{ii}) \Big|\prec (N\eta)^{-1}.
%\end{equation}
Then combining (\ref{entry_law}), the moment assumption (\ref{conditionA3}) for $X$ and the arguments in Section \ref{isoq} below, we can obtain the anisotropic local law (\ref{aniso_law}) for $G(X,z)$. Hence we focus on proving the bound \eqref{ANISO_LAW2}.
%\begin{equation}\label{ANISO_LAW2}
%\left| \langle \mathbf v, G(X,z) \mathbf v\rangle - \langle \mathbf v, \Pi (z)\mathbf v\rangle \right| \prec q^2 +(N\eta)^{-1/2}
%%\end{split}
%\end{equation}
%for $\mathbf v \in \mathbb C^{\mathcal I_1}$ or $\mathbf v \in \mathbb C^{\mathcal I_2}$. Then it concludes the proof of Theorem 3.4 by polarization. 
In fact, (\ref{ANISO_LAW2}) clearly follows from Lemma \ref{comp_claim} and the next two lemmas combined with the polarization identity.

\begin{lem}\label{STRONG_ENTRY}
Let $X$ be an $M\times N$ real random matrix whose entries are independent random variables satisfying (\ref{entry_assmX1}), (\ref{conditionA3}), and the bounded support condition \eqref{eq_support} 
%\begin{equation}
%\max_{i\in \mathcal I_1, \mu \in \mathcal I_2}\vert X_{i\mu}\vert \prec q, \label{eq_support}
%\end{equation}
with $q\le N^{-\phi}$ for some constant $\phi>0$. If (\ref{entry_law}) and (\ref{aver_law}) %\begin{equation}\label{aver_law}
%| m_1(X,z)-m_{1c}(z) | \prec \frac{1}{N\eta}
%\end{equation}
%and
%\begin{equation}\label{entry_law}
%| m_1(X,z)-m_{1c}(z) | \prec \frac{1}{N\eta}, 
%\end{equation}
hold uniformly in $z\in \mathbf D$, then the following local law also holds uniformly in $z\in \mathbf D$:
\begin{equation}\label{entry_law02}
\max_{r =1,2}\max_{a,b\in \mathcal I_r}\left|G_{ab}(X,z) - \Pi_{ab}(z)\right| \prec q^2+(N\eta)^{-1/2}.
\end{equation}
%(\ref{entry_law02}) also holds for all $z\in \mathbf D$. 
\end{lem}

\begin{lem}\label{ANISO_LEM}
Suppose the assumptions in Lemma \ref{STRONG_ENTRY} hold. Let $\Phi(z)$ be a deterministic function on $\mathbf D$ satisfying $c_1(N^{-1/2} + q^2) \le \Phi(z) \le N^{-c_1}$ for some constant $c_1>0$. If we have
\begin{equation}
\max_{a,b\in \mathcal I} \left|G_{ab}(z) - \Pi_{ab}(z) \right|^2 \prec \Phi, \ \ \max_{r =1,2}\max_{a,b\in \mathcal I_r} \left|G_{ab}(z) - \Pi_{ab}(z) \right| \prec \Phi , \label{aniso_lem1}
\end{equation}
uniformly in $z\in \mathbf D$, then %(\ref{ANISO_LAW2}) holds uniformly 
\begin{equation}
\left| \langle \mathbf v, G(X,z) \mathbf v\rangle - \langle \mathbf v, \Pi (z)\mathbf v\rangle \right|  \prec \Phi(z) \label{aniso_lem2}
\end{equation}
uniformly in $z\in \mathbf D$ and deterministic unit vectors $\mathbf v \in \mathbb C^{\mathcal I_1}$ or $\mathbf v \in \mathbb C^{\mathcal I_2}$.
\end{lem}

In next two subsections, we give the proof of Lemma \ref{STRONG_ENTRY} and Lemma \ref{ANISO_LEM}. Note that if we suppose (\ref{entry_law}) holds, then using (\ref{eq_gsq2})-(\ref{eq_gsq4}) and (\ref{G_T}), it is easy to verify that for $z\in \mathbf D$,
%\begin{equation}\label{sum_bound}
%\max\Big\{\sum_{i}\left|G^{(\mathbb T)}_{\mathbf v i}\right|^2, \sum_{i}\left|G^{(\mathbb T)}_{i\mathbf v}\right|^2,\sum_{\mu}\left|G^{(\mathbb T)}_{\mathbf v \mu}\right|^2,\sum_{\mu}\left|G^{(\mathbb T)}_{\mu\mathbf v}\right|^2\Big\}\prec \eta^{-1},
%\end{equation}
%for any deterministic unit vector $\mathbf v\in \mathbb C^{\mathcal I}$ and $\mathbb T \subseteq \mathcal I$ with $|\mathbb T|=O(1)$.
%\begin{lem}
%Suppose (\ref{entry_law}) holds. Then we have
\begin{equation}\label{sum_bound}
\max\Big\{\sum_{i}\left|G^{(\mathbb T)}_{a i}\right|^2, \sum_{i}\left|G^{(\mathbb T)}_{i a}\right|^2,\sum_{\mu}\left|G^{(\mathbb T)}_{a \mu}\right|^2,\sum_{\mu}\left|G^{(\mathbb T)}_{\mu a}\right|^2\Big\}\prec \eta^{-1},
\end{equation}
for any $a\in {\mathcal I}$ and $\mathbb T \subseteq \mathcal I$ with $|\mathbb T|=O(1)$.

\subsection{Proof of Lemma \ref{STRONG_ENTRY}}\label{subsection_aniso}

We only prove
\begin{equation}\label{entry_law2}
\max_{i,j\in \mathcal I_1} \left|G_{ij}(X,z) - \Pi_{ij} (z) \right| \prec q^2 + (N\eta)^{-1/2},
\end{equation}
where $ \Pi_{ij} = -(1+m_{2c}\sigma_i)^{-1}\delta_{ij}$. The proof for (\ref{entry_law02}) with $a,b\in \mathcal I_2$ is exactly the same. First, we recall the following large deviation bounds proved in \cite{EKYY1}. 

\begin{lem} [Lemma 3.8 of \cite{EKYY1}] \label{largederivation}
Let $(x_i)$, $(y_i)$ be independent families of centered and independent random variables, and $(A_i)$, $(B_{ij})$ be families of deterministic complex numbers. Suppose the entries $x_i$ and $y_j$ have variance $O( N^{-1})$ and satisfy (\ref{eq_support}) with $N^{-1/2} \le q\le N^{-\phi}$ for some fixed $\phi>0$. 
%Suppose the range of the indices is $1\le i,j \le K$ for some $K=O(N)$. 
Then for $K=O(N)$, we have the following bounds:
\begin{align}
%& \Big\vert \sum_i A_i x_i \Big\vert \leq \varphi^{\xi}\left[q \max_{i} \vert A_i \vert+ \frac{1}{\sqrt{N}}\Big(\sum_i |A_i|^2 \Big)^{1/2} \right], \\
& \Big\vert \sum_{1\le i,j \le K} x_i B_{ij} y_j \Big\vert \prec q^2 B_d  + qB_o + \frac{1}{N}\Big(\sum_{i\ne j} |B_{ij}|^2\Big)^{{1}/{2}} , \label{larged1}\\
& \Big\vert \sum_{1\le i\ne j \le K} \bar x_i B_{ij} x_j \Big\vert \prec qB_o + \frac{1}{N}\Big(\sum_{i\ne j} |B_{ij}|^2\Big)^{{1}/{2}} , \label{larged2}\\
& \Big\vert \sum_{1\le i \le K} \left(|x_i|^2 -\mathbb E|x_i|^2\right) B_{ii} \Big\vert \prec q B_d , \label{larged3}
\end{align}
where $B_d:=\max_{i} |B_{ii} |$ and $B_o:= \max_{i\ne j} |B_{ij}|.$
\end{lem}  

In fact, these bounds are stated in slightly stronger forms in \cite{EKYY1} with a different notion for high probability events. Here we choose to present (\ref{larged1})-(\ref{larged3}) in the form of stochastic domination, which is more convenient for our use. Moreover, if we assume the fourth moment of $x_i$ is bounded for all $i$ as in (\ref{conditionA3}), then we have a better bound for the LHS of (\ref{larged3}).

\begin{lem} \label{largederivation2}
Suppose the assumptions in Lemma \ref{largederivation} hold and $x_i$, $1\le i \le K$, satisfy (\ref{conditionA3}). Then we have
\begin{align}
\Big\vert \sum_{i} \left(|x_i|^2 -\mathbb E|x_i|^2\right) B_{ii} \Big\vert \prec \left(q^2 + N^{-1/2}\right) B_d . \label{larged32}
\end{align}
%where $B_d:=\max_{i} |B_{ii} |$ and $B_o:= \max_{i\ne j} |B_{ij}|.$
\end{lem} 
\begin{proof}
We abbreviate $z_i:=\left(|x_i|^2 -\mathbb E|x_i|^2\right) B_{ii}/B_d$. By Markov's inequality, it suffices to prove that for any fixed $p\in \mathbb N$,
\begin{equation}\label{large_deviation4}
\mathbb E \Big|\sum_i z_i\Big|^{2p} \prec \left(q^2 + N^{-1/2}\right)^{2p}.
\end{equation}
Note that by the assumption, we have
\begin{equation}\label{moment_second}
\mathbb E z_i = 0, \  \ \mathbb E|z_i|^n \prec q^{2n-4}N^{-2}\text{ for fixed } n\ge 2.
\end{equation}

%\begin{equation}\label{moment_second}\mathbb E(y_\mu)^n \prec q^{2n-4}\mathbb E(y_\mu)^2\prec q^{2n-4}N^{-2}\text{ for }n\ge 2,\end{equation}
%where we used (\ref{conditionA3}). 
Now we expand the LHS of (\ref{large_deviation4}) as
\begin{equation*}
\mathbb E\Big|\sum_i z_i \Big|^{2p} = \sum_{\substack{i_1,\ldots,i_{2p}}}\mathbb E y_{i_1}\cdots y_{i_{2p}} , %= \sum_\Gamma\sum^*_{\substack{b_l, l=1,2,...,k}}y_{\mu_{b_1}}^{n_1}...y_{\mu_{b_k}}^{n_k},
\end{equation*}
where we denote $y_{i_l}:=z_{i_l}$ for $1\le l \le p$ and $y_{i_l}:=\bar z_{i_l}$ for $p+1\le l \le 2p$. To organize the summation over the indices $i_1, \ldots , i_{2p}$, we look at the partitions $\Gamma$ of the set of the labels $\{1,...,2p\}$ according to the equivalence relation that $k,l$ are in the same class if and only if $i_k = i_l$. We use $b_l$, $1\le l \le k$, to denote the equivalence classes of $\Gamma$ and $n_l$ to denote the size of $b_l$. Obviously, $k$, $b_l$ and $n_l$ all depend on $\Gamma$, but we will omit this dependence in the following expressions. Moreover, since the random variables are centered, we must have $n_l \ge 2$ for all $l$ to attain a nonzero expectation. Hence we have
\begin{equation}\label{expandp}
\mathbb E\Big|\sum_i z_i \Big|^{2p} \le \sum_\Gamma\sum^*_{ b_1,\ldots,b_k}\mathbb E|y_{i_{b_1}}|^{n_1}\ldots \mathbb E|y_{i_{b_k}}|^{n_k},
\end{equation}
where $\sum^*$ denotes the summation subject to the conditions that $b_1, \ldots, b_k$ are all distinct, $n_l \ge 2$ for all $l$, and $\sum_{l=1}^k n_l = 2p$. Note that under these conditions, we trivially have $k\le p$.

Using (\ref{moment_second}), we obtain that
\begin{align*}
\sum^*_{b_1,\ldots,b_k}\mathbb E|y_{\mu_{b_1}}|^{n_1}\ldots \mathbb E|y_{\mu_{b_k}}|^{n_k} &\prec \sum^*_{b_1,\ldots,b_k} (q^{2n_1-4}N^{-2})\ldots (q^{2n_k-4}N^{-2})\\
& = \sum^*_{b_1,\ldots,b_k} N^{-2k} q^{4p-4k} \le C N^{-k}q^{4p-4k} .
%&\prec N^{-k}\left(q^2\right)^{2p-2k}\\
%&\prec \Psi^{2k}\left(q^2\right)^{2p-2k}\prec (q^2+\Psi)^{2p}
\end{align*}
Since the number of partitions of $\{1,...,2p\}$ is finite and depends only on $p$, (\ref{expandp}) can then be bounded by
$$\mathbb E\Big|\sum_i z_i \Big|^{2p} \prec \max_{1\le k \le p} N^{-k}q^{4p-4k} \le  q^{4p} + N^{-p},$$
where in the last step, $q^{4p}$ and $N^{-p}$ can be obtained from the extreme cases $k=0$ and $k=p$, respectively. This concludes (\ref{large_deviation4}).
\end{proof}

Now using (\ref{resolvent4}) and (\ref{larged1}), we get that for $i\ne j \in \mathcal I_1$,
\begin{align}
|G_{ij}|  \prec \left|\sum_{\mu, \nu}Y_{i\mu}G_{\mu\nu}^{(ij)}Y^*_{\nu j}\right| &\prec q^2\max_{\mu}|G^{(ij)}_{\mu\mu}|+ q\max_{\mu\ne \nu}{|G^{(ij)}_{\mu\nu}|}+\Big(\frac{1}{N^2}\sum_{\mu\ne \nu}|G^{(ij)}_{\mu\nu}|^2\Big)^{1/2} \nonumber\\
&\prec q^2 + q(q+\Psi) + \left( \frac{1}{N\eta}\right)^{1/2} \prec q^2 + (N\eta)^{-1/2}, \label{off_imp}
\end{align}
where we used (\ref{entry_law}), (\ref{G_T}) and the bound (\ref{sum_bound}). For the diagonal estimate, we need to control the $Z$ variables %(defined in (4.11) of the main article)
\begin{equation}\label{Z_variable}
  Z_{i} :=(1-\mathbb E_{i}) G_{ii}^{-1} = \sigma_i\left(m_2^{(i)} - ( {XG^{\left( i \right)} X^*} )_{ii}\right).
\end{equation}
%where $\mathbb E_{i}[\cdot]:=\mathbb E[\cdot| H^{(i)}],$ i.e. it is the partial expectation in the randomness of the $i$-th row and column of $H$, and we used (\ref{resolvent2}) in the second step. 
Using (\ref{larged2}) and (\ref{larged32}), we get that %for any $\mathbb T \subset \mathcal I$ with fixed length,
\begin{equation}\label{Z_bound}
\begin{split}
 |Z_{i} | & = \sigma_i \Big|\sum_{\mu}G^{(i)}_{\mu\mu}\left(|X_{i\mu}|^2 - \mathbb E|X_{i\mu}|^2\right)+\sum_{\mu\ne\nu}X_{i\mu}G_{\mu\nu}^{(i)}X^*_{\nu i} \Big| \\
& \prec  \left(q^2 + N^{-1/2}\right) + q\max_{\mu\ne \nu}{|G^{(i)}_{\mu\nu}|}+\frac{1}{N}\Big(\sum_{\mu\ne\nu}|G^{(i)}_{\mu\nu}|^2\Big)^{1/2}\\
&  \prec q^2 + (N\eta)^{-1/2},
\end{split}
\end{equation}
where we used (\ref{entry_law}), (\ref{G_T}) and (\ref{sum_bound}) again. Then with (\ref{resolvent2}), we get that
\begin{align*}
G_{ii} - \Pi_{ii} & = \frac{1}{-1-\sigma_i m_{2c} - \sigma_i(m_2^{(i)}-m_{2c}) + Z_i} - \frac{1}{-1-\sigma_i m_{2c}}\\
%&=\frac{1}{-1-m_{2c}}- zm_{1c}+  O_\prec\left(q^2 + (N\eta)^{-1/2}\right) \\
&= O_\prec\left(q^2 + (N\eta)^{-1/2}\right)\nonumber
\end{align*}
where in the second step we used (\ref{Z_bound}), (\ref{aver_law}) and (\ref{G_T}) (with $\tilde \Phi=q+\Psi$).
%, and in the third step we used $(-1-m_{2c})^{-1}=zm_{1c}$ (which follows from the definition of $m_{1c}$ and $m_{2c}$). 
Together with (\ref{off_imp}), we conclude (\ref{entry_law2}).

\subsection{Proof of Lemma \ref{ANISO_LEM}}\label{isoq}

We only prove (\ref{aniso_lem2}) for $\mathbf v\in \mathbb C^{\mathcal I_1}$. The proof for the case with $\mathbf v\in \mathbb C^{\mathcal I_2}$ is exactly the same. Note that by (\ref{aniso_lem1}), we immediately get $\sum_{i}|v_i|^2 \left(G_{ii} - \Pi_{ii}\right)\prec\Phi$. Hence it remains to prove that
\begin{equation*}
\sum_{i\ne j} \bar v_i v_j G_{ij} \prec \Phi.
\end{equation*}
By Markov's inequality, it suffices to show that
\begin{equation}\label{Markov_iso}
\mathbb E\Big|\sum_{i\ne j} \bar v_i v_j G_{ij}\Big|^{2p} \prec \Phi^{2p}
\end{equation}
for any fixed $p\in \mathbb N$. The proof of (\ref{Markov_iso}) is similar to the ones in \cite[Section 5]{isotropic} and \cite[Section 5]{XYY}. The main difference is that in \cite{isotropic,XYY}, the matrix entries are assumed to have arbitrarily high moments, while here we assume that the $X$ entries have finite third moment and support bounded by $q$. In particular, for any fixed $n\ge 3$, we have
\begin{equation}\label{moment_n3}
\mathbb E |X_{i\mu}|^n \prec q^{n-3}N^{-3/2}, \ \ i\in \mathcal I_1, \ \mu \in \mathcal I_2.
\end{equation}
(Note that we have a stronger moment assumption in (\ref{conditionA3}). However, the finite fourth moment condition will not be used in the proof below. We only need the weaker bound (\ref{moment_n3}).) We remark that some of the basic ideas have been illustrated in Section \ref{proof_lem_EG} of the main article.

We first rewrite the product in (\ref{Markov_iso}) as
\begin{align*}
 \Big| {  \sum\limits_{i \ne j} {{\bar v_i} {G_{ij}} {v_j}} } \Big|^{2p} =& \sum\limits_{ {i_k \ne j_k} \in {{\cal I}_1}}\prod\limits_{k = 1}^{p} {{\bar v_{i_k}} {G_{i_kj_k}} {v_{j_k}}}  \cdot \prod\limits_{k = p + 1}^{2p} \overline {{\bar v_{i_k}} {G_{i_kj_k}} {v_{j_k}}}\\
=&\sum_\Gamma\sum_{b_1,...,b_n}^*\prod\limits_{k = 1}^{p} {{\bar v_{\Gamma(i_k)}} {G_{\Gamma(i_k)\Gamma(j_k)}} {v_{\Gamma(j_k)}}}  \cdot \prod\limits_{k = p + 1}^{2p} \overline {{\bar v_{\Gamma(i_k)}} {G_{\Gamma(i_k)\Gamma(j_k)}} {v_{\Gamma(j_k)}}},
\end{align*}
where %(recall the notations in the proof for Lemma \ref{largederivation2}) 
$\Gamma$ ranges over all partitions of the set of labels $\{i_1,...,i_{2p},j_1,...,j_{2p}\}$ with the restriction that $i_k, j_k$ cannot be in the same equivalence class for all $k$, $\{b_1,...,b_n\}$ is the set of equivalence classes for a fixed $\Gamma$, $\Gamma(\cdot)$ is regarded as a mapping from the set of labels to the set of equivalence classes, and $\sum^*$ denotes the summation subject to the condition that $b_1,\ldots, b_n$ all take distinct values and $\Gamma(i_k) \ne \Gamma(j_k)$ for all $k$. Since the number of such partitions $\Gamma$ is finite and depends only on $p$, it suffices to prove that for any fixed $\Gamma$,
\begin{equation}\label{temp1}
\mathbb E\sum_{b_1,...,b_n}^*\prod\limits_{k = 1}^{p} {{\bar v_{\Gamma(i_k)}} {G_{\Gamma(i_k)\Gamma(j_k)}} {v_{\Gamma(j_k)}}}  \cdot \prod\limits_{k = p + 1}^{2p} \overline {{\bar v_{\Gamma(i_k)}} {G_{\Gamma(i_k)\Gamma(j_k)}} {v_{\Gamma(j_k)}}}\prec \Phi^{2p}.
\end{equation}
We abbreviate 
$$P(b_1,...,b_n) := \prod\limits_{k = 1}^{p} { { G_{\Gamma(i_k)\Gamma(j_k)}} }  \cdot \prod\limits_{k = p + 1}^{2p} \overline { {G_{\Gamma(i_k)\Gamma(j_k)}} }.$$
For simplicity, we shall omit the overline for complex conjugate in the following proof. In this way, we can avoid a lot of immaterial notational complexities that do not affect the proof.

For $k=1,...,n$, we denote by $\deg(b_k, P)$ the number of times that $b_k$ appears as an index of the $G$ entries in $P$, i.e. $\deg(b_k, P):= |\Gamma^{-1}(b_k)|$. We define $h:=\#\{1\le k \le n: \deg(b_k, P) = 1\}$, i.e. $h$ is the number of $b_k$'s that only appear once in the indices of $P$. Without loss of generality, we assume these $b_k$'s are $b_1,...,b_h$. Then we have the following properties:
\begin{equation}\label{deg1}
\sum_{k=1}^n\deg(b_k,P)=4p, \ \ \text{and} \ \ \deg(b_k,P)=1 \ \text{ for }k=1,...,h.
\end{equation}
Now we claim that
\begin{equation}\label{main_bound}
\left|\mathbb E P\right|\prec N^{-h/2}\Phi^{2p}.
\end{equation}
Note that by $\|\mathbf v\|_2=1$ and Cauchy-Schwarz inequality, we have $\sum_i |v_i| \le \sqrt{M}$ and $\sum_i |v_i|^n \le 1$ for $n\ge 2$. Then if (\ref{main_bound}) holds, we can bound the left hand side of (\ref{temp1}) by 
$$N^{-h/2}\Phi^{2p}\prod_{k=1}^n\left(\sum_{b_k}|v_{b_k}|^{\deg(b_k, P)}\right)\le N^{-h/2}\Phi^{2p}(\sqrt{M})^{h} \le C\Phi^{2p}.$$
%which concludes (\ref{temp1}). 
Hence it suffices to prove (\ref{main_bound}). 

We define the $S$ variables as %(one can compare them with (\ref{Ssymbol}))
\begin{equation}\label{S_var}
S_{ij} := (YG^{(L)}Y^*)_{ij},
\end{equation} 
for $i,j\in\mathcal I_1$ and $L:=\{b_1,...,b_n\}$. As in (\ref{off_imp}) and (\ref{Z_bound}), we can verify that $|S_{ij}-\sigma_i m_{2c}\delta_{ij}|\prec \Phi$ for $i,j\in \mathcal I_1$ using (\ref{aniso_lem1}), (\ref{G_T}) and Lemmas \ref{largederivation}-\ref{largederivation2}. Then as in Section 3.3 of the main article, we keep expanding the $G$ entries in $P$ using the resolvent expansions in Lemma \ref{lemm_resolvent}, until each monomial in the expression either consists of $S$ variables only or has sufficiently many off-diagonal terms. The following lemma %is a generalization of the Lemma 3.8 in the main article, and 
has been proved in \cite[Lemma 5.9]{isotropic} and \cite[Lemma 5.9]{XYY}.

\begin{lem}\label{lem_exp}
After finitely many expansions, we can write $P$ as
\begin{equation}\label{exp_simple}
P = \sum_{\alpha =1}^{A} c_\alpha Q_\alpha +O_{\prec}(N^{-h/2}\Phi^{2p}),
\end{equation}
where $A \in \mathbb N$ depends only on $p$ and $c_1$ (recall that $\Phi(z) \le N^{-c_1}$ by our assumption), $c_\alpha $'s are constants of order $O(1)$, and $Q_\alpha $ are monomials of $S$ variables only, where the number of $S$ variables in each $Q_\alpha$ depends only on $p$ and $c_1$. Moreover, we have that
\begin{equation}\label{parity2}
\deg_o(b_k, Q_\alpha) \ge \deg_o(b_k, P), \ \ \deg_o(b_k, Q_\alpha) = \deg_o(b_k, P) \mod 2,
\end{equation}
for $k=1,...,n$ and $\alpha= 1,...,A$, and the number of off-diagonal $S$ variables in $Q_\alpha$ is at least $2p$. Here $\deg_o(b_k, Q_\alpha)$ denotes the number of times that $b_k$ appears as an index of the \textit{off-diagonal} $S$ variables in $Q_\alpha$, and $\deg_o(b_k, P) := \deg(b_k, P)$ (which is consistent with the previous definition since $P$ only contains off-diagonal entries).
%\begin{equation}
%\text{The number of off-diagonal }S\text{ variables in }Q \ge 2p
%\end{equation}
\end{lem}

Now given the expansion in (\ref{exp_simple}), we see that to conclude (\ref{main_bound}), it suffices to show that for any $Q_\alpha$,
\begin{equation}\label{main_bound2}
\left|\mathbb E Q_\alpha\right| \prec N^{-h/2}\Phi^{2p}. 
\end{equation}
%Thus for the rest of the section we will show (\ref{main_bound2}) for a 
In the following proof, we fixe one such $Q\equiv Q_\alpha$ and write %$Q = \prod_{j=1}^JS_{b_{k_j}b_{l_j}}$.
%for some constant $J$ only depend on $c,p$. 
%With (\ref{S_var}), we have that
\begin{align*}
Q & = \prod_{j=1}^JS_{b_{k_j}b_{l_j}} =\sum_{\substack{\mu_j,\nu_j\in\mathcal I_2}}\prod_{j=1}^J Y_{b_{k_j}\mu_j}G^{(L)}_{\mu_j\nu_j}Y^*_{\nu_j b_{l_j}}\\
&=\sum_W\sum_{w_1,...,w_m}^*\prod_{j=1}^J Y_{b_{k_j}W(\mu_j)}G^{(L)}_{W(\mu_j)W(\nu_j)}Y_{b_{l_j}W(\nu_j) }
\end{align*}
where $J$ is the number of $S$-variables in $Q$, $W$ ranges over all partitions of the set of indices $\{\mu_1,...,\mu_J,\nu_1,...,\nu_J\}$, $\{w_1,...,w_m\}$ denotes the set of equivalence classes for a particular $W$, $W(\cdot)$ is regarded as a symbolic mapping from the set of indices to the set of equivalence classes, and $\sum^*$ denotes the summation subject to the condition that $w_1,\ldots, w_m$ all take distinct values. Note that the number of partitions $W$ depends only on $J$. For a fixed partition $W$, we denote 
$$R(w_1,...,w_m;W):= \prod_{j=1}^J X_{b_{k_j}W(\mu_j)}G^{(L)}_{W(\mu_j)W(\nu_j)}X_{ b_{l_j}W(\nu_j)}.$$
Then to prove (\ref{main_bound2}), it suffices to show that
\begin{equation}\label{main_bound3}
\left|\mathbb E R(w_1,...,w_m;W)\right|\prec N^{-m-h/2}\Phi^{2p}.
\end{equation}
for any partition $W$. 

%The following are some quick observations:
%\begin{align}
%\label{ob0}&\deg(b_k,R) \ge \deg(b_k, P), \text{ for } k=1,...,n,\\
%\label{ob1}&\deg(b_k,R) \equiv \deg(b_k, P)\mod 2, \text{ for } k=1,...,n.
%\end{align}

To facilitate the proof, we introduce the graphical notations as in Section 3.4 of the main article. We use a connected graph $(V,E)$ to represent $R$, where the vertex set $V$ consists of black vertices $b_1,\ldots, b_n$ and white vertices $w_1,\ldots,w_m$, and the edge set $E$ consists of $(k,\alpha)$ edges representing $X_{b_{k}w_{\alpha}}$ and $(\alpha,\beta)$ edges representing $G^{(L)}_{w_{\alpha}w_{\beta}}$. We denote 
$$e_{k\alpha}:=\text{number of }(k, \alpha)\text{ edges in }R, \ \ d_\alpha:=\text{number of }(\alpha,\alpha)\text{ edges in }R.$$
Note that to attain a nonzero expectation, we must have 
\begin{equation}\label{ka3}
e_{k\alpha}=0 \ \text{ or } \  e_{k\alpha}\ge 2 \ \ \text{ for all } k,\alpha.
\end{equation}
We also define 
$$e_{k\alpha}^{(o)}:=\text{number of }(k, \alpha)\text{ edges that are from off-diagonal }S \text{ in } Q.$$
Then we have 
\begin{equation}\label{off_e}\sum_\alpha e_{k\alpha}^{(o)} = \deg_o(b_k, Q)\end{equation}
%$$e_{kl}^{(d)}:=\text{the number of edges }(k, l)\text{ expanded from the diagonal }S$$

By (\ref{deg1}), (\ref{ka3}) and the parity conservation due to (\ref{parity2}), there exist edges $(1,\alpha_1),...,(h,\alpha_h)$ such that $e_{k\alpha_k}$ is odd and $e_{k\alpha_k}\ge 3$, $1\le k \le h$. Let $H:=\{(1,\alpha_1),...,(h,\alpha_h)\}$ be the set of these edges. Denote by $F$ the set of $(k,\alpha)$ edge such that $e_{k\alpha} \ge 2$ and $(k,\alpha)\notin H$. Denote
$$s_\alpha :=\sum_{k=1}^n e_{k\alpha},\ \ h_{k\alpha}:=\mathbf 1_{(k,\alpha)\in H},\ \ h_\alpha:=\sum_{k=1}^n h_{k\alpha},\ \ f_\alpha:=\sum_{k=1}^n\mathbf 1_{(k,\alpha)\in F}$$
for all $k=1,...,n$ and $\alpha=1,...,m$. By the above definitions, we have $s_\alpha \ge 2$ and $h_\alpha +f_\alpha>0$ (since the classes $w_\alpha$ are nonempty), $s_\alpha \ge 2d_\alpha$, and
\begin{equation}\label{sum_h}
\sum_{\alpha}h_{k\alpha} = \mathbf 1(1\le k \le h), \ \ \sum_\alpha h_\alpha=h.
\end{equation}

%Later we will consider a special case when $h_l+f_l=1$, for such $l$, only one $e_{kl}$ is nonzero. We denote
%$$H_1:= \{(k,l)\mid h_l=1, f_l=0, e_{kl}>0\},$$
%$$F_1:= \{(k,l)\mid h_l=0, f_l=1, e_{kl}>0\}.$$
%Note that $H_1\subset H$, $F_1\subset F$.

Note that there are $\frac{1}{2}\sum_{k,\alpha}e_{k\alpha} - \sum_\alpha d_\alpha$ off-diagonal $G$ edges in $R$. Hence by (\ref{aniso_lem1}) and (\ref{moment_n3}), we have
\begin{align*}
|\mathbb E R| &\prec\prod_{\alpha=1}^m\Big(\Phi^{-d_\alpha}\prod_{k=1}^n \Phi^{\frac{1}{2}e_{k\alpha}}\mathbb E |X_{b_kw_\alpha}|^{e_{k\alpha}} \Big)\nonumber\\
&\prec \prod_{\alpha=1}^m\Phi^{s_\alpha/2 - d_\alpha}\Big(\prod_{(k,\alpha)\in H} q^{e_{k\alpha}-3}N^{-3/2}\Big)\Big(\prod_{(k,\alpha)\in F} q^{e_{k\alpha}-2}N^{-1}\Big) =: \prod_{\alpha=1}^m R_\alpha.
\end{align*}
%where  
%$$R_l := \Phi^{\frac{1}{2}s_l-d_l}\prod_{\substack{k=1\\ (k,l)\in H}}^nq^{e_{kl}-3}N^{-3/2}\prod_{\substack{k=1\\ (k,l)\in F}}^nq^{e_{kl}-2}N^{-1}.$$
Now we consider the following four cases for $R_\alpha$.
\begin{itemize}
\item[(i)] $d_\alpha=0$. In this case we have
\begin{align*}
R_\alpha &\prec\Phi^{s_\alpha/2}\prod_{(k,\alpha)\in H} N^{-3/2}\prod_{(k,\alpha)\in F}N^{-1} = \Phi^{s_\alpha/2}(N^{-1})^{h_\alpha + f_\alpha}N^{-h_\alpha/2}\\
&\prec\Phi^{s_\alpha/2}N^{-1}N^{-h_\alpha/2} \prec \Phi^{\sum_{k=1}^h h_{k\alpha}/2+\sum_{k=h+1}^n e_{k\alpha}^{(o)}/2}N^{-1}N^{-h_\alpha/2},
\end{align*}
where in the third step we used $h_l+f_l > 0$, and in the fourth step we used
% $$\Phi^{s_\alpha/2}\prec \Phi^{\sum_ke_{kl}^{(o)}/2}\prec \Phi^{\sum_{k=1}^h h_{kl}/2+\sum_{k=h+1}^n e_{kl}^{(o)}/2},$$ 
$$s_\alpha \ge\sum_k e_{k\alpha}^{(o)} \ge \sum_{k=1}^h h_{k\alpha} +\sum_{k=h+1}^n e_{k\alpha}^{(o)},$$
where we used that $e_{k\alpha}^{(o)}\ge h_{k\alpha}$ for $1\le k \le h$ (recall that if $(k,\alpha)\in H$, then $e_{k\alpha_k}$ is odd and hence at least one of the edges must come from the off-diagonal $S$).
 
\item[(ii)]$d_\alpha \ne 0$, $h_\alpha =1$ and $f_\alpha=0$. Then there is only one $k$ such that $e_{k\alpha}>0$ and $s_\alpha =e_{k\alpha}$ is odd. Hence we have $s_\alpha /2 \ge d_\alpha + 1/2$ and we can bound $R_\alpha$ as
\begin{align*}
R_\alpha &\prec\Phi^{\frac{1}{2}s_\alpha-d_\alpha}(N^{-1})^{h_\alpha+f_\alpha}N^{-h_\alpha/2}\prec \Phi^{1/2}N^{-1}N^{-h_\alpha/2} \\
&= \Phi^{\sum_{k=1}^h h_{k\alpha}/2+\sum_{k=h+1}^n e_{k\alpha}^{(o)}/2}N^{-1}N^{-h_\alpha/2},
\end{align*}
where in the last step we used
$$1 = \sum_{k=1}^h h_{k\alpha} +\sum_{k=h+1}^n e_{k\alpha}^{(o)} , $$
since all the summands except one $h_{k\alpha}$ are $0$. 

\item[(iii)]$d_\alpha\ne 0$, $h_\alpha=0$ and $f_\alpha=1$. Then there is only one $k$ such that $e_{k\alpha}>0$ and $s_\alpha=e_{k\alpha}$. Thus the  $(\alpha,\alpha)$ edges are expanded from the diagonal $S$ variables (otherwise $\alpha$ must connect to at least two different $k$'s), which implies $\frac{1}{2}s_\alpha - d_\alpha = \frac{1}{2}e_{k\alpha}^{(o)}$. Then we can bound $R_\alpha$ by
\begin{align*}
R_\alpha & \prec\Phi^{\frac{1}{2}s_\alpha - d_\alpha}(N^{-1})^{h_\alpha +f_\alpha }N^{-h_\alpha /2}= \Phi^{\sum_ke_{k\alpha}^{(o)}/2}N^{-1}N^{-h_\alpha/2} \\
& \prec \Phi^{\sum_{k=1}^h h_{k\alpha}/2+\sum_{k=h+1}^n e_{k\alpha}^{(o)}/2}N^{-1}N^{-h_\alpha/2} 
\end{align*}
where, as in Case (i), we used $e_{k\alpha}^{(o)}\ge h_{k\alpha}$ for $1\le k \le h$.

\item[(iv)]$d_\alpha\ne 0$ and $h_\alpha+f_\alpha \ge 2$. Then using $s_\alpha \ge 2d_\alpha$, $q\prec\Phi^{1/2}$ and $N^{-1/2}\prec\Phi$, we get that
\begin{align*}
R_\alpha &\prec\prod_{(k,\alpha)\in H} \Phi^{e_{k\alpha}/2-3/2}N^{-3/2}\prod_{ (k,\alpha)\in F } \Phi^{e_{k\alpha}/2-1}N^{-1} \\
& \prec\prod_{(k,\alpha)\in H} \Phi^{e_{k\alpha}/2-1/2}N^{-1}\prod_{ (k,\alpha)\in F } \Phi^{e_{k\alpha}/2}N^{-1/2}\\
&= \Phi^{(s_\alpha - h_\alpha)/2}N^{- (h_\alpha+f_\alpha)/2}N^{-h_\alpha/2} \le \Phi^{(s_\alpha-h_\alpha)/2}N^{-1}N^{-h_\alpha/2} \\
& \le \Phi^{\sum_{k=1}^h h_{k\alpha}/2+\sum_{k=h+1}^n e_{k\alpha}^{(o)}/2}N^{-1}N^{-h_\alpha/2}
\end{align*}
where in the last step we used the definitions of $s_\alpha$ and $h_\alpha$, $e_{k\alpha} \ge 2h_{k\alpha}$ for $1\le k \le h$ (since $e_{k\alpha}\ge 3$ whenever $h_{k\alpha}=1$), and $h_{k\alpha}=0$ for $k\ge h+1$.
\end{itemize}

%Actually, we can relax the bounds a bit, for $l$ in Case (i) we relax that
% $$\Phi^{s_l/2}\prec \Phi^{\sum_ke_{kl}^{(o)}/2}\prec \Phi^{\sum_{k=1}^h h_{kl}/2+\sum_{k=h+1}^n e_{kl}^{(o)}/2},$$ 
% where we used that $e_{kl}^{(o)}\ge h_{kl}$. 
 
% For $l$ in Case (iv) we relax that $$\Phi^{(s_l-h_l)/2}\prec \Phi^{\sum_{k=1}^h h_{kl}/2+\sum_{k=h+1}^n e_{kl}^{(o)}/2},$$
% where we used that $e_{kl}-h_{kl}\ge h_{kl}$ (because if $h_{kl}=1$ then $e_{kl}\ge 3$) and $h_{kl}=0$ for $k>h$.

%It also should be clear that for $l$ in Case (ii) we can write 
%$$\Phi^{1/2}= \Phi^{\sum_{k=1}^h h_{kl}/2+\sum_{k=h+1}^n e_{kl}^{(o)}/2}$$
%since except one $h_{kl}$ all the other summands are just $0$. 

%Similarly for $l$ in Case (iii) we can write  
%$$ \Phi^{\sum_ke_{kl}^{(o)}/2}=\Phi^{\sum_{k=1}^h h_{kl}/2+\sum_{k=h+1}^n e_{kl}^{(o)}/2}$$

Combining the above four cases, we obtain that 
\begin{align*}
|\mathbb ER| \prec \prod_{\alpha=1}^m R_\alpha \prec N^{-m}N^{-\frac{1}{2}\sum_\alpha h_\alpha} \Phi^{\sum_\alpha \left(\sum_{ k=1}^h h_{k\alpha}/2+\sum_{k=h+1}^n e_{k\alpha}^{(o)}/2\right)}.
%&=N^{-m}N^{-\frac{1}{2}h}\Phi^{\sum_{k=1}^h\sum_{l=1}^m h_{kl}/2+\sum_{k=h+1}^n\sum_{l=1}^m e_{kl}^{(o)}/2}.
\end{align*}
Recall that $\sum_\alpha h_\alpha = h$. Then to prove (\ref{main_bound3}), it remains to show that
\begin{equation}\label{p_toshow}
\sum_\alpha\left(\sum_{ k=1}^h h_{k\alpha} +\sum_{k=h+1}^n e_{k\alpha}^{(o)}\right) \ge 4p.
\end{equation}
For $k=1,...,h$, using (\ref{sum_h}) and (\ref{deg1}) we get that
$$\sum_{\alpha=1}^m h_{k\alpha} = 1 = \deg(b_k,P).$$
For $k=h+1,...,n$, using (\ref{off_e}) and (\ref{parity2}) we get that
$$\sum_{\alpha=1}^m e_{k\alpha}^{(o)} = \deg_o(b_k, Q)\ge \deg(b_k,P) .$$
With (\ref{deg1}), we then conclude (\ref{p_toshow}), which finishes our proof.

\subsection{Proof of Theorem \ref{thm_large}}\label{Section_comparison}
In this subsection, we prove Theorem \ref{thm_large}. By Lemma \ref{comp_claim}, we can assume that the entries of $X$ are centered without loss of generality. 

Our main strategy for the proof is a resolvent comparison method that was developed in \cite[Section 6]{LY}. Given $X$ satisfying the assumptions in Theorem \ref{thm_large}, we first construct a random matrix $\tilde X$ whose entries have the same first four moments as those of $X$ but have size of order $N^{-1/2}$.

\begin{lem} [Lemma 5.1 of \cite{LY}]\label{lem_decrease}
Suppose $X$ satisfies the assumptions in Theorem \ref{thm_large}. Then there exists another matrix $\tilde{X}=(\tilde X_{i\mu})$ such that $\mathbb P(\max_{i,\mu}|\tilde X_{i\mu}| \le CN^{-1/2}) = 1$ for some constant $C>0$ and the first four moments of the entries of $X$ and $\tilde{X}$ match, i.e.
\begin{equation}\label{match_moments}
\mathbb EX_{i\mu}^k =\mathbb E\tilde X_{i\mu}^k, \ \ k=1,2,3,4.
\end{equation}
\end{lem}

Taking $q=N^{-1/2}$ in (\ref{Eaniso_law0}), we see that (\ref{Eaniso_law}) holds for $G(\tilde X,z)$. Then due to (\ref{match_moments}), we expect that $G(X,z)$ has ``similar" properties as $G(\tilde X,z)$, so that (\ref{Eaniso_law}) also holds for $G(X,z)$. This will be proved through a resolvent comparison approach that is developed in \cite[Sections 6]{LY} and \cite[Section 6]{NeceSuff_sample}. More specifically, we will apply the Lindeberg replacement strategy, i.e., we change $\tilde X$ to $X$ entry by entry and show that the error (due to the resolvent expansion) appears at each step is negligible. In this subsection, we introduce some notations that will simplify the presentation of the proof.

Fix a bijective ordering map $\Phi$ on the index set of $X$,
\begin{equation*}
\Phi: \{(i,\mu): i \in \mathcal I_1, \mu \in \mathcal I_2 \} \rightarrow \{1,\ldots,\gamma_{\max}=MN\}.
\end{equation*} 
For any $1\le \gamma \le \gamma_{\max}$, we define the matrix $X^{\gamma}= (X^{\gamma}_{i\mu})$ such that $X_{i\mu}^{\gamma} =X_{i\mu} $ if $\Phi(i,\mu)\leq \gamma$, and $X_{i\mu}^{\gamma} =\tilde{X}_{i\mu}$ otherwise. Note that $X^0=\tilde X$, $X^{\gamma_{\max}}=X$, and $X^\gamma$ has bounded support $q\le N^{-\phi}$ for all $0\le \gamma \le \gamma_{\max}$. Correspondingly, we define
 \begin{equation}\label{Hgamma}
   H^{\gamma} := \left( {\begin{array}{*{20}c}
   { 0 } & Y^\gamma  \\
   {(Y^\gamma)^*} & {0}  \\
   \end{array}} \right), \ \ G^\gamma:= \left( {\begin{array}{*{20}c}
   { - I_{M\times M}} & Y^\gamma  \\
   {(Y^\gamma)^*} & { - zI_{N\times N}}  \\
\end{array}} \right)^{-1},
 \end{equation}
 where $Y^\gamma: =\Sigma^{1/2}X^\gamma$. Note that $H^{\gamma}$ and $H^{\gamma-1}$ differ only at $(i,\mu)$ and $(\mu,i)$ entries, where $\Phi(i,\mu) = \gamma$. Then we define two $\mathcal I\times \mathcal I$ matrices $V$ and $W$ by
$$V_{ab}= \sqrt{\sigma_i}(\delta_{ai}\delta_{b\mu} + \delta_{a\mu}\delta_{bi})X_{i\mu}, \ \ W_{ab}=\sqrt{\sigma_i}(\delta_{ai}\delta_{b\mu} + \delta_{a\mu}\delta_{bi}) \tilde X_{i\mu},$$
such that $H^{\gamma}$ and $H^{\gamma-1}$ can be written as
\begin{equation}\label{Lind_H}
H^\gamma= Q + V, \ \ H^{\gamma-1} = Q+W,
\end{equation}
for some $\mathcal I\times \mathcal I$ matrix $Q$ satisfying $Q_{i\mu}=Q_{\mu i}=0$. 

For simplicity, for any $1\le \gamma \le \gamma_{\max}$, we denote the resolvents by
\begin{equation}
S^\gamma :=G^\gamma, \  \  T^\gamma :=G^{\gamma-1}, \ \ R^\gamma :=\left(Q-\left( {\begin{array}{*{20}c}
   {  I_{M\times M}} & 0  \\
   {0} & { z I_{N\times N}}  \\
\end{array}} \right)\right)^{-1}.  \label{R}
\end{equation}
We often omit the superscript if $\gamma$ is fixed. By (\ref{Lind_H}), we can write
\begin{align}
S= \left(Q -\left( {\begin{array}{*{20}c}
   {  I_{M\times M}} & 0  \\
   {0} & { z I_{N\times N}}  \\
\end{array}} \right) + V\right)^{-1}=(1+RV)^{-1}R.  \label{RESOLVENT}
\end{align}
Thus we can expand $S$ using the resolvent expansion
\begin{equation}
S=R-RVR+(RV)^2R+\ldots+(-1)^m(RV)^m R+(-1)^{m+1}(RV)^{m+1}S. \label{RESOLVENTEXPANSION}
\end{equation}
On the other hand, we can also expand $R$ in terms of $S$:
\begin{equation}
R=(1-SV)^{-1}S=S+SVS+(SV)^2 S+\ldots+ (SV)^m S+ (SV)^{m+1}R. \label{RESOLVENTEXPANSION2}
\end{equation}
We can get similar expansions for $T$ and $R$ by replacing $V$, $S$ with $W$, $T$ in (\ref{RESOLVENTEXPANSION}) and (\ref{RESOLVENTEXPANSION2}). 

By the bounded support conditions for $X$ and $\tilde X$, we have
\begin{equation}
\max_{a,b\in \mathcal I} \vert V_{ab} \vert =\sqrt{\sigma_i} |X_{i\mu}| \prec N^{-\phi}, \ \ \max_{a,b\in \mathcal I} \vert W_{ab} \vert = \sqrt{\sigma_i}|\tilde X_{i\mu}| \le C N^{-1/2}. \label{5BOUND2}
\end{equation}
Note that $S$, $R$ and $T$ satisfy the following deterministic bounds by (\ref{eq_gbound}):
\begin{equation}
\sup_{z\in \mathbf D} \max_{\gamma} \max \left\{ \| S^\gamma \|, \|T^\gamma \|, \|R^\gamma \| \right\}\le \sup_{z\in \mathbf D} (C\eta^{-1}) \le N. \label{5BOUNDT}
\end{equation}
Then using expansion (\ref{RESOLVENTEXPANSION2}) in terms of $T,W$ with $m=3$, the isotropic local law (\ref{aniso_law}) for $T$, and the bound (\ref{5BOUNDT}) for $R$, we can get that for any fixed unit vectors $\mathbf u,\mathbf v \in \mathbb C^{\mathcal I}$, $|R_{\mathbf u\mathbf v}| =O(1)$ with high probability. Thus there exists a uniform constant $C_1>0$ such that with high probability, 
\begin{equation}
\sup_{z\in \mathbf D} \max_{\gamma} \sup_{\text{deterministic unit } \mathbf u,\mathbf v} \max \left\{ | S^\gamma_{\mathbf u \mathbf v} |, |T^\gamma_{\mathbf u \mathbf v}|, |R^\gamma_{\mathbf u \mathbf v}| \right\}\leq C_1. \label{5BOUND1}
\end{equation}

From the definitions of $V$ and $W$, one can see that it is helpful to introduce the following notations to simplify the expressions. %They are introduced in \cite[Definitions 6.2-6.3]{LY} for Wigner matrices.

\begin{defn}[Matrix operators $*_\gamma$] \label{def_operator1}
For $\mathcal I \times \mathcal I$ matrices $A$ and $B$, we define $A *_{\gamma} B$ as
\begin{equation}
(A*_{\gamma}B)_{ab}=A_{a i}B_{\mu b}+A_{a\mu}B_{ib}, \ \ \Phi(i,\mu)=\gamma.
\end{equation} 
We denote the $m$-th power of $A$ under $*_\gamma$-product by $A^{*_\gamma m}$, i.e.,
\begin{equation}
A^{*_\gamma m}:=\underbrace{A*_\gamma A*_\gamma A*_\gamma \ldots*_\gamma A}_{m}.
\end{equation}
%We sometimes drop the subscript $\gamma$ and write $A*B$, $A^{* m}$ for convenience. 
\end{defn}

\begin{defn} [$\mathcal{P}_{\gamma, \mathbf{k}}$ and $\mathcal{P}_{\gamma,k}$] \label{def_operator2}
For $k\in \mathbb N$, $\mathbf{k}=(k_1, \cdots, k_s) \in \mathbb{N}^s$ and $1\le  \gamma \le \gamma_{\max}$, we define 
\begin{equation}
\mathcal{P}_{\gamma, k} G_{\mathbf u\mathbf v} := G_{\mathbf u\mathbf v}^{*_{\gamma}(k+1)}, \ \ \mathcal{P}_{\gamma,\mathbf{k}}\left(\prod_{t=1}^s G_{\mathbf u_t \mathbf v_t}\right):=\prod_{t=1}^s \mathcal{P}_{\gamma, k_t} G_{\mathbf u_t \mathbf v_t},
\end{equation}
where we abbreviate $G_{\mathbf u\mathbf v}^{*_{\gamma}(k+1)} \equiv (G^{*_{\gamma}(k+1)})_{\mathbf u\mathbf v}$. If $\mathfrak G_1$ and $\mathfrak G_2$ are products of resolvent entries as above, then we define
\begin{equation}
 \mathcal{P}_{\gamma,\mathbf{k}} (\mathfrak G_1+\mathfrak G_2):= \mathcal{P}_{\gamma,\mathbf{k}}\mathfrak G_1 +  \mathcal{P}_{\gamma,\mathbf{k}}\mathfrak G_2.
\end{equation}
Note that $ \mathcal{P}_{\gamma, k}$ and $ \mathcal{P}_{\gamma,\mathbf{k}}$ are not linear operators, but just notations we use for simplification.
\end{defn}

Using Definition \ref{def_operator2}, we may write, for example,
$$\mathcal{P}_{\gamma,\mathbf{k}}\left(\prod_{t=1}^s G^\gamma_{\mathbf u_t \mathbf v_t}\right):=\prod_{t=1}^s S^{*_\gamma (k_t+1)}_{\mathbf u_t \mathbf v_t}, \ \ \mathcal{P}_{\gamma,\mathbf{k}}\left(\prod_{t=1}^s G^{\gamma-1}_{\mathbf u_t \mathbf v_t}\right):=\prod_{t=1}^s T^{*_\gamma (k_t+1)}_{\mathbf u_t \mathbf v_t}.$$
For $k, s \in \mathbb{N}$ and $\mathbf{k} \in \mathbb{N}^{s+1}$, it is easy to verify that
\begin{equation}
G^{*_\gamma s} *_\gamma  G^{*_\gamma k}=G^{*_\gamma (s+k)}, \ \ \ \mathcal{P}_{\gamma, \mathbf{k}}(\mathcal{P}_{{\gamma},s}G_{\mathbf u\mathbf v})=\mathcal{P}_{\gamma, s+ \vert \mathbf{k} \vert} G_{\mathbf u\mathbf v} ,\label{FIRST}
\end{equation}
where $\vert \mathbf{k} \vert=\sum_{t=1}^s k_t$. For the second equality, note that $\mathcal{P}_{{\gamma},s}G_{\mathbf u \mathbf v}$ is a sum of the products of the $G$ entries, where each product contains $s+1$ entries.

\begin{proof}[Proof of Theorem \ref{thm_large}]
Now we prove (\ref{Eaniso_law}) with the resolvent comparison method. The basic idea is that we expand $S$ and $T$ in terms of $R$ by repeatedly applying the expansions (\ref{RESOLVENTEXPANSION}) and (\ref{RESOLVENTEXPANSION2}), and then compare the resulting expressions. The main terms will cancel since $X_{i\mu}$ and $\tilde X_{i\mu}$ have the same first four moments, and the error terms are small since $X_{i\mu}$ and $\tilde X_{i\mu}$ have support bounded by $N^{-\phi}$.

%With the isotropic local law (\ref{aniso_law}), one can prove the following Lemma \ref{Greenfunctionrepresent}. 
The proof of Lemma \ref{Greenfunctionrepresent} is almost the same as the one for \cite[Lemma 6.5]{LY}. In fact, we can copy their arguments almost verbatim, except for some notational differences. Hence we omit the details. In the following expressions, for any $\mathbf k=(k_1, \ldots, k_p)\in \mathbb N^p$, we use $|\mathbf k| = \sum k_i$ to denote its $l^1$-norm. %Also we denote by $\mathbb S^{\mathcal I_1}$ the unit sphere of $\mathbb C^{\mathcal I_1}$.

\begin{lem} \label{Greenfunctionrepresent} %[Resolvent representation lemma] 
Suppose $z\in \mathbf D$ and $\gamma=\Phi(i,\mu)$. Fix any $p\in \mathbb N$ and $r>0$. Then for $S,R$ in (\ref{R}), we have 
\begin{equation}\label{ONLYPROVE}
\begin{split}
\mathbb{E}\prod_{t=1}^p S_{\mathbf u_t \mathbf v_t} & = \sum_{0 \leq k \leq 4} A_{k} \mathbb{E}\left[(-\sqrt{\sigma_i}X_{i\mu})^{k}\right] \\
&+\sum_{5 \leq \vert \mathbf{k} \vert \leq {r}/{\phi}, \mathbf k\in \mathbb N^p } \mathcal{A}_{  \mathbf{k}  }\mathbb{E} \, \mathcal{P}_{\gamma,\mathbf{k} }\prod_{t=1}^p S_{\mathbf u_t \mathbf v_t}+O_\prec(N^{-r}), 
\end{split}
\end{equation}
%and
%\begin{equation}\label{ONLYPROVE2}
%\begin{split}
%& \mathbb{E}\prod_{t=1}^s (S-\Pi)_{\mathbf u_t\mathbf v_t} = \sum_{0 \leq k \leq 4} A_{k} \mathbb{E}\left[(-X_{i\mu})^{k}\right] \\
%&\qquad \qquad \quad +\sum_{5 \leq \vert \mathbf{k} \vert \leq {r}/{\phi}, \mathbf k \in \mathbb N^s} \mathcal{A}_{  \mathbf{k}  }\mathbb{E} \, \mathcal{P}_{\gamma,\mathbf{k}}\prod_{t=1}^s (S-\Pi)_{\mathbf u_t\mathbf v_t}+O_\prec(N^{-r}), 
%\end{split} 
%\end{equation}
where $A_k$, $0\le k \le 4$, depend only on $R$, $\mathcal A_{\mathbf k}$'s are independent of $(\mathbf u_t, \mathbf v_t)$, $1\leq t \leq s$, and we have the bound
\begin{equation}
\vert \mathcal{A}_{  \mathbf{k} } \vert \prec N^{-{\vert \mathbf{k} \vert  \phi} /{10}-2} . \label{INDEXBOUND}
\end{equation}
%Similarly, we have
%\begin{equation}
%\mathbb{E}\prod_{t=1}^s (S-\Pi)_{a_tb_t} = \sum_{0 \leq k \leq 4} \tilde{A}_{k} \mathbb{E}\left[(-x_{i\mu})^{k}\right]+\sum_{5 \leq \vert \mathbf{k} \vert \leq {2 \omega}/{\phi}, \mathbf k \in \mathbb Z_+^s} \mathcal{A}_{  \mathbf{k}  }\mathbb{E} \, \mathcal{\tilde{P}}_{\gamma,\mathbf{k}}\prod_{t=1}^s S_{a_t,b_t}+O(N^{-\omega}),  \label{ONLYPROVE2}
%\end{equation}
%where $\tilde{A}_k$, $0\leq k \leq 4$, again depend only on $R$.
\end{lem}

It is obvious that a result similar to Lemma \ref{Greenfunctionrepresent} also holds for the product of $T$ entries. As in (\ref{ONLYPROVE}), we define the notation $\mathcal{A}^{\gamma,a}$, $a=0,1$ as follows:
%\begin{equation} \label{greenpowers}
%\mathbb{E}S_{\mathbf u\mathbf v} = \sum_{0 \leq k \leq 4} A_{k} \mathbb{E}\left[(-X_{i\mu})^{k}\right]+\sum_{5 \leq k \leq {r}/{\phi}} \mathcal{A}_{k}^{\gamma, 0}\mathbb{E} \, \mathcal{P}_{\gamma,k} S_{\mathbf u\mathbf v}+O_\prec(N^{-r}),
%\end{equation}
%\begin{equation} \label{greenpowert}
%\mathbb{E}T_{\mathbf u\mathbf v} = \sum_{0 \leq k \leq 4} A_{k} \mathbb{E}\left[(-\tilde X_{i\mu})^{k}\right]+\sum_{5 \leq k \leq {r}/{\phi}} \mathcal{A}_{k}^{\gamma, 1}\mathbb{E} \, \mathcal{P}_{\gamma,k} T_{\mathbf u\mathbf v}+O_\prec(N^{-r}).
%\end{equation}
\begin{equation} \label{greenpowers}
\begin{split}
&\mathbb{E}\prod_{t=1}^p S_{\mathbf u_t\mathbf v_t} = \sum_{0 \leq k \leq 4} A_{k} \mathbb{E}\left[(-\sqrt{\sigma_i}X_{i\mu})^{k}\right] \\
&\qquad \qquad \quad+\sum_{5 \leq \vert \mathbf{k} \vert \leq {r}/{\phi}, \mathbf k\in \mathbb N^p } \mathcal{A}_{  \mathbf{k}}^{\gamma, 0}\mathbb{E} \, \mathcal{P}_{\gamma,\mathbf{k} }\prod_{t=1}^p S_{\mathbf u_t\mathbf v_t}+O_\prec(N^{-r}),
\end{split}
\end{equation}
\begin{equation} \label{greenpowert}
\begin{split}
&\mathbb{E}\prod_{t=1}^p T_{\mathbf u_t\mathbf v_t} = \sum_{0 \leq k \leq 4} A_{k} \mathbb{E}\left[(-\sqrt{\sigma_i}\tilde{X}_{i\mu})^{k}\right]\\
&\qquad \qquad \quad +\sum_{5 \leq \vert \mathbf{k} \vert \leq {r}/{\phi}, \mathbf k\in \mathbb N^p } \mathcal{A}_{  \mathbf{k}}^{\gamma, 1}\mathbb{E} \, \mathcal{P}_{\gamma,\mathbf{k} }\prod_{t=1}^p T_{\mathbf u_t\mathbf v_t} +O_\prec(N^{-r}).
\end{split}
\end{equation}
Since $A_{k}$, $0 \leq k \leq 4$, depend only on $R$ and $X_{i \mu}$, $\tilde{X}_{i \mu}$ have the same first four moments, we get from (\ref{greenpowers}) and (\ref{greenpowert}) that
%\begin{equation} \label{telescoping_2}
%\begin{split}
%& \mathbb{E}(G-\tilde G)_{\mathbf u\mathbf v}=\sum_{\gamma=1}^{\gamma_{\max}} \left(\mathbb{E}G^\gamma_{\mathbf u\mathbf v} -\mathbb{E}G^{\gamma-1}_{\mathbf u\mathbf v}\right) \\
%& = \sum_{\gamma=1}^{\gamma_{\max}} \sum_{5 \leq k \leq {r}/{\phi}}\left(\mathcal{A}^{\gamma,0}_{k}\mathbb{E} \,\mathcal{P}_{\gamma,k}G^\gamma_{\mathbf u\mathbf v} -\mathcal{A}_{k}^{\gamma,1}\mathbb{E} \, \mathcal{P}_{\gamma,k}\prod_{t=1}^s G^{\gamma-1}_{\mathbf u \mathbf v}\right)+O_\prec(N^{-r}),
%\end{split}
%\end{equation}
%\begin{equation} \label{telescoping_2}
%\begin{split}
%& \mathbb{E}\prod_{t=1}^s G-\Pi)_{\mathbf u_t\mathbf v_t} -\mathbb{E}\prod_{t=1}^s (\tilde G-\Pi)_{\mathbf u_t\mathbf v_t}\\
%= &\sum_{\gamma=1}^{\gamma_{\max}} \left(\mathbb{E}\prod_{t=1}^s (G^\gamma-\Pi)_{\mathbf u_t\mathbf v_t} -\mathbb{E}\prod_{t=1}^s (G^{\gamma-1}-\Pi)_{\mathbf u_t\mathbf v_t}\right) \\
%= &\sum_{\gamma=1}^{\gamma_{\max}} \sum_{\mathbf k\in \mathbb N^s}^{5 \leq \vert \mathbf{k} \vert \leq {r}/{\phi}}\left(\mathcal{A}^{\gamma,0}_{  \mathbf{k}  }\mathbb{E} \,\mathcal{P}_{\gamma,\mathbf{k} }\prod_{t=1}^s (G^\gamma-\Pi)_{\mathbf u_t\mathbf v_t} -\mathcal{A}_{  \mathbf{k}}^{\gamma,1}\mathbb{E} \, \mathcal{P}_{\gamma,\mathbf{k}}\prod_{t=1}^s (G^{\gamma-1}-\Pi)_{\mathbf u_t\mathbf v_t}\right) \\
%& +O_\prec(N^{-r}),
%\end{split}
%\end{equation}
\begin{equation} \label{telescoping_2}
\begin{split}
& \mathbb{E}\prod_{t=1}^p G_{\mathbf u_t\mathbf v_t} -\mathbb{E}\prod_{t=1}^p \tilde G_{\mathbf u_t\mathbf v_t} =\sum_{\gamma=1}^{\gamma_{\max}} \left(\mathbb{E}\prod_{t=1}^p G^\gamma_{\mathbf u_t\mathbf v_t} -\mathbb{E}\prod_{t=1}^p G^{\gamma-1}_{\mathbf u_t\mathbf v_t}\right) \\
& = \sum_{\gamma=1}^{\gamma_{\max}} \sum_{\mathbf k\in \mathbb N^p}^{5 \leq \vert \mathbf{k} \vert \leq {r}/{\phi}}\left(\mathcal{A}^{\gamma,0}_{  \mathbf{k}  }\mathbb{E} \,\mathcal{P}_{\gamma,\mathbf{k} }\prod_{t=1}^p G^\gamma_{\mathbf u_t\mathbf v_t} -\mathcal{A}_{  \mathbf{k}}^{\gamma,1}\mathbb{E} \, \mathcal{P}_{\gamma,\mathbf{k}}\prod_{t=1}^p G^{\gamma-1}_{\mathbf u_t\mathbf v_t}\right) \\
&\quad +O_\prec(N^{-r+2}).
\end{split}
\end{equation}
where we abbreviate $G:=G(X,z)$ and $\tilde G:=G(\tilde X,z)$.% as in (\ref{tele_bound1}). 
%Through a similar argument, we also get that

Applying (\ref{telescoping_2}) with $p=1$, $r=3$ and fixed unit vector $\mathbf u_t = \mathbf v_t = \mathbf v\in \mathbb C^{\mathcal I_1} \ \text{or} \ \mathbb C^{\mathcal I_2}$, we obtain that
\begin{equation} \label{teles_inequality}
\begin{split}
&\left| \mathbb{E} (G-\tilde G)_{\mathbf v\mathbf v} \right| \leq  \sum_{\gamma=1}^{\gamma_{\max}} \sum_{a=0,1} \sum_{5 \leq k \leq {3}/{\phi}} \vert \mathcal{A}_{k}^{\gamma, a}  \vert \left\vert  \mathbb{E} \mathcal{P}_{\gamma,k} G^{\gamma - a}_{\mathbf v\mathbf v} \right\vert+O_\prec(N^{-1}).
\end{split}
\end{equation}
%\begin{equation} \label{teles_inequality}
%\begin{split}
%&\left\vert \mathbb{E}\prod_{t=1}^s (G^{\gamma_{\max}}-\Pi)_{\mathbf u_t\mathbf v_t} \right\vert \leq \left\vert \mathbb{E}\prod_{t=1}^s (G^{0}-\Pi)_{\mathbf u_t\mathbf v_t} \right\vert \\
%&+ \sum_{\gamma=1}^{\gamma_{\max}} \sum_{a=0,1} \sum_{\mathbf k\in \mathbb N^s}^{5 \leq \vert \mathbf{k} \vert \leq {r}/{\phi}} \vert \mathcal{A}_{\mathbf{k}}^{\gamma, a}  \vert \left\vert  \mathbb{E} \mathcal{P}_{\gamma,\mathbf{k}}\prod_{t=1}^s (G^{\gamma - a}-\Pi)_{\mathbf u_t\mathbf v_t} \right\vert+O_\prec(N^{-r}).
%\end{split}
%\end{equation}
Using (\ref{5BOUND1}) and (\ref{INDEXBOUND}), we can bound the sum in (\ref{teles_inequality}) by 
\begin{equation}
 \sum_{\gamma=1}^{\gamma_{\max}} \sum_{a=0,1}  \sum_{5 \leq k \leq 3/{\phi}} \vert \mathcal{A}_{k}^{\gamma, a} \vert  \left\vert \mathbb{E} \mathcal{P}_{\gamma,k} G^{\gamma - a}_{\mathbf v\mathbf v} \right\vert  \prec \sum_{5 \leq k \leq 3/{\phi}} N^{-{k \phi}/{10}} \prec N^{-\phi/2}. 
\end{equation}
(Here we need to apply the Lemma 2.2 (iii) of the main article, and hence need a second moment bound for $|\mathcal{P}_{\gamma,{k}}G^{\gamma - a}_{\mathbf v\mathbf v} |$. This follows easily from (\ref{5BOUNDT}).) Recall that $\mathcal{P}_{\gamma,k}G^{\gamma-a}_{\mathbf v\mathbf v}$ is also a sum of the products of $G$ entries. Then applying (\ref{telescoping_2}) to $|\mathbb{E} \mathcal{P}_{\gamma,k} G^{\gamma-a}_{\mathbf v\mathbf v}|$ and replacing $\gamma_{\max}$ with $\gamma-a$, we obtain that
\begin{equation}\label{teles_inequality1}
\begin{split}
& \left\vert \mathbb{E} \mathcal P_{\gamma, k} G^{\gamma-a}_{\mathbf v\mathbf v} \right \vert \leq \left\vert \mathbb{E} \mathcal P_{\gamma,k} G^{0}_{\mathbf v\mathbf v} \right \vert  \\
& + \sum_{\gamma^{\prime}=1}^{\gamma-a} \sum_{a^{\prime}=0,1} \sum_{\mathbf k^{\prime}\in \mathbb N^{1+k}}^{5 \leq \vert \mathbf{k}^{\prime} \vert \leq 3/{\phi}} \left\vert \mathcal{A}_{\mathbf{k}^{\prime}}^{\gamma^{\prime}, a^{\prime}} \right\vert  \left\vert \mathbb{E} \mathcal{P}_{\gamma^{\prime},\mathbf{k}^{\prime}}\mathcal{P}_{\gamma,{k}} G^{\gamma^{\prime}-a'}_{\mathbf v \mathbf v} \right\vert +O_\prec(N^{-1}).
\end{split}
\end{equation}
%\begin{equation}\label{teles_inequality1}
%\begin{split}
%& \left\vert \mathbb{E} \mathcal P_{\gamma,\mathbf k}\prod_{t=1}^s (G^{\gamma-a}-\Pi)_{\mathbf u_t\mathbf v_t} \right \vert \leq \left\vert \mathbb{E} \mathcal P_{\gamma,\mathbf k} \prod_{t=1}^s (G^{0}-\Pi)_{\mathbf u_t\mathbf v_t} \right \vert \\
%& + \sum_{\gamma^{\prime}=1}^{\gamma-a} \sum_{a^{\prime}=0,1} \sum_{\mathbf k^{\prime}\in \mathbb N^{s+\vert \mathbf{k} \vert}}^{5 \leq \vert \mathbf{k}^{\prime} \vert \leq {r}/{\phi}} \left\vert \mathcal{A}_{\mathbf{k}^{\prime}}^{\gamma^{\prime}, a^{\prime}} \right\vert  \left\vert \mathbb{E} \mathcal{P}_{\gamma^{\prime},\mathbf{k}^{\prime}}\mathcal{P}_{\gamma,\mathbf{k}}\prod_{t=1}^s (G^{\gamma^{\prime}-a^{\prime}} - \Pi)_{\mathbf u_t \mathbf v_t} \right\vert \\
%& +O_\prec(N^{-r}).
%\end{split}
%\end{equation}
Together with (\ref{teles_inequality}) and (\ref{INDEXBOUND}), we get that
\begin{align*}
& \left\vert \mathbb{E}(G - \tilde G)_{\mathbf v\mathbf v} \right\vert \leq \sum_{\gamma,a} \sum_{k} \left\vert \mathcal{A}_{k}^{\gamma, a} \right\vert  \left\vert \mathbb{E} \mathcal{P}_{\gamma,k}G^{0}_{\mathbf v\mathbf v}\right \vert  \nonumber \\
& + \sum_{\gamma, \gamma^{\prime}} \sum_{a, a^{\prime}} \sum_{k, \mathbf{k}^{\prime}} \left| \mathcal{A}_{k}^{\gamma,a} \mathcal{A}_{\mathbf{k}^{\prime}}^{\gamma^{\prime},a^{\prime}} \right| \left\vert \mathbb{E} \mathcal{P}_{\gamma^{\prime},\mathbf{k}^{\prime}}\mathcal{P}_{\gamma,k} G^{\gamma^{\prime}-a^{\prime}}_{\mathbf v\mathbf v} \right\vert+O_\prec(N^{-1}).
\end{align*}
%\begin{align*}
%& \left\vert \mathbb{E}\prod_{t=1}^s (G^{\gamma_{\max}}-\Pi)_{\mathbf u_t\mathbf v_t} \right\vert \leq  \left\vert \mathbb{E}\prod_{t=1}^s (G^{0}-\Pi)_{\mathbf u_t\mathbf v_t} \right\vert \\
%& + \sum_{\gamma,a} \sum_{\mathbf{k}} \left\vert \mathcal{A}_{\mathbf{k}}^{\gamma, a} \right\vert  \left\vert \mathbb{E} \mathcal{P}_{\gamma,\mathbf{k}}\prod_{t=1}^s (G^{0}-\Pi)_{\mathbf u_t\mathbf v_t}\right \vert  \nonumber \\
%& + \sum_{\gamma, \gamma^{\prime}} \sum_{a, a^{\prime}} \sum_{\mathbf{k}, \mathbf{k}^{\prime}} \left| \mathcal{A}_{\mathbf{k}}^{\gamma,a} \mathcal{A}_{\mathbf{k}^{\prime}}^{\gamma^{\prime},a^{\prime}} \right| \left\vert \mathbb{E} \mathcal{P}_{\gamma^{\prime},\mathbf{k}^{\prime}}\mathcal{P}_{\gamma,\mathbf{k}}\prod_{t=1}^s (G^{\gamma^{\prime}-a^{\prime}}-\Pi)_{\mathbf u_t\mathbf v_t} \right\vert+O(N^{-r}).
%\end{align*}
Again using (\ref{5BOUND1}) and (\ref{INDEXBOUND}), we obtain that
\begin{equation}
\sum_{\gamma, \gamma^{\prime}} \sum_{a, a^{\prime}} \sum_{k, \mathbf{k}^{\prime}} \left\vert \mathcal{A}_{k}^{\gamma,a} \mathcal{A}_{\mathbf{k}^{\prime}}^{\gamma^{\prime},a^{\prime}} \right\vert  \left| \mathbb{E} \mathcal{P}_{\gamma^{\prime},\mathbf{k}^{\prime}}\mathcal{P}_{\gamma,k} G^{\gamma^{\prime}-a^{\prime}}_{\mathbf v\mathbf v} \right| \prec N^{-\phi}, \label{process_end}
\end{equation}
%\begin{equation}
%\sum_{\gamma, \gamma^{\prime}} \sum_{a, a^{\prime}} \sum_{\mathbf{k}, \mathbf{k}^{\prime}} \left\vert \mathcal{A}_{\mathbf{k}}^{\gamma,a} \mathcal{A}_{\mathbf{k}^{\prime}}^{\gamma^{\prime},a^{\prime}} \right\vert  \left| \mathbb{E} \mathcal{P}_{\gamma^{\prime},\mathbf{k}^{\prime}}\mathcal{P}_{\gamma,\mathbf{k}}\prod_{t=1}^s (G^{\gamma^{\prime}-a^{\prime}}-\Pi)_{\mathbf u_t\mathbf v_t} \right| \prec N^{-\phi}, \label{process_end}
%\end{equation}
where we used that $k+|\mathbf k'| \ge 10$. Repeating this process, we can make the remainder term smaller and smaller. At the end, we obtain that
\begin{equation*} %\label{teles_inequality2}
\begin{split}
\left\vert \mathbb{E}(G - \tilde G)_{\mathbf v\mathbf v} \right\vert & \leq \sum_{n=0}^{{2}/{\phi}} \sum_{\gamma_1, \cdots, \gamma_n} \sum_{a_1,\cdots, a_n} \sum_{\mathbf{k}_1,\cdots, \mathbf{k}_n} \Big\vert \prod_{j} \mathcal{A}_{\mathbf{k}_j}^{\gamma_j,a_j} \Big\vert \left\vert \mathbb{E}\mathcal{P}_{\gamma_n, \mathbf{k}_n} \cdots \mathcal{P}_{\gamma_1, \mathbf{k}_1} G^{0}_{\mathbf v\mathbf v} \right\vert \\
& +O_\prec(N^{-1}) , 
\end{split}
\end{equation*} 
%\begin{equation*} %\label{teles_inequality2}
%\begin{split}
%& \left\vert \mathbb{E}\prod_{t=1}^s (G^{\gamma_{\max}}-\Pi)_{\mathbf u_t\mathbf v_t} \right\vert \\
% \leq & \sum_{n=0}^{{3r}/{\phi}} \sum_{\gamma_1, \cdots, \gamma_n} \sum_{a_1,\cdots, a_n} \sum_{\mathbf{k}_1,\cdots, \mathbf{k}_n} \left\vert \prod_{j} \mathcal{A}_{\mathbf{k}_j}^{\gamma_j,a_j} \right\vert \left\vert \mathbb{E}\mathcal{P}_{\gamma_n, \mathbf{k}_n} \cdots \mathcal{P}_{\gamma_1, \mathbf{k}_1} \prod_{t=1}^s (G^{0}-\Pi)_{\mathbf u_t\mathbf v_t} \right\vert \\
%& +O_\prec(N^{-r}) , 
%\end{split}
%\end{equation*} 
where 
\begin{equation} \label{indexassumption}
\mathbf{k}_1 \in \mathbb{N}^1, \  \mathbf{k}_2 \in \mathbb{N}^{1+ \vert \mathbf{k}_1 \vert}, \ \mathbf{k}_3 \in \mathbb{N}^{1+ \vert \mathbf{k}_1 \vert+ \vert \mathbf{k}_2 \vert}, \ \text{etc.},  \ \text{ and } \ 5 \leq \vert \mathbf{k}_i \vert \leq \frac{3}{\phi}.
\end{equation}
Using (\ref{INDEXBOUND}), we obtain that
\begin{equation} \label{teles_inequality3}
\begin{split}
&\left\vert \mathbb{E}(G - \tilde G)_{\mathbf v\mathbf v} \right\vert  \\
& \prec \max_{\mathbf{k}, n}(N^{-2})^n(N^{-\frac{\phi}{10}})^{\sum_{i}\vert \mathbf{k}_i \vert}\sum_{\gamma_1, \cdots, \gamma_n}  \left\vert \mathbb{E}\mathcal{P}_{\gamma_n, \mathbf{k}_n} \cdots \mathcal{P}_{\gamma_1, \mathbf{k}_1} G^{0}_{\mathbf v\mathbf v} \right\vert +N^{-1} .
\end{split}
\end{equation}  

Now we complete the proof of (\ref{Eaniso_law}) using the estimate (\ref{teles_inequality3}) and the bound (\ref{Eaniso_law}) for $G^0=\tilde G$. We see that it suffices to control the terms
\begin{align}\label{average_comparison1}
{\mathcal{P}}_{\gamma_n, \mathbf{k}_n  } \cdots {\mathcal{P}}_{\gamma_1, \mathbf{k}_1}\tilde{G}_{\mathbf v\mathbf v}
\end{align}
for $\mathbf k_1,\ldots, \mathbf k_n$ satisfying (\ref{indexassumption}).
%Recall that (\ref{5BOUNDT}) happens with probability less than $\exp(-c \varphi^{\xi_1})$, where $c>0$ and independent of $N$. By (\ref{KEYBOUNDS}), we have 
%\begin{equation}
%\vert \mathbb{E}(\tilde{m}_2-m_{2c})^p \vert \leq (K+\varphi^{C}N^{-1})^p
%\end{equation} 
%where $C$ is a large enough deterministic number.   
By definition of ${\mathcal P}$, (\ref{average_comparison1}) is a sum of at most $C^{\sum \vert \mathbf{k}_i \vert}$ products of $G_{\mathbf v b}$, $G_{b\mathbf v}$ and $G_{ab}$ entries, where the total number of $G$ entries in each product is at most $\sum {\vert \mathbf{k}_i \vert}+1=O(\phi^{-2})$. Due to the deterministic bound (\ref{5BOUNDT}), (\ref{average_comparison1}) is always bounded by $N^{O(\phi^{-2})}$, and hence Lemma 2.2 (iii) of the main article can be applied. 

For each product in (\ref{average_comparison1}), there are two $\mathbf v$'s in the indices of $G$. These two $\mathbf v$'s appear as $G_{\mathbf v a} G_{b \mathbf v}$ in the product, where $a, b$ come from some $\gamma_k$ and $\gamma_{l}$ ($1 \leq k, l \leq n$) via ${\mathcal{P}}$. %(see Definition \ref{def_operator2}). 
%In case (1), we have 
%$$|\tilde{G}_{\mathbf v\mathbf v}-{\Pi}_{\mathbf v\mathbf v}| \prec \Psi$$ 
%by (\ref{aniso_lawweak}) with $q=N^{-1/2}$. For case (2), we have the estimate
%\begin{align*}
%\Big|\sum_{a\in \mathcal I}G_{\mathbf v a}^{n_1} G_{a \mathbf v}^{n_2}\Big| & \le \prod_{j=1}^2 \Big(\sum_{a\in \mathcal I} |G_{\mathbf v a}|^{2n_j}\Big)^{1/2} \le \prod_{j=1}^2 \Big(\sum_{a\in \mathcal I} |G_{\mathbf v a}|^{2}\Big)^{n_j/2} \\
%& \prec N^{(n_1+n_2)/2}\Psi^{n_1 + n_2},
%\end{align*}
%\begin{align*}
%\Big|\sum_{a\in \mathcal I}G_{\mathbf v a}^{n_1} G_{a \mathbf v}^{n_2}\Big| \prec N^{(n_1+n_2)/2}\Psi^{n_1 + n_2} , \ \ n_1+n_2\ge 1,
%\end{align*}
%using Cauchy-Schwarz inequality and the bound (\ref{sum_bound}). 
Thus after taking the average $N^{-2}\sum_{\gamma_k}$ and $N^{-2}\sum_{\gamma_l}$, the term $G_{\mathbf v a} G_{b \mathbf v}$ contributes a factor $O_\prec((N\eta)^{-1})$ by (\ref{sum_bound}) and Cauchy-Schwarz inequality. For all other $G$ factors in the product with no $\mathbf v$'s, we control them by $O_\prec(1)$ using (\ref{5BOUND1}). Thus for any fixed $\gamma_1,\ldots,\gamma_n$, $\mathbf k_1,\ldots, \mathbf k_n$, we have proved that
\begin{equation*}
N^{-2n}\sum_{\gamma_1, \cdots, \gamma_n}  \left\vert \mathbb{E} {\mathcal{P}}_{\gamma_n, \mathbf{k}_n  } \cdots {\mathcal{P}}_{\gamma_1, \mathbf{k}_1} \tilde{G}_{\mathbf v\mathbf v}\right\vert \prec \frac{1}{N\eta} . %\label{KEY444}
\end{equation*}
Then using (\ref{teles_inequality3}) and (\ref{Eaniso_law}) for $\tilde G$, we obtain that
\begin{equation*}
\vert \mathbb{E} G_{\mathbf v\mathbf v} - \Pi_{\mathbf v\mathbf v} | \prec \vert \mathbb{E} \tilde G_{\mathbf v\mathbf v} - \Pi_{\mathbf v\mathbf v}| + \frac{1}{N\eta} \prec \frac{1}{N\eta}.
\end{equation*}
%where we abbreviated $\mathcal G_1 = z^{-1}G$ and $\tilde{\mathcal G}_1 = z^{-1}\tilde G$. 
This then concludes the proof of Theorem \ref{thm_large} by polarization. 
\end{proof}

%\bibliographystyle{abbrv}
%\bibliography{covariance_bib}

\end{document}